\documentclass{amsart}

\usepackage{etex}
\usepackage[english]{babel}
\usepackage[cp1250]{inputenc}
\usepackage{amssymb,amsthm,amsfonts,amsmath}
\usepackage{mathrsfs}
\usepackage{verbatim}
\usepackage{url}
\usepackage{color}
\usepackage[usenames,dvipsnames]{xcolor}
\usepackage{amsopn}
\DeclareMathOperator{\operconv}{oper-conv}

\DeclareMathOperator{\proj}{proj}
\DeclareMathOperator{\Span}{span}
\DeclareMathOperator{\Lat}{Lat}
\DeclareMathOperator{\Int}{int}
\DeclareMathOperator{\ran}{Ran}
\DeclareMathOperator{\Ran}{Ran}

\newcommand{\Sym}{\mathbb{S}}

\newcommand{\RR}{\mathbb R}

\newcommand{\NN}{\mathbb N}

\newcommand{\CC}{\mathbb C}

\newcommand{\cG}{\mathscr G}
\newcommand{\cH}{\mathscr H}
\newcommand{\cK}{\mathscr K}

\newcommand{\cB}{\mathcal B}

\newcommand{\cL}{\mathcal L}

\newcommand{\cS}{\mathcal S}
\newcommand{\fK}{\mathcal K}

\newcommand{\cA}{\mathcal A}
\newcommand{\cX}{\mathcal X}

\newcommand{\benu}{\begin{enumerate}}
\newcommand{\eenu}{\end{enumerate}}
\newcommand{\bop}{\begin{opomba}}
\newcommand{\eop}{\end{opomba}}

\newcommand{\id}{\mathrm{Id}}

\newcommand{\diag}{\mathrm{diag}}

\newcommand{\beqn}{\begin{align*}}
\newcommand{\eeqn}{\end{align*}}

\newcommand{\bdefi}{\begin{definition}}
\newcommand{\edefi}{\end{definition}}
\newcommand{\bcor}{\begin{corollary}}
\newcommand{\ecor}{\end{corollary}}
\newcommand{\bthe}{\begin{theorem}}
\newcommand{\ethe}{\end{theorem}}
\newcommand{\bpro}{\begin{proposition}}
\newcommand{\epro}{\end{proposition}}
\newcommand{\blem}{\begin{lemma}}
\newcommand{\elem}{\end{lemma}}
\newcommand{\brem}{\begin{remark}}
\newcommand{\erem}{\end{remark}}
\newcommand{\bequ}{\begin{equation}}
\newcommand{\eequ}{\end{equation}}
\newcommand{\bprf}{\begin{proof}}
\newcommand{\eprf}{\end{proof}}

\newtheorem{theorem}{Theorem}[section]
\newtheorem{corollary}[theorem]{Corollary}
\newtheorem{lemma}[theorem]{Lemma}
\newtheorem{proposition}[theorem]{Proposition}

\theoremstyle{definition}
\newtheorem{definition}[theorem]{Definition}
\newtheorem{remark}[theorem]{Remark}
\newtheorem{example}[theorem]{Example}

\numberwithin{equation}{section}

\title[Operator Positivstellens\"atze on matrix convex sets]{Operator Positivstellens\"atze for noncommutative polynomials positive on matrix convex sets}

\author{Alja\v z Zalar}
\address{Alja\v z Zalar, Institute of Mathematics, Physics, and Mechanics, Jadranska 19, 1000
Ljubljana, Slovenia}
\email{aljaz.zalar@imfm.si}

\date{\today}
\keywords{
free convexity, linear matrix inequality (LMI), spectrahedron, spectrahedrop, completely positive, Positivstellensatz, polar dual,
Gleichstellensatz, quadratic module, free real algebraic geometry, noncommutative polynomial, free positivity}

\subjclass[2010]{ 14P10, 15A22, 46L07, 47A63 (Primary);
11E25, 13J30, 46L89, 46N10, 47L25 (Secondary)}

\begin{document}

\maketitle
\begin{abstract} 
	This article studies algebraic certificates of positivity for noncommutative (nc) operator-valued polynomials on matrix convex sets, such as the solution set $D_L$, called a free Hilbert spectrahedron, of the linear operator inequality (LOI)
$L(X)=A_0\otimes I+\sum_{j=1}^g A_{j}\otimes X_j\succeq 0,$ where $A_j$ are self-adjoint linear operators on a separable Hilbert space, $X_j$ matrices and $I$ is an identity matrix. If $A_j$ are matrices, then $L(X)\succeq 0$ is called a linear matrix inequality (LMI) and $D_L$ a free spectrahedron. For monic LMIs, i.e., $A_0=I$, and nc matrix-valued polynomials the certificates of positivity were established by Helton, Klep and McCullough in a series of articles
with the use of the theory of complete positivity from operator algebras and classical separation arguments from real algebraic geometry. Since the full strength of the theory of complete positivity is not restricted to finite dimensions, but works well also in the infinite-dimensional setting,
we use it to tackle our problems. 
First we extend the characterization of the inclusion $D_{L_1}\subseteq D_{L_2}$ from monic \emph{LMIs} to monic \emph{LOIs} $L_1$ and $L_2$.
As a corollary one immediately obtains the description of a polar dual of a free Hilbert spectrahedron $D_L$ and its projection, called a free Hilbert spectrahedrop. Further on, using this characterization in a separation argument, we obtain  a certificate for multivariate matrix-valued nc polynomials $F$ positive semidefinite on a free Hilbert spectrahedron defined by a monic LOI. Replacing the separation argument by an operator Fej\'er-Riesz theorem
enables us to extend this certificate, in the univariate case, to operator-valued polynomials $F$.
Finally, focusing on the algebraic description of the equality $D_{L_1}=D_{L_2}$, we remove the assumption of boundedness from the description in the LMIs case by an extended analysis. However, the description does not extend to LOIs case by counterexamples.
\end{abstract}

\section{Introduction}

In this section we state the main concepts and results of this paper. Subsection \ref{motivation} places the content of the paper in a general context. In Subsections \ref{free-sets}-\ref{nc-poly-and-oper-psatze} definitions intertwine with the main results. Subsection \ref{guide} is a guide to the organization of the rest of the paper.

Throughout the paper $\cH$, $\cH_1$, $\cH_2$, $\cK$, $\cG$ stand for separable real Hilbert spaces unless stated otherwise.

\subsection{Context}\label{motivation}

The name Positivstellensatz refers to an algebraic certificate for a given polynomial $p$ to have a positivity property on a given closed semialgebraic set. Finding a certificate for an operator-valued polynomial $p$ positive semidefinite on an arbitrary closed semialgebraic set is a hard problem. Even if $p$ is a matrix-valued polynomial, the optimal certificates are only known to exist for very special sets, namely matrix convex sets
defined as matrix solution sets of linear matrix inequalities (LMIs). The aim of this paper is to generalize characterizations of noncommutative (nc) matrix-valued polynomials which are positive semidefinite on  a LMI set to characterizations of nc 
operator-valued polynomials which are positive semidefinite on arbitrary matrix convex sets.
By \cite{EFF-WIN}, every closed matrix convex set is a matrix solution set of a linear operator inequality (LOI). 

Our problem belongs to the field of free real algebraic geometry (free RAG); see \cite{HKM5} and references therein. Free RAG has two branches - free positivity and free convexity. Both branches present exciting mathematical challenges, and lend themselves to many applications. 

Free positivity is an analog of classical real algebraic geometry \cite{BCR,LAS,LAU,MAR, PD, PUT, SCE}, a theory of polynomial inequalities embodied in Positivstellens\"atze. It makes contact
with noncommutative real algebraic geometry \cite{CIM, HEL, HKM1, HKM2, HM1, MC, SCH}.
Free Positivstellens\"atze have applications to quantum physics \cite{PNA10}, operator algebras \cite{KS1}, quantum statistical mechanics \cite{KS2, CKP}, the quantum moment problems and multiprover games \cite{DLTW}. 

Matrix convex sets and free convexity arise naturally in a number of contexts, including engineering systems theory, operator spaces, systems and algebras and is closely linked to unital completely positive maps \cite{ARV72, PAU, FAR, HMPV}. The simplest examples of matrix convex sets are matrix solution sets of LMIs. A large class of linear systems engineering problems transforms to LMIs \cite[\S 1.1]{HKM5}, which led to a major advance in those problems during the past two decades \cite{SIG}. Furthermore, LMIs underlie the theory of semidefinite programming, an important recent innovation in convex optimization \cite{NEM}.
As mentioned above every closed matrix convex set is a matrix solution sets of a LOI by \cite{EFF-WIN}.

\subsection{Free sets, matrix convex sets, linear pencils and LOI sets} \label{free-sets}
This work fits into the wider context of free analysis \cite{Voi04,Voi10,KVV+,MS11,Pop10,AM+,BB07,dOHMP09,HKM1,PNA10}, so we start by recalling some of the standard notions used throughout this article.

\subsubsection{Free sets - matrix level}
Fix a positive integer $g\in \NN$. We use $\Sym_{n}$ to denote real symmetric $n\times n$ matrices and
$\Sym^g$ for the sequence $(\Sym_n^g)_n$. A \textbf{subset} $\Gamma$ of $\Sym^g$ is a sequence 
$\Gamma=(\Gamma(n))_n$, where $\Gamma(n)\subseteq \Sym_n^g$ for each $n$. The subset $\Gamma$ is
\textbf{closed with respect to direct sums} if $A=(A_1,\ldots,A_g)\in\Gamma(n)$ and $B=(B_1,\ldots,B_g)\in \Gamma(m)$ implies
	$$A\oplus B=\left(\left[\begin{array}{cc} A_1 &  0 \\ 0 & B_2\end{array}\right],\ldots,
		\left[\begin{array}{cc} A_g &  0 \\ 0 & B_g\end{array}\right]
		\right)\in \Gamma(n+m).$$
It is closed with respect to \textbf{(simultaneous) unitary conjugation} if for each $n$, each $A\in \Gamma(n)$ and each
$n\times n$ unitary matrix $U$,
	$$U^\ast A U=(U^{\ast}A_1U,\ldots, U^{\ast}A_gU)\in \Gamma(n).$$
The set $\Gamma$ is a \textbf{free set} if it is closed with respect to direct sums and simultaneous unitary conjugation.
If in addition it is closed with respect to \textbf{(simultaneous) isometric conjugation}, i.e., if for each $m\leq n$, each
$A=(A_1,\ldots,A_g)\in\Gamma(n)$, and each isometry $V: \RR^m\to\RR^n$, 
	$$V^\ast A V=(V^{\ast}A_1V,\ldots, V^{\ast}A_gV)\in \Gamma(m),$$
then $\Gamma$ is \textbf{matrix convex} \cite{HKM4}.

\subsubsection{Free sets - operator level}
Fix a separable Hilbert space $\cK$.
Let $\Lat(\cK)$ denote the \textbf{lattice of closed subspaces} of $\cK$. For a $K\in \Lat(\cK)$, we use $\mathbb{S}_{K}$ to denote the set of all self-adjoint operators on $K$. 
Let $\mathbb{S}_{\cK}$ stand for the set $(\mathbb{S}_{K})_K$.
A collection $\Gamma=(\Gamma(K))_K$ where $\Gamma(K)\subseteq \Sym_K^g$ for each $K$ a closed subspace of $\cK$, is a 
\textbf{free operator set} \cite{HKM4} if it is closed under direct sums and with respect simultaneous conjugation by unitary operators. If in addition it is closed with respect to simultaneous conjugation by isometries $V:H\to K$, where $H,K\in \Lat(\cK)$, then $\Gamma$ is \textbf{operator convex}.

\subsubsection{Linear pencils and LOI sets} \label{linear-pencils}

Let $\cH$ be separable real Hilbert space and $I_{\cH}$ the identity operator on $\cH$. For self-adjoint operators $A_0,A_1,\ldots,A_g\in \Sym_{\cH}$, the expression
	$$L(x)=A_0+\sum_{j=1}^g A_j x_j$$
is a \textbf{linear (operator) pencil}. If $\cH$ is finite-dimensional, then $L(x)$ is a \textbf{linear matrix pencil}. 
If $A_0=I_{\cH}$, then $L$ is \textbf{monic}.
If $A_0=0$, then $L$ is \textbf{homogeneous}.
To every tuple $A=(A_1,\ldots,A_g)\in \Sym_\cH^g$ we associate a homogeneous linear pencil $\Lambda_A$ and a monic linear pencil $L_A$ by
	$$\Lambda_A(x):=\sum_{j=1}^g A_j x_j,\quad L_A(x):= I_{\cH}+\Lambda_A(x).$$
The \textbf{operator Hilbert convex hull} $\operconv_\cK\{A\}$ of $A$ is the set
	$$\operconv_\cK\{A\}:=\bigcup_{(\cG, \pi, V)\in \Pi} V^\ast \pi(A) V=
		\bigcup_{(\cG, \pi, V)\in \Pi} (V^\ast \pi(A_1) V,\ldots,V^\ast \pi(A_g)V),$$
where $\Pi$ is the set of all triples $(\cG, \pi,V)$ of a separable real Hilbert space $\cG$, a contraction $V:\cK\to \cG$ and
 a unital $\ast$-homomorphism $\pi:B(\cH)\to B(\cG)$.

Given a tuple of self-adjoint operators $X=(X_1,\ldots,X_g)\in \Sym_K^g$ on a closed subspace $K$ of a Hilbert space $\cK$, the \textbf{evaluation}
$L(X)$ is defined as
	$$L(X)=A_0\otimes I_K+\sum_{j=1}^g A_j\otimes X_j,$$
where $I_K$ stands for an identity operator on $K$.

We call the set 
	$$D_L(1)=\{ x\in \RR^g\colon L(x)\succeq 0\}$$ 
a \textbf{Hilbert spectrahedron} or a \textbf{LOI domain}, the set 
	$$D_L = (D_L(n))_n\quad \text{where}\quad D_L(n) = \{ X\in \mathbb{S}_n^g\colon L(X)\succeq 0\},$$
a \textbf{free Hilbert spectrahedron} or a \textbf{free LOI set}, the set
	$$\partial D_L = (\partial D_L(n))_n\quad \text{where}\quad \partial D_L(n) = \{ X\in \mathbb{S}_n^g\colon L(X)\succeq 0, L(X)\not\succ 0\}$$
the  \textbf{boundary of a free Hilbert spectrahedron} and the set
	$$D_L^{\cK}=(D_L(K))_{K\in \Lat(\cK)}\quad \text{where}\quad
		D_L(K) = \{ X\in \mathbb{S}_K^g\colon L(X)\succeq 0\}.$$ 
an \textbf{operator free Hilbert spectrahedron} or an \textbf{operator free LOI set}, where $\cK$ is a separable real Hilbert space.
Note that $D_L(1)\subseteq \RR^g$ is a closed convex set and by the classical Hahn-Banach theorem every convex closed subset of $\RR^g$ is of this form. If $L$ is a linear matrix pencil, then we omit the word Hilbert from the definitions. 

\subsection{Inclusion of free Hilbert spectrahedra}

Our first main result is an algebraic characterization of the inclusion $D_{L_1}\subseteq D_{L_2}$ where $L_1$ and $L_2$ are monic linear operator pencils.

\begin{theorem}[Operator linear Positivstellensatz] \label{OperLP-intro}
	Let $L_j$, $j=1,2$, be monic linear operator pencils with coefficients from $B(\cH_j)$, $j=1,2$. Then $D_{L_1}\subseteq D_{L_2}$ if and only if there exist a separable real Hilbert space $\cK$,
	a contraction $V:\cH_2\to \cK$, a positive semidefinite operator $S\in B(\cH_2)$ and 
	a $\ast$-homomorphism $\pi:B(\cH_1)\to B(\cK)$ such that
		$$L_2=S+V^\ast \pi(L_1)V.$$

	Moreover, if $D_{L_1}(1)$ is bounded, then $V$ can be chosen to be isometric and $\pi$ a unital $\ast$-homomorphism,
	i.e.,
				$$L_2=V^\ast \pi(L_1)V.$$
\end{theorem}
	
For the proof see Corollary \ref{OperLP}. The main techniques used are complete positivity and the theory of operator algebras.
Namely, we define a unital $\ast$-linear map $\tau$ between the
linear spans of the coefficients of both pencils, connect
$D_{L_1}\subseteq D_{L_2}$ with $\tau$'s complete positivity,
invoke the Arveson extension theorem to extend it to a completely positive map on $B(\cH_1)$ and finally use the Stinespring representation theorem to obtain the result.

We demonstrate by Examples \ref{monicity-needed} and \ref{empty-operator-spectrahedron}, that the assumption of monicity of $L_j$, $j=1,2$, is in general needed in Theorem \ref{OperLP-intro}.

Inclusion of free spectrahedra for matrix pencils 
was considered in \cite{HKM1} and \cite{HKM2}. Our approach is the same
as the one from \cite{HKM1}, where the problem was
solved in the finite-dimensional case for a bounded set $D_{L_1}(1)$ 
(see \cite[Corollary 3.7]{HKM1}). We were able to modify it to work independently of the finite-dimensionality and the
boundedness of $D_{L_1}(1)$. Namely, Theorem \ref{OperLP-intro} extends \cite[Corollary 3.7]{HKM1} from matrix to operator pencils 
$L_1$, $L_2$ and removes the assumption of boundedness of the set $D_{L_1}(1)$.
\cite[Corollary 4.1]{HKM2} solves the problem 
in the finite-dimensional case also for an unbounded set
$D_{L_1}(1)$ but uses completely different techniques, including a Putinar-type separation argument, which does not seem to extend to the infinite-dimensional case.

\subsection{Equality of free spectrahedra} \label{equality-of-free-sp}
Our second main result is a characterization of different linear pencils which give the same free spectrahedron, see Theorem \ref{Unbound-GS-Intro} below. 
Before stating the  result we introduce some definitions. Let $A_0,A_1,\ldots,A_g\in \Sym_\cH$ be self-adjoint operators and
	$L(x)=A_0+\sum_{j=1}^g A_j x_j$
a linear pencil.
Let  $H\subseteq \cH$ be a closed subspace of $\cH$ which is \textbf{invariant} under each $A_j$, i.e., $A_jH\subseteq H$. Since each $A_j$ is self-adjoint, it also follows that $A_j H^\perp\subseteq H^\perp$, i.e.,  $\cH$ is automatically reducing for each $A_j$.
Hence, with respect to the decomposition $\cH=H\oplus H^\perp$,
$L$ can be written as the direct sum,
	$$L=\tilde L\oplus \tilde L^\perp=
			\left[\begin{array}{cc} \tilde L & 0 \\ 0 & \tilde L^\perp
					\end{array}\right], \quad \text{where}\quad
					\tilde L=I_H+\sum_{j=1}^g \tilde A_j x_j,$$
and $\tilde A_j$ is the restriction of $A_j$ to $H$.
We say that $\tilde L$ is a \textbf{subpencil} of $L$.
If $H$ is a proper closed subspace of $\cH$, then $\tilde L$ is a \textbf{proper subpencil} of $L$.
If $D_{L}=D_{\tilde L}$, then $\tilde L$ is a \textbf{whole subpencil} of $L$. If $L$ has no proper whole subpencil,
then $L$ is \textbf{$\sigma$-minimal}. 

\begin{theorem}[Linear Gleichstellensatz] \label{Unbound-GS-Intro}
	Let $L_k=I_{d_k}+\sum_{j=1}^g A_{k,j} x_j$, $k=1,2$, 
$d_j\in \NN$, $A_{k,j}\in \Sym_{d_k}$, be monic linear matrix pencils. 
Then $D_{L_1}=D_{L_2}$ if and only if every $\sigma$-minimal whole subpencil
$\tilde L_1$ of $L_1$ is unitarily equivalent to any $\sigma$-minimal whole subpencil $\tilde L_2$ of $L_2$, i.e., there is a unitary matrix $U$ such that 
	$$\tilde L_2=U^{\ast} \tilde L_1 U.$$
\end{theorem}

First we remark that even though $D_{L_1}=D_{L_2}$ if and only if $D_{L_1}\subseteq D_{L_2}$ and $D_{L_2}\subseteq D_{L_1}$, 
it is not clear how to prove Theorem \ref{Unbound-GS-Intro} only by using the matrix version of Theorem \ref{OperLP-intro}. 
The proof is more involved (see Theorem \ref{UnbounLG}). 
Our approach uses the idea from \cite{HKM1}, where it is shown how $D_L$ is governed by the multiplicative structure $C^{\ast}(\cS)$, i.e., the $C^{\ast}$-algebra generated by the set $\cS$ of the coefficients of $L$. Using this and the theory of  real $C^\ast$-algebras, 
Theorem \ref{Unbound-GS-Intro} is proved 
under the assumption $D_{L_1}(1)=D_{L_2}(1)$ is bounded in \cite[\S 3.3]{HKM1}. Analyzing the proof one can notice that it works  for $\sigma$-minimal pencils $\tilde L_1$, $\tilde L_2$ that satisfy
the implication 
	$$D_{\tilde L_1}=D_{\tilde L_2} \Rightarrow D_{^h \! \tilde L_1}=D_{^h\! \tilde L_2},$$
where 
	$$^\text{h} \! \tilde L_j(x_0,\ldots,x_g)=x_0\tilde L_j(x_0^{-1} x_1,\ldots,x_0^{-1} x_g), \quad j=1,2,$$ are the \textbf{homogenizations} of $\tilde L_j$, $j=1,2$. Note that the evaluation of a homogeneous linear pencil $L(x)=\sum_{j=0}^g A_jx_j$ on a tuple of symmetric matrices $X=(X_0,X_1,\ldots,X_g)\in \Sym_n^{g+1}$ is defined as
	$$L(X)=\sum_{j=0}^g A_j\otimes X_j$$
and 
$$D_{L} = (D_{L}(n))_n\quad 
		\text{where}\quad D_{L}(n) = \{ X\in \mathbb{S}_n^{g+1}\colon L(X)\succeq 0\}$$
is its free Hilbert spectrahedron.
By this observation and a lengthy case analysis we establish Theorem \ref{Unbound-GS-Intro} irrespective of the boundedness of the set $D_{L_1}(1)=D_{L_2}(1)$ in Section \ref{LOI-equality-problem-soln}.

	However, Theorem \ref{Unbound-GS-Intro} does not extend to linear operator pencils.
Example \ref{without-reduct-subpencil} shows, that not every operator pencil has a whole subpencil which is $\sigma$-minimal,
while Example \ref{not-unitarily-equivalent} gives two $\sigma$-minimal operator pencils which have the same free
Hilbert spectrahedron but are not unitarily equivalent.

\subsection{Free Hilbert spectrahedrops and polar duals}

\subsubsection{Free Hilbert spectrahedrops}\label{free-spectrahedrops}

Let $\cH$, $\cK$ be separable real Hilbert spaces. Let $D$, $\Omega_j$, $\Lambda_k\in \Sym_\cH$ be self-adjoint operators and 
	$$L(x,y)=D+\sum_{j=1}^g \Omega_j x_j + \sum_{k=1}^h \Gamma_k y_k\in \Sym_\cH\!\left\langle x,y\right\rangle$$
a linear pencil in the variables $(x_1,\ldots,x_g; y_1,\ldots,y_h)$.
We call the set
	$$\proj_x D_{L}(1):= \{ x\in \RR^g\colon \exists\; y\in \RR^h \text{ such that } L(x,y)\succeq 0\}$$
a \textbf{Hilbert spectrahedral shadow} \cite{BLEK-PAR-TOM}, the set
	$$\proj_x D_{L}=(\proj_x D_{L}(n))_n,$$
where
	$$\proj_x D_{L}(n):= \{ X\in \Sym_n^g\colon \exists\; Y\in \Sym_n^h \text{ such that } 
			L(X,Y)\succeq 0\},$$
a \textbf{free Hilbert spectrahedrop}, 
and the set 
	$$\proj_x D_{L}^{\cK}=(\proj_x D_{L}(K))_{K\in \Lat(\cK)},$$
where
	$$\proj_x D_{L}(K)= \{ X\in \Sym_K\colon \exists\; Y\in \Sym_K \text{ such that } L(X,Y)\succeq 0\},$$
an \textbf{operator free Hilbert spectrahedrop}.
If $L$ is a linear matrix pencil, then we omit the word Hilbert from the definitions.

\subsubsection{Polar duals}

Let $\cK$ be a real separable Hilbert space.
The \textbf{free polar dual} (resp.\ the \textbf{free Hilbert polar dual})  $\fK^{\circ}=(\fK^\circ(n))_n$ of a free set $\fK\subseteq \Sym^g$ (resp.\ a free operator set $\fK\subseteq \Sym^g_{\cK}$) is 
		$$\fK^{\circ}(n)=\{A\in \Sym_n^g\colon L_A(X)=I_n\otimes I+\sum_{j=1}^g A_j\otimes X_j\succeq 0
			\text{ for all } X\in \fK\}.$$

The \textbf{operator free polar dual} (resp.\ the \textbf{operator free Hilbert polar dual})  $\fK^{\cK,\circ}=(\fK^{\circ}(K))_{K\in \Lat(\cK)}$ of a free set $\fK\subseteq \Sym^g$ (resp.\ a free operator set $\fK\subseteq \Sym^g_{\cK}$) is 
		$$\fK^{\circ}(K)=\{A\in \Sym_K^g\colon L_A(X)=I_K\otimes I+\sum_{j=1}^g A_j\otimes X_j\succeq 0
			\text{ for all } X\in \fK\}.$$

\subsubsection{Polar duals of free Hilbert spectrahedra and free Hilbert spectrahedrops}\label{PolarDuals}

In this subsection we state our main results on polar duals of free spectrahedra and free spectrahedrops.
Let $\cH$, $\cK$ be separable real Hilbert spaces.

\begin{theorem} \label{polar-spectrahedron}
	Suppose $L:=I_{\cH}+\sum_{j=1}^g A_{j}x_j$ is 
	a monic linear operator pencil. The operator free Hilbert polar dual $(D_{L}^\cK)^{\cK,\circ}$ of the set
	$D_{L}^\cK$ is the set
				$$\operconv_\cK\{(A_1,\ldots,A_g)\}.$$
\end{theorem}

\begin{theorem}\label{polar-spectrahedrop}
	Suppose $L:=I_{\cH}+\sum_{j=1}^g \Omega_{j}x_j+\sum_{k=1}^{h} \Gamma_k y_k$
	is a monic linear operator pencil, where $\Omega_j, \Gamma_k\in \Sym_\cH$.
		The operator free Hilbert polar dual of $\fK:=\proj_x D_L^{\cK}$ is the 
	set 
	\begin{eqnarray*}
		\fK^{\cK,\circ} &=& \{A\in \mathbb{S}_{\cK}^{g}\colon (A,0)\in (D_{L}^\cK)^{\cK,\circ}\}\\
		&=& \left\{A\in \mathbb{S}_{\cK}^{g}\colon 
			\exists\;\text{a separable Hilbert space } \cG, \text{ an isometry }V:\cK\to\cG,\right. \\
		&&	\text{ and } \ast\!\text{-homomorphism }\pi:B(\cH\oplus \RR)\to B(\cG) \text{ s.t. }\\
		&& \left.A=V^{\ast}\pi(\left[\begin{array}{cc} \Omega & 0\\ 0& 	
			0\end{array}\right])V,\quad
			0=V^{\ast}\pi(\left[\begin{array}{cc} \Gamma & 0\\ 0& 	
			0\end{array}\right])V\right\},
	\end{eqnarray*}
\end{theorem}

	For the proof of Theorems \ref{polar-spectrahedron} and \ref{polar-spectrahedrop} see Theorems 
\ref{OperLP-satz-unbounded} and \ref{OperLP-trunc}, respectively. They follow from Theorem \ref{OperLP-intro}.
Theorem  \ref{polar-spectrahedron} states that an operator free convex set is generated by a finite set if and only if it is the polar dual of an operator free Hilbert spectrahedron, while Theorem \ref{polar-spectrahedrop} states a similar result for corresponding projections. The polar dual problems for finite-dimensional Hilbert spaces $\cH$ and $\cK$ were considered in \cite[\S 4.3, \S 4.4]{HKM3}.

\subsection{Nc polynomials and operator Positivstellens\"atze} \label{nc-poly-and-oper-psatze}

\subsubsection{Words and nc polynomials} \label{words-and-nc-polynomials}
We write $\left\langle x\right\rangle$ for the monoid freely generated by $x=(x_1,\ldots,x_g)$, i.e., 
$\left\langle x\right\rangle$ consists of \textbf{words} in the $g$ noncommuting letters $x_1,\ldots,x_g$. 
Let $\RR\!\left\langle x\right\rangle$ denote the associative $\RR$-algebra freely generated by $x$, i.e., the elements of 
$\RR\!\left\langle x\right\rangle$ are polynomials in the noncommuting variables $x$ with coefficients in $\RR$. The elements are
called $\textbf{noncommutative (nc) polynomials}$. Endow $\RR\!\left\langle x\right\rangle$ with the natural \textbf{involution $^\ast$} which fixes $\RR\cup \{x\}$ pointwise, reverses the order of words, and acts linearly on polynomials. Polynomials invariant under this involution are $\textbf{symmetric}$. The length of the longest word in a noncommutative polynomial 
$f\in \RR\!\left\langle x\right\rangle$ is denoted by $\deg(f)$. The set of all words of degree at most $k$ is 
$\left\langle x\right\rangle_k$ and $\RR\!\left\langle x\right\rangle_k$ is the vector space of all noncommutative polynomials of degree at most $k$. 

Fix separable Hilbert spaces $\cH_1$, $\cH_2$. \textbf{Operator-valued nc polynomials} are the elements of 
$B(\cH_1,\cH_2)\otimes \RR\!\left\langle x\right\rangle$. We write 
	$$P=\sum_{w\in \left\langle x\right\rangle}A_w \otimes w\in B(\cH_1,\cH_2)\otimes \RR\!\left\langle x\right\rangle$$
for an element $P\in B(\cH_1,\cH_2)\otimes \RR\!\left\langle x\right\rangle$, where the sum is finite. The involution $^\ast$ extends
to operator-valued polynomials by
	$$P^\ast=\sum_{w\in \left\langle x\right\rangle}A_w^\ast \otimes w^\ast\in
		 B(\cH_2,\cH_1)\otimes \RR\!\left\langle x\right\rangle.$$
If $\cH_1=\cH_2$ and $P=P^\ast$, then we say $P$ is \textbf{symmetric}.

\subsubsection{Polynomial evaluations}
If $P\in B(\cH_1,\cH_2)\otimes \RR\!\left\langle x\right\rangle$ is a nc operator-valued polynomial and $X\in B(\cK)^g$, where
$\cK$ is a separable Hilbert space, then 
	$$P(X)\in B(\cH_1,\cH_2)\otimes B(\cK)$$ 
is defined in the natural way by replacing $x_i$ by $X_i$ and sending the empty word to the identity operator on $\cK$. 
Note that if $P\in \RR^{\ell_1\times \ell_2}\!\left\langle x\right\rangle$ is a matrix-valued polynomial, where $\ell_1,\ell_2\in\NN$ are natural numbers, then $P(X):\cK^{\ell_2}\to \cK^{\ell_1}$ is an operator mapping form $\cK^{\ell_2}$  to $\cK^{\ell_1}$
and has a matrix representation $(p_{ij}(X))_{ij}$, where $P=(p_{ij}(x))_{ij}.$

\subsubsection{Free Hilbert semialgebraic sets}\label{free-semi-sets}

A symmetric operator-valued nc polynomial $P$ determines the \textbf{free Hilbert semialgebraic set} by
	$$D_P = (D_P(n))_n\quad \text{where}\quad D_P(n) = \{ X\in \mathbb{S}_n^g\colon P(X)\succeq 0\},$$
and
	the \textbf{operator free Hilbert semialgebraic set } by
	$$D^{\cK}_P = (D_P(K))_{K\in \Lat(\cK)}\quad \text{where}\quad
		D_P(K) = \{ X\in \mathbb{S}_K^g\colon P(X)\succeq 0\}.$$
Clearly, the sets $D_P$ and $D_P^\cK$ are a free set and a free operator set, respectively.
If $P$ is a symmetric matrix-valued nc polynomial, then we omit the word Hilbert in the definitions of $D_P$ and $D_P^\cK$.

\subsubsection{Operator Positivstellens\"atze }\label{NSatz}

Now we turn our attention to nc polynomials positive semidefinite on free Hilbert spectrahedra.

\begin{theorem}[Operator convex Positivstellensatz] \label{Oper-to-mat-intro}
		Let $L\in \Sym_{\cH}\!\left\langle x \right\rangle$ be a monic linear operator pencil. 
	Then for every symmetric matrix-valued noncommutative polynomial
	$F\in \RR^{\nu\times \nu}\!\left\langle x \right\rangle$ with $F|_{D_L}\succeq 0$,
	there is a separable real Hilbert space $\cK$, a $\ast$-homomorphism $\pi:B(\cH)\to B(\cK)$,
	finitely many matrix polynomials $R_j\in \RR^{\nu\times \nu}\!\left\langle x \right\rangle$ 
	and
	operator polynomials $Q_k\in B(\RR^{\nu},\cK)\otimes \RR\!\left\langle x \right\rangle$ 
	all of degree at most $\frac{\deg(F)+2}{2}$
	such that 
		$$F=\sum_{j} R_j^\ast R_j + \sum_k Q_k^\ast \pi(L) Q_k.$$
\end{theorem}

For the proof see Theorem \ref{CPsatz-matrix}. The problem for finite-dimensional Hilbert spaces $\cH$ and $\cK$ was considered 
in \cite{HKM2,HKM1}. In \cite{HKM2}, Theorem \ref{Oper-to-mat-intro} was obtained for linear matrix pencils $L$ by modifying a Putinar-type argument. In our approach we essentially use Theorem \ref{OperLP-intro} and a version of the Hahn-Banach theorem \cite[Theorem 2.2]{HKM3} to apply
the separation argument from \cite{HKM2} and extend the result to operator pencils $L$. 

Theorem \ref{Oper-to-mat-intro} extends to matrix-valued nc polynomials positive semidefinite on a free Hilbert spectrahedrop. 

In the univariate case we are able to extend Theorem \ref{Oper-to-mat-intro} to operator-valued nc polynomials $F$
by reducing the problem to the inclusion of free Hilbert spectrahedra. For the reduction we use variants of the operator Fej\'er-Riesz theorem \cite{ROS}.

\begin{theorem} \label{univariate-intro}
		Suppose $L=I_{\cH}+A_1 y\in \Sym_{\cH}\!\left\langle y \right\rangle$ 
	is a univariate monic linear operator pencil. 
	Then for every symmetric operator-valued noncommutative polynomial
	$F\in B(\cK)\otimes \RR\!\left\langle y \right\rangle$ with $F|_{D_L}\succeq 0$,
	there exists a separable real Hilbert space $\cG$, a $\ast$-homomorphism $\pi:B(\cH)\to B(\cG)$
	and
	finitely many operator polynomials $R_j\in B(\cK)\otimes \RR\!\left\langle x \right\rangle$ 
	and $Q_k\in B(\cK,\cG)\otimes \RR\!\left\langle x \right\rangle$ 
	all of degree at most $\frac{\deg(F)+2}{2}$
	such that 
		$$ F=\sum_{j} R_j^\ast R_j + \sum_k Q_k^\ast \pi(L) Q_k.$$
\end{theorem}

By Examples \ref{monicity-needed} and \ref{empty-operator-spectrahedron} below the assumption of monicity of $L$ is in general needed in Theorem \ref{univariate-intro}. 
It remains an open question if Theorem \ref{Oper-to-mat-intro} extends to operator-valued nc polynomials $F$.

\subsection{Reader's guide}\label{guide}

The paper is organized as follows. In Section \ref{LOI-domination-section} we study the inclusion of free Hilbert spectrahedra (see 
Subsection \ref{LOI-domination-problem-soln}) and polar duals of free Hilbert sprectrahedra and free Hilbert spectrahedrops (see Subsection \ref{polar-duals-soln}). The main results are proved with the use of completely positive maps and operator algebras.
In Section \ref{LOI-equality-problem-soln} we consider equality of free spectrahedra. In Subsection \ref{LMI-equality-soln}
we extend the characterization of matrix pencils with the same free spectrahedron from bounded spectrahedra to unbounded ones, while in Subsections \ref{non-existence-of-reduct} and \ref{counterexample for the LG} we show that the characterization does not generalize to operator pencils. Section \ref{positivstellensatz-multivariable}
studies the existence of a Positivstellensatz for multivariate nc operator-valued polynomials positive semidefinite on a free Hilbert  spectrahedron. The main result, Theorem \ref{CPsatz-matrix} is the solution for matrix-valued polynomials. This result is then extended to projections of free Hilbert spectrahedra in Theorem \ref{CPsatz-matrix-drops}. Finally, Section \ref{positivstellensatz-univariate} focuses on a Positivstellensatz for univariate operator-valued polynomials and presents the proof of Theorem \ref{univariate-intro}.

Recently Davidson, Dor-On, Moshe Shalit and Solel posted an arxiv preprint \cite{DDSS} which also considers inclusion of free Hilbert spectrahedra but in a complex setting.  Since they mostly focus on bounded spectrahedra, the overlap with our results (in Sections 
\ref{LOI-domination-section} and \ref{LOI-equality-problem-soln}) is minimal.

\section{Inclusion of free Hilbert spectrahedra and polar duals of free Hilbert spectrahedra and free Hilbert spectrahedrops} \label{LOI-domination-section}

In this section we characterize the inclusion  of free Hilbert spectrahedra (see Corollary \ref{OperLP}) and describe operator free Hilbert polar duals of a free Hilbert spectrahedron (see Theorem \ref{OperLP-satz-unbounded}) and a free Hilbert spectrahedrop (see Theorem \ref{OperLP-trunc}).
The main techniques used are complete positivity and the theory of operator algebras.
We define the unital $\ast$-linear map $\tau$ between the linear spans of the coefficients of the given linear pencils. There are two crucial observations. The first is the connection between the inclusion $D_{L_1}\subseteq D_{L_2}$ and the complete positivity of $\tau$  given by Theorem \ref{n-pos-of-tau}, while the second is an an algebraic trick of extending the pencil to the direct sum with the monic scalar pencil $1$, which makes the extended map $\tilde \tau$ completely positive if and only if $D_{L_1}\subseteq D_{L_2}$. The proof of Theorem \ref{OperLP-intro} then follows by invoking the real version of Arveson extension theorem and finally using the Stinespring representation theorem. Theorems \ref{polar-spectrahedron} 
and  \ref{polar-spectrahedrop} are consequences of Theorem \ref{OperLP-intro}. 

\subsection{Domination of free Hilbert spectrahedra}  \label{LOI-domination-problem-soln}
Let $\cH_1,\cH_2, \cK$ be separable real Hilbert spaces.
Given $L_1$ and $L_2$ monic linear operator pencils
	\begin{eqnarray*}
		L_1(x) &:=& I_{\cH_1}+\sum_{j=1}^g A_{j}x_j\in \mathbb{S}_{\cH_1}\!\left\langle x\right\rangle,\quad
		L_2(x) := I_{\cH_2}+\sum_{j=1}^g B_{j}x_j\in \Sym_{\cH_2}\!\left\langle x\right\rangle,
	\end{eqnarray*}
we are interested in the algebraic characterization of the inclusion of the free LOI sets
(resp.\ operator free LOI sets)
	$$D_{L_1}\subseteq D_{L_2} \quad (\text{resp.}\ D_{L_1}^\cK\subseteq D_{L_2}^{\cK}).$$

In this subsection we first prove the equivalence between both inclusions,
then introduce the unital $\ast$-linear maps $\tilde \tau$ and $\tau$ between the linear spans of the (extended) coefficients of both pencils, study the well-definedeness and complete positivity of both maps and finally prove the main result; see
Corollary \ref{OperLP}. We also show by an example that the monicity of pencils is necessary (Example \ref{monicity-needed}).

 \subsubsection{Equivalence of the inclusions $D_{L_1}\subseteq D_{L_2}$ and $D_{L_1}^\cK\subseteq D_{L_2}^{\cK}$}

\begin{proposition}	\label{mat-to-oper-spec-relation}
	We have the following equivalence:
		$$D_{L_1}^{\cK}\subseteq D_{L_2}^{\cK} \quad\Leftrightarrow\quad D_{L_1}\subseteq D_{L_2}.$$
\end{proposition}

To prove proposition we need  a lemma.

\begin{lemma} \label{matrix-domain-sufficient-for-linear}
		Let 
		$L(x)=A_0+\sum_j A_{j}x_j\in \Sym_\cH\!  \left\langle x\right\rangle$
	be a linear operator pencil and $X\in S_{\cK}^g$ be a tuple self-adjoint operators on a Hilbert space $\cK$. Then
	$X\in D_L^{\cK}$ if and only if $V^\ast X V\in D_{L}(m)$ for every $m\in\NN$ 
	and every isometry $V\in B(\RR^{m},\cK)$.
\end{lemma}

\begin{proof}
	Let $X\in D_L^{\cK}$. We have 
	\begin{eqnarray*}
		L(V^\ast X V)
			&=& A_0\otimes I_{\RR^m} + \sum_j A_j \otimes V^\ast X_j V \\
			&=& (I_\cH\otimes V)^\ast (A_0\otimes I_{\RR^m}+ \sum_j A_j \otimes X_j) (I_\cH\otimes V)\\
			&=& (I_\cH\otimes V)^\ast L(X) (I_\cH\otimes V)\succeq 0.
	\end{eqnarray*}
	Hence $V^\ast X V\in D_{L}(m)$.\\ 
	
	Let us now assume $V^\ast X V\in D_{L}(m)$ for every isometry $V\in B(\RR^{m},\cK)$, $m\in \NN$. 		
	Suppose $X\notin D_L^{\cK}$. 
	Then there is a vector 
	$v:=\sum_{k=1}^m h_k\otimes u_k\in \cH\otimes \cK$ such that 
		$\left\langle L(X)v,v\right\rangle<0.$
	Without loss of generality we may assume $u_1,\ldots,u_m$ are orthonormal.
	Hence
	\begin{eqnarray*}
		\left\langle L(X)v,v\right\rangle
			&=& \left\langle (\sum_j A_j \otimes X_j ) (\sum_k h_k\otimes u_k),\sum_k h_k\otimes u_k\right\rangle_{\cH\otimes \cK}\\
			&=& \sum_j\sum_k\sum_l \left\langle (A_j\otimes X_j) (h_k\otimes u_k), h_l \otimes u_l\right\rangle_{\cH\otimes \cK}\\
			&=& \sum_j\sum_k\sum_l \left\langle A_jh_k, h_l \right\rangle_{\cH}\left\langle X_ju_k, u_l \right\rangle_{\cK}<0.
	\end{eqnarray*}
	Let $e_k$ be the standard basis vectors for $\RR^m$. 
	Let us define a linear map $V:\RR^m\to \cH$ by $e_k\mapsto u_k$. Since $\{e_1,\ldots,e_m\}$ and $\{u_1,\ldots,u_m\}$ are orthonormal,
	$V$ is an isometry. Therefore, $L(V^\ast X V)\succeq 0$. We have 
	\begin{eqnarray*}
		0         &\leq&\left\langle L(V^\ast X V)\tilde v,\tilde v\right\rangle_{\cH\otimes \RR^m}\\
			&=& \left\langle (\sum_j A_j \otimes V^\ast X_j V) (\sum_k  h_k\otimes e_k),\sum_k h_k\otimes  e_k\right\rangle_{\cH\otimes \RR^m}\\
			&=& \sum_j\sum_k\sum_l \left\langle ( A_j  \otimes V^\ast X_j V ) (h_k\otimes e_k), h_l \otimes e_l\right\rangle_{\cH\otimes \RR^m}\\
			&=& \sum_j\sum_k\sum_l \left\langle  A_j  h_k,  h_l \right\rangle_{\cH}\left\langle V^\ast X_j V  e_k,  e_l \right\rangle_{\RR^m}\\
			&=& \sum_j\sum_k\sum_l \left\langle  A_j  h_k, h_l \right\rangle_{\cH}\left\langle X_j V e_k, V  e_l \right\rangle_{\cK}\\
			&=& \sum_j\sum_k\sum_l \left\langle A_jh_k, h_l \right\rangle_{\cH}\left\langle X_ju_k, u_l \right\rangle_{\cK}
			=\left\langle L(X)v,v\right\rangle.
	\end{eqnarray*}
	This is a contradiction. Hence $X\in D_L^{\cK}$.
\end{proof}

\begin{proof}[Proof of Proposition \ref{mat-to-oper-spec-relation}]
		The non-trivial direction is $D_{L_1}\subseteq D_{L_2}$ implies $D_{L_1}^{\cK}\subseteq D_{L_2}^{\cK}$. Let us take $X\in D_{L_1}^{\cK}$. 
	By Lemma \ref{matrix-domain-sufficient-for-linear}, 
		$$X\in D_{L_2}^{\cK} \Leftrightarrow V^\ast X V\in D_{L_2}(m)\text{ for every 
	isometry }	
	V\in B(\RR^{m},\cK), \; m\in\NN.$$
	By Lemma \ref{matrix-domain-sufficient-for-linear}, $X\in D_{L_1}^{\cK}$ implies
		$$V^\ast X V\in D_{L_1}(m)\text{ for every
 	isometry } 	
	V\in B(\RR^{m},\cK), \; m\in\NN.$$
	But $D_{L_1}\subseteq D_{L_2}$ implies
		$$V^\ast X V\in D_{L_2}(m)\text{ for every 
	isometry }	
	V\in B(\RR^{m},\cK), \; m\in\NN.$$
	This concludes the proof.
\end{proof}

\subsubsection{Connection with complete positivity} \label{connection-with-com-pos} 
The following operator systems will play an important role in the sequel: 
	\begin{eqnarray*}
		\mathcal{S}_1&:=&\Span\{I_{\cH_1},A_1,\ldots, A_g\},\\
		\widetilde{ \mathcal{S}_1}&:=&\Span\{I_{\cH_1}\oplus 1,A_1\oplus 0,\ldots, A_g\oplus 0\}\subseteq 
			B(\cH_1\oplus \RR),\\
		\mathcal{S}_2&:=&\Span\{I_{\cH_2},B_1,\ldots, B_g\}\subseteq B(\cH_2).
	\end{eqnarray*}
Recall from \S \ref{equality-of-free-sp} that the homogenization $^\text{h}\! L_j$ of $L_j$ is
	$$^{\text{h}\!} L_j(x_0,\ldots,x_g)=x_0 L_j(x_0^{-1}x_1,\ldots,x_0^{-1}x_g).$$
By Lemma \ref{well-defined-map} below,  the inclusion $D_{L_1}(1)\subseteq D_{L_2}(1)$ implies that the unital linear map 
	$$\tilde \tau:\widetilde{ \mathcal{S}_1} \to \mathcal{S}_2,\quad A_j\oplus 1\mapsto B_j,$$
is well-defined, while the stronger inclusion $D_{^\text{h}\! L_1}(1)\subseteq D_{^\text{h}\! L_2}(1)$ implies the well-definedness of the unital linear map 
	$$\tau:\mathcal{S}_1 \to \mathcal{S}_2,\quad A_j \mapsto B_j.$$
In particular, $\tau$ is well-defined if  $D_{L_1}(1)$ is bounded.
	
\begin{lemma}\label{well-defined-map}
	Assume the notation as above. 
	\begin{enumerate}
		\item If $D_{L_1}(1)\subseteq D_{L_2}(1)$, then the map $\tilde\tau$ is well-defined. 
		\item If $D_{^\text{h}\! L_1}(1)\subseteq D_{^\text{h}\! L_2}(1)$ or $D_{L_1}(1)$ is a bounded set, then the map $\tau$ is 
		well-defined.
	\end{enumerate}
\end{lemma}

\begin{proof}
	First we prove (1).
	It suffices to prove that 
		$$\mu_0 (I_{\cH_1}\oplus 1 ) + \sum_{i=1}^g \mu_i (A_i\oplus 0) =0 \quad \text{implies}\quad
		\mu_0 I_{\cH_2}+\sum_{i=1}^g \mu_i B_i =0,$$
	where $\mu_0,\ldots,\mu_g\in \RR$. First we notice that $\mu_0=0$.
	From $\sum_{i=1}^g \mu_i (A_i\oplus 0) =0$  it follows that 
		$\sum_{i=1}^g t\mu_i (A_i\oplus 0) =0 \text{ for every } t\in\RR.$
	Hence,  $(t\mu_1,\ldots,t\mu_g)\in D_{L_1}(1)\subseteq D_{L_2}(1)$ for every $t\in\RR$.
	Suppose to the contrary that $\sum_{i=1}^g \mu_i B_i \neq 0$. Since 
	 $\sum_{i=1}^g \mu_i B_i$ is self-adjoint, it follows that there is $h\in \cH_2$ such that
	 $\left\langle  (\sum_{i=1}^g \mu_i B_i )h,h\right\rangle\neq 0.$
	 But then 
	 $t(\mu_1,\ldots, \mu_g) \notin D_{L_2}$ for $t\to \infty$ or $t\to -\infty$, which is a contradiction. 	Hence $\sum_{i=1}^g \mu_i B_i =0$ and 
	the map $\tilde\tau$ is well-defined.

	For the proof of (2) let us first consider the inclusion $D_{^\text{h}\! L_1}(1)\subseteq D_{^\text{h}\! L_2}(1)$.
	We have to prove that 
		$$\mu_0 I_n + \sum_j \mu_j A_j=0\quad \text{implies} \quad \mu_0 I_m + \sum_j \mu_j B_j=0.$$
	Suppose to the contrary that $\mu_0 I_m + \sum_j \mu_j B_j\neq 0$. Since
	$\mu_0 I_m + \sum_j \mu_j B_j$ is self-adjoint, it follows that there is $h\in \cH_2$ such 
	that $\left\langle(\mu_0 I_m + \sum_j \mu_j B_j)h,h\right\rangle\neq 0$. 
 	Therefore
	$t(\mu_0, \mu_1,\ldots, \mu_g) \notin D_{^\text{h}\! L_2}$ for $t\to \infty$ or $t\to -\infty$. 
	But this is a contradiction with 
	$t(\mu_0,\mu_1,\ldots,\mu_g)\in D_{^\text{h}\! L_1}\subseteq D_{^\text{h}\! L_2}$.
	Hence $\tau$ is well-defined.

	Now we consider the case of a bounded set $D_{L_1}(1)$. 
	In this case the set $$\{I_{\cH_1},A_1,\ldots,A_g\}$$ is linearly independent; the proof is the same
	as in the matrix case (see \cite[Proposition 2.6]{HKM1}). Thus $\tau$ is well-defined.
\end{proof}

The following example shows that for unbounded sets $D_{L_1}(1)$, the assumption $D_{L_1}(1)\subseteq D_{L_2}(1)$ does not
suffice for the well-definedness of the map $\tau$.

\begin{example}
	Let $\ell_1=1+x$ and $\ell_2=1$ be monic linear scalar polynomials. 
	Note that $D_{\ell_1}(1)=[-1,\infty)\subset \RR=D_{\ell_2}(1)$
	but by the definition of the map $\tau:\RR\to \RR$ we have
	that $\tau(1)=1$ and $\tau(1)=0$, 
	which is a contradiction.
\end{example}

Now we define the $n$-positivity, $n\in \NN$, and the complete positivity of a map 
	$$\phi:\cS_1\to \cS_2$$ 
mapping between operator systems $\cS_j\subseteq B(\cH_j)$, $j=1,2,$ invariant under the transpose. 
For $n\in \NN$, $\phi$ induces the map 
	$$\phi_n=I_n \otimes \phi:\RR^{n\times n} \otimes \cS_1 =\cS_1^{n\times n} \to \cS_2^{n\times n}, \quad
		M\otimes A\mapsto M\otimes \phi(A),$$
called an \textbf{ampliation} of $\phi$. Equivalently,
	$$\phi\left(\left[\begin{array}{ccc} T_{11} & \cdots & T_{1n}\\ \vdots & \ddots & \vdots \\ T_{n1} & \cdots & T_{nn}  \end{array}\right] \right)
	=\left[\begin{array}{ccc} \phi(T_{11}) & \cdots & \phi(T_{1n})\\ \vdots & \ddots & \vdots \\ \phi(T_{n1}) & \cdots & \phi(T_{nn})  \end{array}\right].$$
We say that $\phi$ is \textbf{$n$-positive} if $\phi_n$ is a positive map. If $\phi$ is $n$-positive for every $n\in\NN$, then $\phi$ is \textbf{completely positive}.
If $\phi_n$ is an isometry for every $n\in \NN$, then $\phi$ is \textbf{completely isometric}.

	In the following theorem we prove that the $n$-positivity of $\tau$ is equivalent to the inclusion
$D_{^\text{h}\! L_1}(n)\subseteq D_{^\text{h}\! L_2}(n)$. If $D_{L_1}(1)$ is bounded this is equivalent to the
inclusion $D_{L_1}(n)\subseteq D_{L_2}(n)$. Since
$D_{^\text{h}\! \tilde L_1}(n)\subseteq D_{^\text{h}\! L_2}(n)$ if and only if $D_{L_1}(n)\subseteq D_{L_2}(n)$, 
$\tilde \tau$ is $n$-positive if and only if $D_{L_1}(n)\subseteq D_{L_2}(n)$. 

\begin{theorem}\label{n-pos-of-tau}
	Let 
	\begin{eqnarray*}
		L_1&=& I_{\cH_1}+\sum_{j=1}^g A_{j}x_j\in \Sym_{\cH_1}\!\left\langle x\right\rangle, \quad
		L_2= I_{\cH_2}+\sum_{j=1}^g B_{j}x_j\in \Sym_{\cH_2}\!\left\langle x\right\rangle
	\end{eqnarray*}
	be monic linear operator pencils. 
	Suppose $D_{^\text{h}\! L_1}(1)\subseteq D_{^\text{h}\! L_2}(1)$. Let 
	$\tau:\mathcal{S}_1\to \mathcal{S}_2$ be the unital linear map $A_{j}\mapsto B_{j}$.
	Then:
	\begin{enumerate}
		\item[(1)] $\tau$ is $n$-positive if and only if $D_{^\text{h}\! L_1}(n)\subseteq D_{^\text{h}\! L_2}(n)$.
		\item[(2)] $\tau$ is completely positive if and only if $D_{^\text{h}\! L_1}\subseteq D_{^\text{h}\! L_2}$.
		\item[(3)] If $\dim(\cH_2)=n$, then $\tau$ is completely positive if and only if $\tau$ is $n$-positive.
	\end{enumerate}
	In particular, if $D_{L_1}(1)$ is bounded, then 
		$$D_{L_1}(n)\subseteq D_{L_2}(n)\quad \Leftrightarrow \quad D_{^\text{h}\! L_1}(n)\subseteq D_{^\text{h}\! L_2}(n)$$ 
	and hence
	\begin{enumerate}
		\item[(4)] $\tau$ is $n$-positive if and only if $D_{L_1}(n)\subseteq D_{L_2}(n)$.
		\item[(5)] $\tau$ is completely positive if and only if $D_{L_1}\subseteq D_{L_2}$.
	\end{enumerate}
	Let $\tilde L_1$ be the monic linear pencil
		$$\tilde L_1= I_{\cH_1\oplus \RR}+\sum_{j=1}^g \tilde A_{j}x_j\in \Sym_{\cH_1\oplus \RR}\!\left\langle 
			x\right\rangle,$$
	where $\tilde A_j:=\left[\begin{array}{cc} A_j & 0 \\ 0 & 0\end{array}\right] \in \Sym_{\cH_1}$.	
	Then
		$$D_{L_1}(n)\subseteq D_{L_2}(n)\quad \Leftrightarrow 
			\quad D_{^\text{h}\! \tilde L_1}(n)\subseteq D_{^\text{h}\! L_2}(n).$$ 
	If $D_{L_1}(1)\subseteq D_{L_2}(1)$,
	let $\tilde \tau:\widetilde{\mathcal{S}_1}\to \mathcal{S}_2$ be the unital linear map $A_{j}\oplus 0\mapsto B_{1}.$
	Then:
	\begin{enumerate}
		\item[(7)] $\tilde\tau$ is $n$-positive if and only if $D_{L_1}(n)\subseteq D_{L_2}(n)$.
		\item[(8)] $\tilde\tau$ is completely positive if and only if $D_{L_1}\subseteq D_{L_2}$.
	\end{enumerate}
\end{theorem}

\begin{remark}
	Notice that $\tau$ and $\tilde \tau$ are well-defined; see Lemma \ref{well-defined-map}.
\end{remark}

To prove Theorem \ref{n-pos-of-tau} we need an additional technical lemma.

\begin{lemma}\label{shuffle}
	Let $L=I_\cH+\sum_{j=1}^g A_j x_j\in  \Sym_{\cH}\!\left\langle x\right\rangle$ be a monic linear pencil. Then
	for a tuple $X=(X_1,\ldots,X_g)\in \Sym_n$ we have
	$$L(X)=I_\cH\otimes I_n+\sum_{j=1}^g A_j \otimes X_j \succeq 0 \quad \Leftrightarrow \quad
		I_n\otimes I_\cH+\sum_{j=1}^g X_j \otimes A_j\succeq 0.
	$$
\end{lemma}

\begin{proof}
	The lemma follows by observing that after applying a permutation called the 
	\textit{canonical shuffle} \cite{PAU} to $L(X)$ we obtain 
	$I_n\otimes I_\cH+\sum_{j=1}^g X_j \otimes A_j$.
\end{proof}

\begin{proof} [Proof of Theorem \ref{n-pos-of-tau}]
	First we prove (1) and (2). Since (2) follows from (1), it suffices to prove (1).
	The nontrivial direction is $(\Leftarrow)$. 
	Suppose $\tilde T\in \widetilde{\mathcal{S}_1}^{n\times n}$ is positive semidefinite. 
	Then $\tilde T$ is of the form
		$Y\otimes I_{\cH_1}+ \sum_j X_j \otimes A_j$
	for some $Y,X_1,\ldots,X_g\in \RR^{n\times n}.$
	We have to prove that
		$\tilde \tau(\tilde T)=Y\otimes I_{\cH_2}+ \sum_j X_j\otimes B_j\succeq 0.$
	From $\tilde T=\tilde T^{\ast}$, it follows that 
		$$Y\otimes I_{\cH_1}+ \sum_j X_j \otimes A_j=
			\frac{1}{2}((Y+Y^{\ast})\otimes I_{\cH_1} + \sum_j (X_j+X_j^{\ast}) \otimes A_j).$$
	Thus we may assume that $Y,X_1,\ldots,X_g\in \Sym_{n}$.
	But since 
		$$(Y,X_1,\ldots,X_g)\in D_{^\text{h}\! L_1}(n)\subseteq D_{^\text{h}\! L_2}(n),$$ 
	it follows by Lemma \ref{shuffle} that $\tau(T)$ is $n$-positive.

	Second we prove that if $D_{L_1}(1)$ is bounded, then 
		$$D_{L_1}(n)\subseteq D_{L_2}(n)\quad
			\Leftrightarrow \quad D_{^\text{h}\! L_1}(n)\subseteq D_{^\text{h}\! L_2}(n).$$
	The non-trivial direction is $(\Rightarrow)$. 
	Let us take $X:=(X_0,X_1,\ldots, X_n)\in D_{^\text{h}\! L_1}(n)$. 
	We have to prove that  $X\in D_{^\text{h}\! L_2}(n)$. 
	If $X_0\succeq 0$, then this follows (possibly after approximation argument $X_0\mapsto X_0+\epsilon I_n$)
	from $D_{L_1}(n)\subseteq D_{L_2}(n)$. Let us suppose that $X_0\not\succeq 0$.
	Then there exists $v\in \RR^n$ with $\left\langle X_0 v,v\right\rangle<0$. Define $V:\RR\to \RR^n$ by
	$r\mapsto rv$. The map $V^\ast:\RR^n\to \RR$ is given by $u\mapsto \left\langle u,v\right\rangle$. We have
		\begin{eqnarray*}
			\left(I_{\cH_1}\otimes V \right)^\ast {^\text{h}\! L_1(X)} (I_{\cH_1}\otimes V)
			&=&	I_{\cH_1}\otimes V^\ast X_0 V +\sum_j A_{j}\otimes V^\ast X_j V \\
			&=&	I_{\cH_1}\otimes \left\langle X_0 v,v\right\rangle +
				\sum_j A_{j}\otimes  \left\langle X_j v,v\right\rangle \succeq 0.
		\end{eqnarray*}
	Since $ I_{\cH_1}\otimes \left\langle X_0 v,v\right\rangle  \prec 0,$ 
	it follows that 
	$\sum_j  A_{j}\otimes \left\langle X_j v,v\right\rangle \succ 0.$
	Thus $$(t\left\langle X_1 v,v\right\rangle,\ldots,t\left\langle X_g v,v\right\rangle)\in D_{L_1}(1)$$ for every $t>0$, 
	which	contradicts the boundedness of $D_{L_1}(1)$. 

	Finally, we prove that
		$$D_{L_1}(n)\subseteq D_{L_2}(n)\quad \Leftrightarrow \quad D_{^\text{h}\! \tilde L_1}(n)\subseteq D_{^\text{h}\! L_2}(n).$$
	The nontrivial direction is $(\Rightarrow)$.
	Let us take $(X_0,X_1,\ldots, X_n)\in D_{^\text{h}\! \tilde L_1}(n)$. 
	We have to prove that  $(X_0,X_1,\ldots, X_n)\in D_{^\text{h}\! \tilde L_2}(n)$. 
	We know that
		$$(I_{\cH_1}\oplus 1)\otimes X_0 + \sum_j (A_j\oplus 0) \otimes X_j\succeq 0.$$
	Clearly, this is equivalent to 
		$$
			I_{\cH_1}\otimes X_0 +\sum_j A_j \otimes X_j\succeq 0 \text{ and } X_0\succeq 0.
		$$
	We have to prove that
		$$I_{\cH_2}\otimes X_0+ \sum_j B_j \otimes X_j\succeq 0.$$
	By the approximation argument we can replace $X_0\succeq 0$ with $X_0+\epsilon I_{n}$ for $\epsilon>0$.
	So without loss of generality we may assume that $X_0\succ 0$. 
	Since $X_0^{-\frac{1}{2}}$ is invertible, it follows that 
	$\ran(I_j\otimes Y_0^{-\frac{1}{2}})={\cH_j\otimes \RR^{n}}$ for $j=1,2$.
	Therefore
	\begin{eqnarray*}
		&&	0\preceq I_{\cH_1}\otimes X_0+\sum_j A_j \otimes X_j\\
		&\Leftrightarrow&
			0\preceq (I_{\cH_1}\otimes X_0^{-\frac{1}{2}})^\ast 
			(I_{\cH_1}\otimes X_0+\sum_j A_j \otimes X_j)
			(I_{\cH_1}\otimes X_0^{-\frac{1}{2}})\\
		&=& I_{\cH_1} \otimes I_n+\sum_j A_j \otimes X_0^{-\frac{1}{2}}X_j X_0^{-\frac{1}{2}}.
	\end{eqnarray*}
	(and $0\preceq I_{\cH_2}\otimes X_0+\sum_j B_j \otimes X_j 
				\Leftrightarrow
			0\preceq I_{\cH_2}\otimes I_n+\sum_j B_j \otimes X_0^{-\frac{1}{2}}X_j X_0^{-\frac{1}{2}}$.). 
	By the assumption $D_{L_1}(n)\subseteq D_{L_2}(n)$,
		$I_{\cH_1}\otimes I_n+\sum_j A_j \otimes X_0^{-\frac{1}{2}}X_j X_0^{-\frac{1}{2}}\succeq 0$
	implies 
		$I_{\cH_2}\otimes I_n+\sum_j B_j \otimes X_0^{-\frac{1}{2}}X_j X_0^{-\frac{1}{2}}\succeq 0$, which concludes the proof.
\end{proof}

		If $L_1$ and $L_2$ are monic linear matrix pencils such that
 	$D_{L_1}(1)$ is unbounded and $D_{L_1}\subseteq D_{L_2}$, it is not necessary  that
	$D_{^\text{h}\! L_1}\subseteq D_{^\text{h}\! L_2}$ by Example	\ref{counterexample} below. 

\begin{example}\label{counterexample}
	For the following monic linear matrix pencils
		\begin{eqnarray*}
		L_1(x_1,x_2)&=&\left[\begin{array}{ccc} 1+2x_1+2x_2 & 0 & 0\\ 
			0 & 1+2x_1 & 0 \\
			0 & 0 & 1+2x_2 \end{array}\right],\\
		L_2(x_1,x_2)&=&\left[\begin{array}{ccc} 1+x_1+x_2 & 0 & 0\\ 
			0 & 1+x_1 & 0 \\
			0 & 0 & 1+x_2 \end{array}\right],
		\end{eqnarray*}
	we have
		\begin{eqnarray*}
			D_{L_1}(n)&=&\{(X_1,X_2)\in \Sym_n \colon X_1+X_2\succeq -\frac{1}{2}I_n, X_1\succeq -\frac{1}{2}I_n, X_2 \succeq -\frac{1}{2}I_n\},\\
			D_{L_2}(n)&=&\{(X_1,X_2)\in \Sym_n \colon X_1+X_2\succeq -I_n, X_1\succeq -I_n, X_2 \succeq -I_n\}
		\end{eqnarray*}
	for every $n\in \NN$. 
	Hence, $D_{L_1}(n)\subseteq D_{L_2}(n)$ for every $n\in \NN$, i.e., $D_{L_1}\subseteq D_{L_2}$.
	But 
		$$(-1,\frac{1}{2},\frac{1}{2})\in D_{^\text{h}\! L_1}(1)\setminus D_{^\text{h}\! L_2 }(1)$$ 
	and hence 
		$$D_{^\text{h}\! L_1}(1)\not\subseteq D_{^\text{h}\! L_2 }(1).$$
\end{example}

\subsubsection{Characterization of the inclusion $D_{L_1}\subseteq D_{L_2}$} We characterize the domination  $D_{L_1}\subseteq D_{L_2}$ by using the connection with complete positivity explained in \S \ref{connection-with-com-pos}.

\begin{corollary}[Operator linear Positivstellensatz] \label{OperLP}
		Let 
	$$
		L_1= I_{\cH_1}+\sum_{j=1}^g A_{j}x_j\in \Sym_{\cH_1}\!\left\langle x\right\rangle, \quad
		L_2= I_{\cH_2}+\sum_{j=1}^g B_{j}x_j\in \Sym_{\cH_2}\!\left\langle x\right\rangle
	$$
	be monic linear operator pencils. If $D_{L_1}\subseteq D_{L_2}$ then:
	\begin{enumerate}
		\item There exist a separable real Hilbert space $\cK$, an isometry $V:\cH_2\to \cK$ and a unital 
			$\ast$-homomorphism $\pi:B(\cH_1\oplus \RR)\to B(\cK)$ such that
				$$L_2(x)=V^\ast \pi(\left[\begin{array}{cc} 0 & 0 \\ 0 & 1\end{array}\right])V
					+V^\ast \pi(\left[\begin{array}{cc} L_1 & 0 \\ 0 & 0\end{array}\right])V.$$
		\item There  exist a separable real Hilbert space $\cK_0$, a contraction $V_0:\cH_2\to \cK_0$, 
			a unital $\ast$-homomorphism $\pi_0:B(\cH_1)\to B(\cK_0)$ 
			and a positive semidefinite operator 
			$S\in B(\cH_2)$ such that
				$$L_2(x)=S+V_0^\ast \pi_0(L_1(x))V_0.$$
		\item If $D_{L_1}(1)$ is bounded, then $V_0$ in (2) can be chosen to be isometric and $S=0$, i.e.,
				$$L_2(x)=V_0^\ast \pi_0(L_1(x))V_0.$$
	\end{enumerate}
\end{corollary}

\begin{proof}
	First we will prove (1). By Theorem \ref{n-pos-of-tau} (8) the map $\tilde \tau$ is completely positive. 
	By the real version of Arveson's extension theorem \cite[Proposition 4]{CZ}
	(take 
	$E=B(\cH_1)$, $E_0=\widetilde{\mathcal S_1}$, $K_n(E)=\{A\in M_n(B(\cH_1))\colon A\succeq 0\}$), there exists a 
	completely positive extension $\tilde\tau:B(\cH_1)\to B(\cH_2)$ for $\tilde\tau:\widetilde{\mathcal S_1}\to \mathcal{S}_2$. 
			By the Stinespring theorem,
			there exist a separable real Hilbert space $\cK$, a $\ast$-homomorphism $\pi$ and an isometry
		      $V:\cH_2\to\ell^2$ such that 
			$\tilde\tau(C)=V^\ast \pi(C) V$ for all $C\in B(\cH_1\oplus \RR)$.
			Hence, 
			\begin{eqnarray*}
					L_2(x)=\tilde\tau(\left[\begin{array}{cc} L_1(x) & 0 \\ 0 & 1\end{array}\right])
				&=& V^\ast \pi(\left[\begin{array}{cc} L_1(x) & 0 \\ 0 & 1\end{array}\right]) V\\
				&=& V^\ast \pi(\left[ \begin{array}{cc} 0 & 0\\ 0& 1\end{array} \right])V+
					V^\ast \pi(
					\left[ \begin{array}{cc}  L_1(x)  & 0\\ 0& 0\end{array} \right])V.\\
			\end{eqnarray*}

	Now we will prove (2). Observe that
		$\pi(\left[ \begin{array}{cc} I_{\cH_1} & 0\\ 0& 0\end{array} \right])$ is a hermitian 
	idempotent, hence a projection onto
		$\cK_0:=\Ran(\pi(\left[ \begin{array}{cc} I_{\cH_1} & 0\\ 0& 0\end{array} \right]))$,
	by \cite[3.3 Proposition]{CON}.
	We define a contraction
				$$V_0:=P^{\cK}_{\cK_0}\pi(
						\left[ \begin{array}{cc} I_{\cH_1} & 0\\ 0& 0\end{array} \right])V:\cH_1\to \cK_0,$$
			where $P^{\cK}_{\cK_0}$ is a projection from $\cK$ to $\cK_0$.
		We define a new representation 
				$$\pi_0:B(\cH_1)\to B(\cK_0), \quad 
					A\mapsto
					P^{\cK}_{\cK_0}
					\pi(\left[ \begin{array}{cc} A& 0\\ 0& 0\end{array} \right])|_{\cK_0}.$$
		Since 
		$\Ran(\pi(\left[ \begin{array}{cc} A& 0\\ 0& 0\end{array} \right])\subseteq \cK_0$, 
	$\pi_0$ is well-defined.
		Thus
			\begin{eqnarray*}
					L_2(x)
				&=& S+
					V_0^\ast \pi_0(L_1(x))V_0,
			\end{eqnarray*}
	where $S:= V^\ast \pi(\left[ \begin{array}{cc} 0 & 0\\ 0& 1\end{array} \right])V\succeq 0$.		

	Finally, the proof of (3) is the same as the first part of the proof of (1) working with $\tau$ instead of $\tilde\tau$.
\end{proof}

\begin{remark} \label{rem-after-OperLP}
	\begin{enumerate}
		\item
	If $D_{L_1}(1)$ is unbounded, then in (2) of Theorem \ref{OperLP}, $V_0$ cannot always be chosen to be isometric 	
	(and hence $S=0$). See Example \ref{counterexample} above: if $L_2=V_0^\ast \pi(L_1) V_0$ for an isometry
	$V_0$, then $D_{^\text{h}\! L_1}\subseteq D_{^\text{h}\! L_2}$ which is not true. If $L_1$ and $L_2$ are monic linear matrix
	pencils and we restrict ourselves to $\ast$-homomorphisms $\pi$ mapping into finite dimensional spaces, 
	then $V_0$ can be chosen to be isometric if and only if $\Span\{A_1,\ldots,A_g\}$ does not contain a positive definite matrix 
	by \cite[Remark 4.4]{HKM2}.
		\item
	If $\cH_1$ is finite-dimensional, then every unital $\ast$-homomorphism $\pi:B(\cH_1)\to B(\cK)$ is unitarily equivalent to the direct sum of the identity $\ast$-homomorphism. Hence if $\cH_2$ is infinite-dimensional, then we can replace 
	in Corollary \ref{OperLP} above	
		$$\pi(\left[\begin{array}{cc} 0& 0 \\ 0 & 1\end{array}\right]),\quad
		\pi(\left[\begin{array}{cc} L_1 & 0 \\ 0 & 0\end{array}\right])\quad \text{and} \quad
		\pi_0(L_1)$$ 
	by
		$$\oplus_{i=1}^{\infty} \left[\begin{array}{cc} 0 & 0 \\ 0 & 1\end{array}\right],\quad
		\oplus_{i=1}^{\infty} \left[\begin{array}{cc} L_1 & 0 \\ 0 & 0\end{array}\right]\quad\text{and}\quad
		\oplus_{i=1}^{\infty} L_1\quad \text{respectively}.$$
	If $\cH_2$ is finite-dimensional, then those sums are finite as in \cite[Corollary 3.7]{HKM1} and \cite[Corollary 4.1]{HKM2}.
	\item The assumption of monicity of pencils can be replaced by the assumption of nonempty $D_{L_1}$ and the existence of an 
		invertible positive definite element in the linear span of coefficients of $L_1$. In the statement of Corollary \ref{OperLP}, $V$
		then becomes a bounded operator, which is not necesarrily a contraction.
	\end{enumerate}
\end{remark}

If $\cH_2$ is finite-dimensional of dimension $n$, then the inclusion $D_{L_1}(n)\subseteq D_{L_2}(n)$ is sufficient for the conclusion of Corollary \ref{OperLP} to hold.

\begin{corollary}[Operator-to-matrix linear Postivstellensatz] \label{OperLP-scal}
		Let 
	\begin{eqnarray*}
		L_1(x)&=& I_{\cH}+\sum_{j=1}^g A_{j}x_j\in \Sym_\cH\!\left\langle x\right\rangle,\quad
		L_2(x) =I_n+\sum_{j=1}^g B_{j}x_j\in \RR^{n\times n}\!\left\langle x\right\rangle
	\end{eqnarray*}
	be a monic linear operator polynomial and a monic linear matrix polynomial, respectively.
	If $D_{L_1}(n)\subseteq D_{L_2}(n)$ then:
	\begin{enumerate}
		\item There exist a Hilbert space $\cK$, an isometry $V\in B(\RR^{n},\cK)$,
			and a unital $\ast$-homomorphism $\pi:B(\cH\oplus \RR)\to B(\cK)$ such that
				$$L_2(x)=V^\ast \pi(\left[\begin{array}{cc} 0 & 0 \\ 0 & 1\end{array}\right])V
							+
					V^\ast \pi(\left[\begin{array}{cc} L_1 & 0 \\ 0 & 0\end{array}\right])V.$$
		\item There exist a Hilbert space $\cK_0$,  a contraction $V_0\in B(\RR^n,\cK)$,
			a unital $\ast$-homomorphism $\pi_0:B(\cH)\to B(\cK)$ and a positive semidefinite matrix $S\in \Sym_n$
			such that
				$$L_2(x)=S+V_0^\ast \pi_0(L_1)V_0.$$
		\item If $D_{L_1}(1)$ is bounded, then $V_0$ in (2) can be chosen to be an isometry and $S=0$, that is 
				$$L_2(x)=V_0^\ast \pi_0(L)V_0.$$
	\end{enumerate}
\end{corollary}

\begin{proof}
	By Theorem \ref{n-pos-of-tau}, $D_{L_1}(n)\subseteq D_{L_2}(n)$ implies $D_{L_1}\subseteq D_{L_2}$.
	Now everything follows by Corollary $\ref{OperLP}$.
\end{proof}

\subsubsection{Counterexample for non-monic pencils} We present an example which shows that the assumption of monicity of pencils in Corollary \ref{OperLP} is necessary. The example is a generalization of \cite[Example 2]{Z}.

\begin{example} \label{monicity-needed}
		Let $L(y)=\left[\begin{array}{cc} 1 & y\\ y & 0 \end{array}\right]$ be a linear matrix polynomial
	with a spectrahedron $D_{L}=\{0\}$. The polynomial $\ell(y)=y$ is non-negative on $D_{L}(1)$, but there do not
	exist a Hilbert space $\cK$, a unital $\ast$-homomorphism 
	$\pi:B(\RR^2)\to B(\cK)$, polynomials 
	$r_j\in \RR\!\left\langle y\right\rangle$ and operator polynomials
	$q_k\in B(\RR,\cK)\!\left\langle y\right\rangle$
	such that 
		$$y= \sum_j r_j^2 + \sum_k q_k^\ast \pi(L) q_k.$$
\end{example}

\begin{proof}	
		For $K=\RR^2$, the identity $\ast$-homomorphism $\pi$, i.e., $\pi(x)=x$, and polynomials 
	$r_j\in \RR\!\left\langle y\right\rangle$, $q_k\in \RR^{2\times 1}\!\left\langle y\right\rangle$
	the proof is already done in \cite[Example 2]{Z}.

		Let us now prove a general case. If $\cK$, $\pi$, $r_j$, $q_k$ existed, we would have
	\begin{eqnarray*}
		y &=& \sum_j r_j^\ast r_j + \sum_k q_k^\ast \pi(L) q_k\\
		   &=&  \sum_j r_j^\ast r_j + \sum_k q_k^\ast \pi(E_{11}) q_k + \sum_k q_k^\ast \pi(E_{12}+E_{21}) q_k y.
	\end{eqnarray*}
	Let us write 
		\begin{eqnarray*}
		      	r_j(y) &=&\sum_{m=0}^{N_j} r_{j,m}y\in \RR\!\left\langle y \right\rangle,\quad
			q_k(y)=\sum_{m=0}^{M_k} q_{k,m}y\in B(\RR,\cK)\!\left\langle y \right\rangle,
		\end{eqnarray*}
	where $N_j\in \NN_0$ is such that $a_{j,N_j}\neq 0$ and 
	$M_k\in \NN_0$ is such that $b_{k,M_k}\neq 0$. 
	Comparing the coefficients at $1$ of both sides we get
	\begin{eqnarray*}
		0 &=& \sum_j r_{j,0}^2 +\sum_k q_{k,0}^{\ast}\pi(E_{11}) q_{k,0}.
	\end{eqnarray*}
	Since 
		$$\pi(E_{11})=\pi(\left[\begin{array}{cc} 1 & 0\\ 0 & 0 \end{array}\right])=
			\pi(\left[\begin{array}{cc} 1 & 0\\ 0 & 0 \end{array}\right]^2)=
			\pi(\left[\begin{array}{cc} 1 & 0\\ 0 & 0 \end{array}\right])^2$$
	and
		$$\pi(E_{11})=\pi(E_{11}^\ast)=\pi(E_{11})^\ast,$$ 
	$\pi(E_{11})$ is a hermitian idempotent, hence a projection.
	Therefore 
		$$\sum_k q_{k,0}^{\ast}\pi(E_{11}) q_{k,0}=
			\sum_k q_{k,0}^{\ast}(\pi(E_{11}))^\ast\pi(E_{11}) q_{k,0}\geq 0.$$
	Thus, 
		$$r_{j,0}^2=\sum_k q_{k,0}^{\ast}(\pi(E_{11}))^\ast \pi(E_{11}) q_{k,0}=0.$$
	It follows that 
		$$r_{j,0}=0\quad \text{and}\quad 0=\pi(E_{11}) q_{k,0}=q_{k,0}^{\ast}\pi(E_{11})^\ast\in B(\cK,\RR).$$
	Indeed,
		\begin{eqnarray*}
			0
			&=&\left\langle((\pi(E_{11}) q_{k,0})^\ast \pi(E_{11}) q_{k,0})1, 
				1\right\rangle	\\
			&=&\left\langle\pi(E_{11}) q_{k,0}1, \pi(E_{11}) q_{k,0}1\right\rangle\\
			&=&\left\|\pi(E_{11}) q_{k,0}1\right\|.
		\end{eqnarray*}
	It follows that 
		$$\ran(q_{k,0})\in \ker \pi(E_{11})(\cK).$$
	Hence,
		\begin{eqnarray*}
			& &  q_{k,0}^{\ast}\pi(E_{12}+E_{21})q_{k,0}
			 =    q_{k,0}^{\ast}\pi(E_{11}E_{12}+E_{21}E_{11}) q_{k,0}\\
			&=& q_{k,0}^{\ast}\pi(E_{11}E_{12})q_{k,0}
			+q_{k,0}^{\ast}\pi(E_{21}E_{11}) q_{k,0}\\
			&=&\underbrace{q_{k,0}^{\ast}\pi(E_{11})}_{=q_{k,0}^{\ast}\pi(E_{11})^\ast=0}\pi(E_{12})q_{k,0}
			+q_{k,0}^{\ast}\pi(E_{21})\underbrace{\pi(E_{11}) q_{k,0}}_{0}
			=0.
		\end{eqnarray*}
	The coefficient at $y$ on the RHS is
		\begin{eqnarray*}
			& & \sum_j r_{j,1}^\ast \underbrace{r_{j,0}}_{0}+\sum_j \underbrace{r_{j,0}^\ast}_{0} r_{j,1}+
			\sum_k q_{k,1}^\ast \underbrace{\pi(E_{11})q_{k,0}}_{0}+
			\sum_k \underbrace{q_{k,0}^\ast \pi(E_{11})}_{0}q_{k,1}+\\
			&+& \sum_k \underbrace{q_{k,0}^\ast \pi(E_{12}+E_{21})q_{k,0}}_0.
		\end{eqnarray*}
	This is a contradiction, which finishes the proof.
\end{proof}

\subsection{Polar duals of free Hilbert spectrahedra and free Hilbert spectrahedrops} \label{polar-duals-soln}
In this subsection we describe operator free polar duals of free Hilbert spectrahedra and free Hilbert spectrahedrops (see Theorems \ref{OperLP-satz-unbounded} and \ref{OperLP-trunc} below).

\begin{theorem} \label{OperLP-satz-unbounded}
		Let $A\in \Sym_\cH^g$  be a tuple of self-adjoint operators from $\Sym_\cH$. Then the operator free Hilbert polar dual 	
	$(D_{L_A}^{\cK})^{\cK,\circ}$ is given by
		$$\operconv_{\cK}(A)=(D_{L_A}^{\cK})^{\cK,\circ}.$$
\end{theorem}

\begin{proof}
	 It is easy to see that $\operconv_\cK\{A\}\subseteq (D_{L_A}^{\cK})^{\cK,\circ}$. Indeed, let us take $X:=V^\ast \pi(A) V\in\operconv_\cK\{A\}$, where $\cG$ is a separable real Hilbert space, $V:\cK\to \cG$ a contraction and $\pi: B(\cH)\to B(\cG)$ a unital $\ast$-homomorphism. We have to prove that 
		$L_X|_{D^{\cK}_{L_{A}}}\succeq 0$. 
	For every	$Y\in  D_{L_{A}}^{\cK}$ we have
		\begin{eqnarray*}
		L_{X}(Y)
			&=& L_{V^\ast \pi(A) V}(Y)\\
			&=& I\otimes I + \sum_j V^\ast \pi(A_j)V\otimes Y_j\\
			&\succeq& (V\otimes I)^\ast (I\otimes I + \sum_j \pi(A_j)\otimes Y_j) (V\otimes I).
		\end{eqnarray*}
	Using that
		\begin{eqnarray*}
			I\otimes I + \sum_j \pi(A_j)\otimes Y_j
			&=& (\pi\otimes I) (I\otimes I + \sum_j A_j \otimes Y_j)\\
			&=& (\pi\otimes I) (L_{A}(Y))\succeq 0,
		\end{eqnarray*}
	where the last inequality follows by $\pi\otimes I$ being a $\ast$-homomorphism,
	it follows that $X\in (D_{L_A}^{\cK})^{\cK,\circ}$.

	Let us now prove the opposite direction, i.e., 
	$(D_{L_A}^{\cK})^{\cK,\circ}\subseteq \operconv_\cK\{A\}.$
	Suppose that $X\in \mathbb{S}_\cK^g$ belongs to $(D_{L_A}^{\cK})^{\cK,\circ}$. 
	We have to prove that $X \in\operconv_\cK(A).$ 
	By assumption ${L_X}|_{D_{L_A}^{\cK}}\succeq 0$. 
	Using Corollary \ref{OperLP} (2) there exist a separable real Hilbert space $\cG$, a
	contraction $V:\cK\to \cG$, a unital $\ast$-homomorphism 
	$\pi:B(\cH)\to B(\cG)$ and a positive semidefinite 
	operator $S\in B(\cK)$ such that
		$L_X=S+V^\ast \pi(L_A)V$. In particular, $X=V^{\ast} \pi(A) V\in \operconv_\cK(A).$
\end{proof}

The set $\operconv_{\cK}(A)$ is closed in the weak operator topology.

\begin{corollary}
	For a tuple $A\in \Sym_\cH^g$
	the set $\operconv_\cK(A)$ is closed in the weak 
		operator topology. In particular, it is closed in the norm topology.
\end{corollary}

\begin{proof}
	Let $\overline{\text{oper-conv}}^{w}_{\cK}(A)$ denote the closure of $\operconv_\cK(A)$ in the weak operator topology. We have
	\begin{eqnarray*}
		(D_{A}^\cK)^{\cK,\circ} 
				&\underbrace{=}_{\text{Theorem }\ref{OperLP-satz-unbounded}}& \operconv_{\cK}(A)
				\subseteq\overline{\operconv}_\cK(A)\\
				&\subseteq& \overline{\text{oper-conv}}^{w}_{\cK}(A)
				\underbrace{\subseteq}_{\text{easy}} (D_{A}^\cK)^{\cK,\circ} .
	 \end{eqnarray*}
\end{proof}

\begin{theorem}	\label{OperLP-trunc}
		Let $(\Omega,\Gamma)\in \Sym_\cH^{g+h}$ be a tuple of operators from $\Sym_\cH^{g+h}$ 
	and $$\fK=\proj_x D_{L_{(\Omega,\Gamma)}}^{\cK}.$$ The operator free Hilbert polar dual $\fK^{\cK,\circ}$
	is the set 
	\begin{eqnarray*}
		\fK^{\cK,\circ} 
		&=& \{A\in \mathbb{S}_{\cK}^{g}\colon (A,0)\in D_{L}^{\cK,\circ}\}\\
		&=& \{A\in \mathbb{S}_{\cK}^{g}\colon 
			\exists\; \text{a separable real Hilbert space }\cG, \text{ an isometry }V:\cK\to\cG \\
		&&\text{ and }
		\ast\!\text{-homomorphism }\pi:B(\cH\oplus \RR)\to B(\cG) \text{ s.t. }\\
		&& 
			A=V^{\ast}\pi(\left[\begin{array}{cc} \Omega & 0\\ 0& 	
			0\end{array}\right])V,\quad
			0=V^{\ast}\pi(\left[\begin{array}{cc} \Gamma & 0\\ 0& 	
			0\end{array}\right])V\},
	\end{eqnarray*}
	where $\Omega'=\Omega\oplus 0$ and $\Gamma'=\Gamma\oplus 0$.

	If $\fK(1)$ is bounded, then $\fK^{\cK,\circ}$
	is the set
	\begin{eqnarray*}
		\fK^{\cK,\circ} 
		&=& \{A\in \mathbb{S}_{\cK}^{g}\colon (A,0)\in D_{L}^{\cK,\circ}\}\\
		&=& \{A\in \mathbb{S}_{\cK}^{g}\colon 
			\exists\text{ a separable real Hilbert space } \cG, \text{an isometry }V:\cK\to\cG \\
		&& \text{ and}
			 \ast\!\text{-homomorphism }\pi:B(\cH)\to B(\cG) \text{ s.t. }\\ 
		&&
			A=V^{\ast}\pi(\Omega)V,\quad 0=V^{\ast}\pi(\Gamma)V\}.
	\end{eqnarray*}
\end{theorem}

\begin{proof}
By definition, 
	$\fK^{\cK,\circ} = \{A\in \mathbb{S}_{\cK}^{g}\colon L_A|_{\fK}\succeq 0\}.$
It is easily seen that
	$$\{A\in \mathbb{S}_{\cK}^{g}\colon L_A|_{\fK}\succeq 0\}=
	\{A\in \mathbb{S}_{\cK}^{g}\colon L_{(A,0)}|_{D_{L}^{\cK}}\succeq 0\}.$$
To get the second equality in the first statement of the theorem use Corollary \ref{OperLP}.
If not only $\fK(1)$ but also $D_{L_{(\Omega,\Gamma)}}(1)$ is bounded, then the second equality in the second statement of the theorem also follows by Corollary \ref{OperLP}. From now on we assume that only $\fK(1)$ is bounded.
Let $(A,0)\in D_{L}^{\cK,\circ}$. \\

\noindent\textbf{Claim:} $D_{^{\text{h}}\! L_{(\Omega,\Gamma)}}\subseteq D_{^{\text{h}}\! L_{(A,0)}}$. \\

By assumption 
$D_{L_{(\Omega,\Gamma)}}\subseteq D_{ L_{(A,0)}}$. 
Let
 	$$(X_0,\ldots,X_g,Y_1,\ldots,Y_k)\in D_{^\text{h}\! L_{(\Omega,\Gamma)}}(n)\quad\text{for some }n\in \NN.$$

First we prove that $X_0\succeq 0$. It suffices to prove this fact for $n=1$ (by the same reduction as in the proof of
Theorem \ref{n-pos-of-tau}).
Let as assume on contrary that $X_0<0$.
From
	$$I_\cH \otimes X_0 + \sum_j \Omega_j \otimes X_j + \sum_k \Gamma_k \otimes Y_k \succeq 0,$$
it follows that
	$$I_\cH \otimes tX_0 + \sum_j \Omega_j \otimes tX_j + \sum_k \Gamma_k \otimes tY_k \succeq 0\quad
\text{for every }t>0.$$ Since $tX_0<0$ for $t>0$, it follows that
	$$I_\cH \otimes 1 + \sum_j \Omega_j \otimes tX_j + \sum_k \Gamma_k \otimes tY_k \succeq 0\quad
\text{for every }t>0.$$
	Therefore 
	$$(tX_1,\ldots,tX_g)\in \fK(1)\quad \text{for every }t>0.$$ 
If $(X_1,\ldots,X_g)\neq 0^g$, this contradicts the 
boundedness of $\fK(1)$. Else  $(X_1,\ldots,X_g)= 0^g$. But then 
	$$\sum_k \Gamma_k \otimes Y_k\succ -I_\cH \otimes X_0=I_\cH\otimes \left|X_0\right|,$$ 
and hence for every $(X_1,\ldots,X_g)\in \RR^g$ there exists $t>0$ such that 
	$$I_\cH \otimes 1 + \sum_j \Omega_j \otimes X_j + \sum_k \Gamma_k \otimes tY_k \succeq 0.$$
This again contradicts  the boundedness of $\fK(1)$. 

Now for $(X_0,\ldots,X_g,Y_1,\ldots,Y_k)\in D_{^\text{h}\! L_{(\Omega,\Gamma)}}(n)$  with $X_0\succeq 0$,
as in the proof of Theorem \ref{n-pos-of-tau},
it follows that
		$(X_0,\ldots,X_g,Y_1,\ldots,Y_k)\in D_{^{\text{h}}\! L_{(A,0)}}.$
This concludes the proof of Claim.\\

By Theorem \ref{n-pos-of-tau}.(2), the unital linear map 
	$$\tau:\Span\{ I_\cH,\Omega_1,\ldots,\Omega_g,\Gamma_1,\ldots,\Gamma_h\}\to \Span\{I_\cK,A_1,\ldots,A_g\},$$
	$$\Omega_j \mapsto A_j, \quad \Gamma_k \mapsto 0_\cK,$$
is completely positive. Now the same proof as for (3) of Corollary \ref{OperLP} applies to get the second equality in the second statement of the theorem.
\end{proof}

\section{Equality of free spectrahedra}\label{LOI-equality-problem-soln}

In  this section we consider the equality of free spectrahedra. Our main result, which extends 
\cite[Theorem 1.2]{HKM1} from bounded to unbounded spectrahedra, states that up to obvious redundancies, two linear matrix pencils define the same free spectrahedron if and only if they are unitarily equivalent. We refer the reader to \S  \ref{equality-of-free-sp} for basic definitions, context and the precise statement of the main result, i.e., Theorem \ref{Unbound-GS-Intro}. In Subsection \ref{LMI-equality-soln} we present the proof of Theorem \ref{Unbound-GS-Intro} (see Theorem \ref{UnbounLG} below), in Subsections \ref{non-existence-of-reduct} and 
\ref{counterexample for the LG} we show that Theorem \ref{Unbound-GS-Intro} does not extend to linear operator pencils. More precisely, in Subsection \ref{non-existence-of-reduct}
we present an operator pencil that does not have a whole subpencil which is $\sigma$-minimal, while in Subsection \ref{counterexample for the LG} we give two $\sigma$-minimal operator pencils with the same free spectrahedron but are not unitarily equivalent.

\subsection{Characterization of the equality $D_{L_1}=D_{L_2}$} \label{LMI-equality-soln}
The main result of this subsection is Theorem \ref{UnbounLG}. We use the approach from \cite[\S 3.3]{HKM1}, where the result is proved for a
bounded spectrahedron $D_{L_1}(1)=D_{L_2}(1)$. The crucial observation is that boundedness of $D_{L_1}(1)=D_{L_2}(1)$ is not essential for the approach to work, as it works for a
$\sigma$-minimal pencils $D_{L_1}=D_{L_2}$ that satisfy $D_{^h \!  L_1}=D_{^h\!  L_2}.$ 
We prove this holds also in the unbounded case.

\begin{theorem}[Linear Gleichstellensatz] \label{UnbounLG}
		Let 
	\begin{eqnarray*}
		L_1&=& I_{d}+\sum_{j=1}^g A_{j}x_j\in \Sym_{d}\!\left\langle x\right\rangle, \quad
		L_2= I_{e}+\sum_{j=1}^g B_{j}x_j\in \Sym_{e}\!\left\langle x\right\rangle
	\end{eqnarray*}
	be monic linear matrix pencils. If $D_{L_1}=D_{L_2}$ and $L_1$, $L_2$ are $\sigma$-minimal, then $d=e$ and 
	there is a unitary matrix $n\times n$ matrix $U$ such that 
		$$U^\ast L_1 U=L_2.$$
\end{theorem}


To prove the theorem we will need some preliminary results.
Even though Theorem \ref{UnbounLG} does not extend from matrix to operator pencils, most of the preliminary results do
in fact work for operator pencils.

Let $\cH$, $\cK$ be separable real Hilbert spaces and let us define the unital linear spaces
\begin{eqnarray*}
	\mathcal S_1&:=& \Span\{I_\cH,A_j\colon j=1,\ldots,g\}\subseteq \Sym_\cH,\\  
	\mathcal S_2&:=&\Span\{I_\cK,B_j\colon j=1,\ldots,g\}\subseteq \Sym_\cK.
\end{eqnarray*}

The following proposition translates the equality $D_{^\text{h}\! L_1}=D_{^\text{h}\! L_2}$ into properties of the unital map $\tau$ mapping from $\cS_1$ to $\cS_2$. Recall from Section 2 that a map $\tau$ is completely isometric if and only if every ampliation $\tau_n$, $n\in \NN$, is an isometry.

\begin{proposition} \label{completely-pos-and-isom}
		Let $L_1= I_{\cH}+\sum_{j=1}^g A_{j}x_j\in \Sym_{\cH}\!\left\langle x\right\rangle$
	and $L_2= I_{\cK}+\sum_{j=1}^g B_{j}x_j\in \Sym_{\cK}\!\left\langle x\right\rangle$
	be monic linear operator pencils. Then $D_{^\text{h}\! L_1}=D_{^\text{h}\! L_2}$ if and only if 
	the unital linear map $\tau:\mathcal S_1 \to \mathcal S_2$, $A_j\mapsto B_j,$
	is well-defined and completely isometric.
\end{proposition}

 For the proof of the implication $(\Rightarrow)$ we will need two observations. The first is an observation on convex sets (see Lemma \ref{equality-of-convex-sets}) and the second connects the equality of free spectrahedra of homogenizations of monic pencils with the equality of their boundaries (see Lemma \ref{hom-equalities-and-borders}).

For a set $C\subseteq\RR^n$ we write $bC$ for its boundary (in the topology of $\RR^n$).
%
\begin{lemma}\label{equality-of-convex-sets}
		Let $C_1, C_2\subseteq \RR^n$ be closed convex sets, $0\in \Int C_1 \cap \Int C_2$. If $b C_1\subseteq b C_2$ then $C_2\subseteq C_1$.
\end{lemma}

\begin{proof}
		By the way of contradiction, assume $C_2\not\subseteq C_1$ and let 
	$a\in C_2\setminus C_1$. The interval $[0,a]$ intersects $bC_1$
	in $\mu a$ for some $0<\mu<1$. Then $\mu a\in bC_1$. By assumption $\mu a\in bC_2$. Since $0\in \Int C_2$, $C_2$ contains a small disk $D(0,\epsilon)$. Then the convex hull $K$ of the set 
$D(0,\epsilon)\cup \{a\}$ is contained in $C_2$ and $\mu a\in \Int K\subseteq \Int C_2$ contradicting $\mu a \in bC_2$.
\end{proof}

\begin{lemma}\label{hom-equalities-and-borders}
		Let $L_1\in \Sym_{\cH}\!\left\langle x\right\rangle$ and $L_2 \in\Sym_{\cK}\!\left\langle x\right\rangle$
	be monic linear operator pencils. Then $D_{^\text{h}\! L_1}=D_{^\text{h}\! L_2}$ if and only if $\partial D_{^\text{h}\! L_1}=\partial D_{^\text{h}\! L_2}$.
\end{lemma}

\begin{remark}
	Note that for a homogeneous linear pencil  $L\in \Sym_{d}\!\left\langle x\right\rangle$ it is always true that
	$bD_L(n)\subseteq \partial D_L(n)$, but not necessarily $bD_L(n)=\partial D_L(n)$. 
	Taking 
		$$L(x_0)=\left[\begin{array}{cc} 1 & 0\\ 0 & 0 \end{array}\right] x_0$$ 
	we see that 
		$$bD_L(1)=\{0\}\subset \{x_0\in \RR\colon x_0\geq 0\}=\partial D_L(1).$$
	However, if $L$ is monic, then we have $bD_L(n)= \partial D_L(n)$ for every $n\in \NN$ by Lemma \ref{boundaries} below.
\end{remark}

\begin{lemma}\label{boundaries}
	Let $L=I_\cH x_0+\sum_{j=1}^{g}A_j x_j\in \Sym_{\cH}\!\left\langle x\right\rangle$
	be a monic homogeneous linear operator pencil. Then $bD_L(n)=\partial D_L(n)$ for every $n\in \NN$.
\end{lemma}

\begin{proof}
	The nontrivial inclusion is $\partial D_L(n)\subseteq bD_L(n)$. Since 
		$$D_L(n)=\Int D_L(n) \cup bD_L(n),$$ 
	it suffices to prove that 
	if $X\in \partial D_L(n)$ then $X\notin \Int D_L(n)$. Let 
		$$Y=(I_n,0,\ldots,0)\in \Sym_n^{g+1}.$$
	If $X\in \partial D_{L}(n)$, then $L(X)\not\succ 0$ and hence
		$$L(X-\epsilon Y)=L(X)-\epsilon I_d\otimes I_n\not\succeq 0$$ 
	for every $\epsilon>0$. Thus $X\notin \Int D_L(n)$, which concludes the proof.
\end{proof}

\begin{proof}[Proof of Lemma \ref{hom-equalities-and-borders}]
		Since $L_1$ and $L_2$ are monic, the equality  $\partial D_{^\text{h}\! L_1}=\partial D_{^\text{h}\! L_2}$ is equivalent
	to  the equality $b D_{^\text{h}\! L_1}(n)=b D_{^\text{h}\! L_2}(n)$ for every $n\in \NN$, by Lemma \ref{boundaries}. 
	Now the implication ($\Rightarrow$) is obvious, while the implication  ($\Leftarrow$) follows by Lemma \ref{equality-of-convex-sets}.
\end{proof}

\begin{proof}[Proof of Proposition \ref{completely-pos-and-isom}]
	First we will prove the direction $(\Rightarrow).$
	By Lemma \ref{well-defined-map}, the inclusion 
	$D_{^\text{h}\! L_1}(1)\subseteq D_{^\text{h}\! L_2}(1)$ implies that $\tau$ is well-defined.
	 By Theorem \ref{n-pos-of-tau}.(1),  the inclusion $D_{^\text{h}\! L_1}\subseteq D_{^\text{h}\! L_2}$
	implies that $\tau$ is completely positive. 
	
	Finally we prove that the map $\tau$ is completely isometric.	
	Suppose $T\in \mathcal{S}_1^{n\times n}$ has norm one. Then 
		$$W:=\left[\begin{array}{cc} I & T \\ T^\ast & I\end{array}\right]\succeq 0\quad\text{and}\quad W\not\succ 0.$$
	By complete positivity of $\tau$, we have that
		$$\tau(W):=\left[\begin{array}{cc} I & \tau(T) \\ \tau(T^\ast) & I\end{array}\right]\succeq 0.$$
	By Lemma \ref{hom-equalities-and-borders}, 
	$\partial D_{^\text{h}\! L_1}=\partial D_{^\text{h}\! L_2}$ and thus $\tau(W)\not\succ 0$.
	Hence, $\tau(T)$ has norm one. Thus  $\tau$ is completely isometric.

	Now we prove the implication $(\Leftarrow)$, i.e., $D_{^\text{h}\! L_1}= D_{^\text{h}\! L_2}$.
	Since the map $\tau$ is unital completely isometric, the map $\tau^{-1}$ is also unital completely isometric. 
	By \cite[Corollary 7.6]{PAU}, $\tau$ and $\tau^{-1}$ are completely positive. Thus, by Theorem \ref{n-pos-of-tau}.(2), the equality
	$D_{^\text{h}\! L_1}= D_{^\text{h}\! L_2}$ follows.	
\end{proof}

\begin{remark} \label{complete-isometry-and-equality-of-sets}
		The equality $D_{^\text{h}\! L_1}=D_{^\text{h}\! L_2}$ clearly implies
	$D_{L_1}=D_{L_2}$. If $D_{L_1}(1)$ is bounded, then $D_{L_1}=D_{L_2}$ implies	$D_{^\text{h}\! L_1}=D_{^\text{h}\! L_2}$. Hence, for a bounded set 
	$D_{L_1}(1)$ we can replace $D_{^\text{h}\! L_1}=D_{^\text{h}\! L_2}$ by
	$D_{L_1}=D_{L_2}$. Moreover, since closed bounded convex sets are the convex hulls of their boundaries \cite[(3.3) Theorem]{BARV}, we can replace $D_{L_1}=D_{L_2}$ by
	$\partial D_{L_1}\subseteq \partial D_{L_2}$ (This implies $D_{L_1}\subseteq D_{L_2}$ and Lemma \ref{equality-of-convex-sets} implies $D_{L_2}\subseteq D_{L_1}$).
	
			For unbounded sets this cannot be done in general. Take
		$\tilde L_2=L_1\oplus L_2$ and $\tilde L_1=L_1$ where $L_1$ and
		$L_2$ are as in Example \ref{counterexample}. Then we have
		$D_{\tilde L_2}=D_{L_1}=D_{\tilde L_1}$ but
		$(-1,\frac{1}{2},\frac{1}{2})\in 
			D_{^\text{h}\! \tilde L_1}(1)\setminus D_{^\text{h}\! \tilde L_2}(1)$.
\end{remark}

Proposition \ref{Central-proj-and-Silov} below states the properties of the 
$C^\ast$-algebra generated by a $\sigma$-minimal \emph{matrix} pencil. 
This is the direction $(\Rightarrow)$ of 
\cite[Proposition 3.17]{HKM1} without the assumption of 
boundedness of $D_{L_1}(1)$ which is not needed in the proof. (Note also that \textsl{truly linear pencil} in the assumptions of \cite[Proposition 3.17]{HKM1} should be replaced by \textsl{monic linear pencil}.) The result uses Arveson's noncommutative Choquet theory \cite{ARV69, ARV08, ARV10} and to a lesser extent \cite{DM05}.

The following definitions are needed in the statement of Proposition \ref{Central-proj-and-Silov}.
Let 
	$$\cS:= \Span\{I_d,A_j\colon j=1,\ldots,g\}\subseteq \Sym_d$$ 
be the unital linear subspace in $\Sym_d$ and let $C^{\ast}(\cS)$ be the real $C^\ast$-algebra generated by $\cS$ in $M_d(\RR)$.
Let $K$ be the biggest two sided ideal of $C^\ast(\cS)$ such that the natural map
	$$C^{\ast}(\cS)\to C^{\ast}(\cS)/K,\quad a\mapsto \tilde a:=a+K$$
is completely isometric on $\cS$. $K$ is called the \textbf{\v{S}ilov ideal} for $\cS$ in $C^{\ast}(\cS)$.
A \textbf{central projection} $P$ in $C^{\ast}(\cS)$ is a projection $P\in C^{\ast}(\cS)$ such that $PA=AP$ for all $A\in C^{\ast}(\cS)$. A projection $Q$ is a \textbf{reducing projection for} $C^{\ast}(\cS)$ if $QA=AQ$ for all $A\in C^{\ast}(\cS)$. 

\begin{proposition}\label{Central-proj-and-Silov}
	Let $L\in \Sym_d\! \left\langle x\right\rangle$ be a monic linear matrix pencil. If $L$ is $\sigma$-minimal, then
	\begin{enumerate}
		\item Every minimal reducing projection $Q$ is in $C^\ast(\mathcal S)$.
		\item The \v{S}ilov ideal of $C^\ast(\mathcal S)$ is $(0)$.
	\end{enumerate}
\end{proposition}

\begin{proof}
	The proof is the same to the proof of the implication $(\Rightarrow)$ of \cite[Proposition 3.17]{HKM1} which assumes $D_L(1)$ is bounded. (Note also that \textsl{truly linear pencil} in the assumptions 
	should be replaced by \textsl{monic linear pencil}.)
	We only emphasize that the complete isometry of the mapping $C^{\ast}(\cS)P \ni A \mapsto A(I-Q)$	
	at the end of the second paragraph and Proposition \ref{completely-pos-and-isom} implies that $^\text{h}\! L$ is not $\sigma$-minimal.
	Now by Lemma \ref{exten-reduc}.(1) below, $L$ is not $\sigma$-minimal.
	The same arguments are used also for the contradiction with $\sigma$-minimality of $L$ in the third and fourth paragraphs.
\end{proof}

\begin{remark}
	\begin{enumerate}
		\item[(1)]
		For the implication $(\Leftarrow)$ of \cite[Proposition 3.17]{HKM1}
	one needs the fact that the equality $D_{L}=D_{\tilde L}$ implies 
	the complete isometry of the map $S\mapsto SQ$ where the coefficients
	of $L$ and $\tilde L$ are $I,A_1,\ldots, A_g$ and $Q, QA_1Q,\ldots,
	Q A_g Q$ respectively. By Remark \ref{complete-isometry-and-equality-of-sets} this is not always true for unbounded sets $D_L(1)$. Therefore the proof of the implication $(\Leftarrow)$ of \cite[Proposition 3.17]{HKM1} does not extend to unbounded sets $D_L(1)$. However, this implication is not needed for our purposes. 
		\item[(2)] It is not known if Proposition \ref{Central-proj-and-Silov} extends to operator pencils. 
		However, extending Proposition \ref{Central-proj-and-Silov} to operator
		pencils does not extend Theorem \ref{completely-pos-and-isom} to operator pencils 
	(see \S \ref{non-existence-of-reduct}  and \S \ref{counterexample for the LG}). 
	The proof of Theorem \ref{completely-pos-and-isom} also uses the classification of finite dimensional real $C^{\ast}$-algebras. 
	\end{enumerate}
\end{remark}

The following lemma states the equivalence between the $\sigma$-minimality of a monic linear operator pencil $L$ and its homogenization
$^\text{h}\! L$ and the connection with the $\sigma$-minimality of the extended linear pencil $L\oplus 1$.

\begin{lemma}\label{exten-reduc}
	Let 
		$L(x)= I_{\cH}+\sum_{j=1}^g A_{j}x_j\in \Sym_{\cH}\!\left\langle x\right\rangle$
	be a monic linear operator pencil and $\tilde L= L\oplus 1 \in  \Sym_{\cH\oplus\RR}\!\left\langle x\right\rangle$
	its extended linear pencil. Then:
	\begin{enumerate}
		\item If $L$ is $\sigma$-minimal, then $^\text{h}\! L$ is $\sigma$-minimal.
		\item If $L$ is $\sigma$-minimal and $D_{^\text{h}\! L}\neq D_{^\text{h}\! \tilde L}$, then $^\text{h}\! \tilde L$ is $\sigma$-minimal.
	\end{enumerate}
\end{lemma}

\begin{proof}
	First we prove (1). 
	Assume that $L$ is $\sigma$-minimal and prove that $^\text{h}\! L$ is $\sigma$-minimal. By the way of contradiction, let 
	$^\text{h}\! L$ have a proper whole subpencil
		$$V_1^\ast {^\text{h}\! L}V_1 =I_{H_1}x_0+\hat A_j x_j\quad \text{with} \quad  
			D_{V_1^\ast {^\text{h}\! L} V_1}=D_{^\text{h}\! L},$$
	where $V_1:H_1\to \cH$ is the inclusion of a proper closed subspace $H_1$ into $\cH$. 
	The equality  $D_{V_1^\ast {^\text{h}\! L} V_1}=D_{^\text{h}\! L}$ implies that 
 	$D_{V_1^\ast L V_1}=D_{L}$.
	Hence	$V_1^\ast LV_1$ is a proper whole subpencil of $L$, which contradicts the $\sigma$-minimality of $L$.

	It remains to prove (2).
	If $^\text{h}\! \tilde L$ is not $\sigma$-minimal, then it has 
	a proper whole subpencil 
		$$V_2^\ast {^\text{h}\! \tilde L} V_2=I_{H_2}x_0+\hat A_j x_j\quad\text{with}\quad 
			D_{V_2^\ast {^\text{h}\! \tilde L} V_2 }=D_{^\text{h}\! \tilde L}$$ 
	where $V_2:H_2 \to \cH\oplus \RR$ is the inclusion of a closed subspace $H_2$ into $\cH\oplus \RR$.
	The equality $D_{V_2^\ast {^\text{h}\! \tilde L} V_2 }=D_{^\text{h}\! \tilde L}$ implies that
		\begin{equation}\label{equal-of-sets} D_{V_2^\ast \tilde LV_2}=D_{\tilde L}=D_{L}. \end{equation}

	Let $P_{\cH}:\cH\oplus \RR\to \cH$ be the projection onto $\cH$, i.e.,
		$P_{\cH}(v,\alpha)=v.$\\
	
	\textbf{Claim:} $P_{\cH}(H_2)=\cH$.\\

	For every tuple $X\in \Sym_n^g$ we have
		\begin{eqnarray}
		 (V_2^{\ast}\tilde L V_2) (X)&=& (V_2\otimes I_n)^\ast \tilde L(X)(V_2\otimes I_n) \nonumber\\
			&=&  (V_2\otimes I_n)^\ast (L(X)\oplus I_n) (V_2\otimes I_n) \nonumber\\
			&=& (V_2\otimes I_n)^\ast (L(X)\oplus 0) (V_2\otimes I_n)+(V_2\otimes I_n)^\ast (0\oplus I_n) (V_2\otimes I_n)	
			\nonumber\\
			&\succeq&  (V_2\otimes I_n)^\ast (L(X)\oplus 0) (V_2\otimes I_n). \label{dominance}
		\end{eqnarray}
	Let 
		$V_3:P_{\cH}(H_2)\to \cH$
	be the inclusion of a closed subspace $P_{\cH}(H_2)$ into $\cH$.
	Notice that  
		$$
		V_2(v\oplus \alpha)=V_3(v)\oplus \alpha\quad \text{where }v\oplus \alpha\in H_2 \text{ and }
			v\in P_{\cH}(H_2), \alpha\in \RR.
		$$
	Every vector $f\in H_2\otimes \RR^n$ is of the form $$f=\sum_j(v_j\oplus \alpha_j)\otimes u_j$$ where
	$v_j\in P_{\cH}(H_2), \alpha_j\in \RR$ and $u_j\in \RR^n$ for each $j$. We have
		\begin{eqnarray*}
	&&	\left\langle(V_2\otimes I_n)^\ast (L(X)\oplus 0) (V_2\otimes I_n)f,f\right\rangle_{H_2\otimes \RR^n}=\\
	&=&	\left\langle (L(X)\oplus 0) \sum_{j=1}^k (v_j\oplus \alpha_j)\otimes u_j,
			\sum_{j=1}^k (v_j\oplus \alpha_j)\otimes u_j\right\rangle_{H_2\otimes \RR^n}\\
	&=&	\left\langle L(X) \sum_{j=1}^k v_j\otimes u_j,
			\sum_{j=1}^k v_j\otimes u_j\right\rangle_{P_{\cH}(H_2)\otimes \RR^n}\\
	&=&	\left\langle (V_3\otimes I_n)^\ast L(X) (V_3\otimes I_n) \sum_{j=1}^k v_j\otimes u_j,
			\sum_{j=1}^k v_j\otimes u_j\right\rangle_{P_{\cH}(H_2)\otimes \RR^n}.
		\end{eqnarray*}
	By this calculation and (\ref{dominance}) we conclude that
		\begin{equation} \label{eq-4-for-minimality} D_{V_3^{\ast}L V_3}\subseteq D_{V_2^{\ast}\tilde L V_2}.\end{equation}
	If $P_{\cH}(H_2)\subsetneq \cH$, then by $\sigma$-minimality of $L$ we have
		\begin{equation} \label{eq-3-for-minimality} D_{L}\subsetneq D_{V_3^\ast L V_3}. \end{equation}
	Now (\ref{equal-of-sets}),  (\ref{eq-4-for-minimality}) and (\ref{eq-3-for-minimality}) give us a contradiction
		$$D_L\subsetneq D_{V_3^\ast L V_3}\subseteq D_{V_2^\ast \tilde L V_2} =D_L.$$
	Thus $P_{\cH}(H_2)=\cH$ which proves the claim.\\

	Since by assumption $D_{^\text{h}\! L}\neq D_{^\text{h}\! \tilde L}$, there is $n\in \NN$
	and a tuple $(X_0,\ldots,X_g)\in \Sym_n^{g+1}$ such that
		$$^\text{h}\! L(X)\succeq 0 \quad \text{and}\quad ^\text{h}\!\tilde L(X)=
			^\text{h}\!\!\! L(X)\oplus X_0\not\succeq 0.$$
	Therefore $X_0\not\succeq 0$. So there exists a vector $u\in \RR^n$ such that $\left\langle X_0u,u\right\rangle=-1$.
	Let $V:\cH\oplus \RR u \to \cH\oplus \RR^n$ be the inclusion of $\cH\oplus \RR u$ into $\cH\oplus \RR^n$.
	Then 
		\begin{eqnarray*}
			V^\ast {^\text{h}\! L}(X) V	=
				{^\text{h}\! L}(-1,\left\langle X_1 u,u\right\rangle,\ldots,\left\langle X_g u,u\right\rangle)&\succeq& 0,\\
		{^\text{h}\! \tilde L}(-1,\left\langle X_1 u,u\right\rangle,\ldots,\left\langle X_g u,u\right\rangle)=
			V^\ast {^\text{h}\! L}(X) V\oplus -1&\not\succeq& 0.
		\end{eqnarray*}
	Thus 
 		$$(-1,\mu_1,\ldots,\mu_g):=(-1,\left\langle X_1 u,u\right\rangle,\ldots,\left\langle X_g u,u\right\rangle)
			\in D_{^\text{h}\! L}(1)\setminus D_{^\text{h}\! \tilde L}(1).$$ 
	Hence $\sum_j A_j \mu_j \succ I_{\cH}$ and $\sum_j A_j \mu_j$ is invertible. Therefore
		\begin{equation} \label{invertibility} 
			(\sum_j A_j \mu_j)\cH=\cH.
		\end{equation} 
Using this and the fact that $H_2$ is reducing for every $\tilde A_j$
	we conclude that 
		$$\cH\oplus 0\underbrace{=}_{(\ref{invertibility})} 
			(\sum_j A_j \mu_j)\cH\oplus 0\underbrace{=}_{\text{Claim}}
		(\sum_j  A_j \mu_j)P_\cH(H_2)\oplus 0= (\sum_j  \tilde A_j \mu_j)H_2 \subseteq H_2.$$
	Since $H_2$ is a proper closed subspace of $\cH\oplus\RR$, it follows that $H_2=\cH\oplus 0$. Therefore 
		$$V_2^{\ast}{^\text{h}\! \tilde L} V_2={^\text{h}\! L}$$ 
	is proper whole subpencil of $^\text{h}\! \tilde L$. In particular $D_{^\text{h}\! L}= D_{^\text{h}\! \tilde L}$, which 
	contradicts the assumption in the statement of the lemma.
\end{proof}

\begin{proof}[Proof of Theorem \ref{UnbounLG}]
%
	We will separate two cases. In Case 1 we will establish the theorem under the assumption 
	$D_{^\text{h}\! L_1}=D_{^\text{h}\! L_2}$. In Case 2 we will show that assumming
	$D_{^\text{h}\! L_1}\neq D_{^\text{h}\! L_2}$ leads to a contradiction.\\
	
	\noindent	\textbf{Case 1:} $D_{^\text{h}\! L_1}=D_{^\text{h}\! L_2}$.\\

	By Lemma \ref{exten-reduc}.(1) the linear pencils $^\text{h}\! L_1$ and $^\text{h}\! L_2$ are $\sigma$-minimal.
	By Lemma \ref{completely-pos-and-isom} the unital linear map
		$\tau:\mathcal S_1 \to \mathcal S_2,$ $A_j\mapsto B_j$
	is well-defined, completely isometric $\ast$-isomorphism.
	By Proposition \ref{Central-proj-and-Silov}, the \v{S}ilov ideals for $\cS_1$ in $C^\ast(\cS_1)$ and $\mathcal S_2$ in
	$C^\ast(\cS_2)$ are trivial. 
	Now the remaining part of the proof is the same as for the proof of \cite[Theorem 3.12]{HKM1}.
	\\

	\noindent\textbf{Case 2:} $D_{^\text{h}\! L_1}\neq D_{^\text{h}\! L_2}$.\\

	Let $\tilde L_j, j=1,2,$ be the extended linear pencils of $L_j$, $j=1,2$, defined by
		$$\tilde L_1:= L_1\oplus 1, \quad \tilde L_2:= L_2\oplus 1.$$
	We have the equalities
		$$D_{\tilde L_1}=D_{L_1}=D_{L_2}=D_{\tilde L_2}.$$

	\noindent \textbf{Claim 1:} $D_{^\text{h}\! \tilde L_1}=D_{^\text{h}\! \tilde L_2}$.\\

	By Theorem \ref{n-pos-of-tau}, the inclusion $D_{L_1}\subseteq D_{L_2}=D_{\tilde L_2}$ 
	implies $D_{^\text{h}\! \tilde L_1}\subseteq D_{^\text{h}\! \tilde L_2}$ and by analogy  
	the inclusion $D_{L_2}\subseteq D_{L_1}=D_{\tilde L_1}$ implies  $D_{^\text{h}\! \tilde L_2}\subseteq D_{^\text{h}\! \tilde L_1}$.
	Hence $D_{^\text{h}\! \tilde L_1}=D_{^\text{h}\! \tilde L_2}$.\\
	
	\noindent \textbf{Claim 2.} $D_{^\text{h}\! \tilde L_\ell} \neq D_{^\text{h}\! L_\ell}$ for $\ell=1,2$.\\

	Since we are in Case 2, $D_{^\text{h}\! L_1} \neq D_{^\text{h}\! L_2}$. It follows that 
		$$D_{^\text{h}\! \tilde L_1} \neq D_{^\text{h}\! L_1}\quad \text{or}\quad D_{^\text{h}\! \tilde L_2} \neq D_{^\text{h}\! L_2}.$$
	By symmetry let us suppose $D_{^\text{h}\! \tilde L_2} \neq D_{^\text{h}\! L_2}$ and prove that
	$D_{^\text{h}\! \tilde L_1} \neq D_{^\text{h}\! L_1}$.
	Assume that $D_{^\text{h}\! \tilde L_1} =D_{^\text{h}\! L_1}$.
	Since  $D_{^\text{h}\! \tilde L_2} \subsetneq D_{^\text{h}\! L_2}$, there exists a tuple 
		$$(X_0,X_1,\ldots,X_g)\in D_{^\text{h}\! L_2}\setminus D_{^\text{h}\! \tilde L_2}$$ 
	where $X_0, X_1,\ldots, X_g\in \Sym_n.$
	Note that $X_0\not\succeq 0$. Let $v\in \RR^n$ be a vector such that $\left\langle X_0 v,v\right\rangle=-1$.
	Hence 
		$$(-1,\left\langle X_1v,v\right\rangle,\ldots,\left\langle X_gv,v\right\rangle) 
			\in D_{^\text{h}\! L_2}(1)\setminus D_{^\text{h}\! \tilde L_2}(1).$$
	Thus $\sum_j B_j \left\langle X_jv,v\right\rangle\succ 0$. 
	Since $D_{L_1}=D_{L_2}$ it follows that 
		$$\sum_j A_j \left\langle X_jv,v\right\rangle\succeq 0.$$
	(Otherwise $\lim_{t\to \infty} I_m+ t\sum_j A_j \left\langle X_jv,v\right\rangle\not\succeq 0$ but 
	 $\lim_{t\to \infty} I_m+ t\sum_j B_j \left\langle X_jv,v\right\rangle\succeq 0$, which is a contradiction.)
	But since 
	$D_{^\text{h}\! L_1}=D_{^\text{h}\! \tilde L_1}$, it follows that 
		$$\sum_j A_j \left\langle X_jv,v\right\rangle\not \succ 0.$$
	Let 
		$$H=\Ran(\sum_j A_j \left\langle X_jv,v\right\rangle).$$ 
	We will prove that 
		$$V^\ast L_1 V,$$
	where $V:H\to \RR^d$ is the inclusion of $H$ into $\RR^d$, is 
	a proper whole subpencil of $L_1$,
	which contradicts to the fact that $L_1$ is $\sigma$-minimal.
	We have to prove that $H$ is reducing for every $A_j$ and that $D_{V^\ast L_1 V}=D_{L_1}$.
	It suffices to prove that $A_j|_{H^{\perp}}=0$.
	Since 
		$$I\otimes 0+ \sum_j B_j \left\langle X_jv,v\right\rangle\pm \epsilon B_j \succeq 0$$ 
	for every $j$ and every $\epsilon>0$ small 
	enough, it follows from $D_{L_1}=D_{L_2}$ and an approximation argument that
		\begin{equation}\label{psd} 
			\sum_j A_j \left\langle X_jv,v\right\rangle\pm \epsilon A_j \succeq 0
		\end{equation}
	for every $j$ and every $\epsilon>0$ small 
	enough. Let us take $u\in H^\perp$. By (\ref{psd}) we first conclude that 
		$$\left\langle A_ju,u \right\rangle =0.$$
	and second that
		$$(\sum_j A_j \left\langle X_jv,v\right\rangle\pm \epsilon A_j) u=0.$$ 
	Finally, since $u\in H^\perp$ it follows that
		$$A_ju=0.$$
	Therefore $A_j|_{\cH^\perp}=0$ for every $j=1,\ldots,g$.
	This proves the claim.\\
	
	\noindent \textbf{Claim 3.} Linear pencils $^\text{h}\! \tilde L_\ell$ are minimal for $\ell=1,2$.\\
	
	This follows by Claim 2 and Lemma \ref{exten-reduc}.(2).\\

	Now we define new pencils $\hat L_\ell(x_0,x_1,\ldots,x_g)$ for $\ell=1,2$ by
		$$\hat L_\ell(x_0,x_1,\ldots,x_g):=I_n\oplus 1 + (I_n\oplus 1) x_0 +\sum_j \tilde A_j x_j.$$

	\noindent \textbf{Claim 4.} The following observations are true:
		\begin{enumerate}
			\item\label{obs-1} For $X:=(X_0,X_1,\ldots,X_g)\in (\RR^{n\times n})^{g+1}$ we have 
				$$\hat L_\ell(X_0,X_1,\ldots,X_g)= ^\text{h}\!\tilde{L}_\ell(X_0+I_d,X_1,\ldots,X_g).$$
			\item\label{obs-2} $D_{\hat L_1}=D_{\hat L_2}.$
			\item\label{obs-3}  For $X:=(X_{-1},X_0,X_1,\ldots,X_g)\in (\RR^{n\times n})^{g+2}$ we have 
				$$^\text{h}\! \hat L_\ell(X_{-1},X_0,X_1,\ldots,X_g)=
				^\text{h}\! \tilde L_\ell(X_{-1}+X_0+I_d,X_1,\ldots,X_g).$$
			\item\label{obs-4} $D_{^\text{h}\! \hat L_1}=D_{^\text{h}\! \hat L_2}.$
			\item\label{obs-5} Let 
				\begin{eqnarray*}
					\widetilde{\mathcal S_1}:=\Span\{I_n\oplus 1, \tilde A_j\colon j=1,\ldots,g\}\\
					\widetilde{\mathcal S_2}:=\Span\{I_n\oplus 1, \tilde B_j\colon j=1,\ldots,g\}\\
				\end{eqnarray*}
				and let $\tilde \tau:\widetilde{\mathcal S_1}\to \widetilde{\mathcal S_2}$ be the unital linear map
				defined by 
					$\tilde A_j\mapsto \tilde B_j$.
				Then $\tau$ is well-defined and completely isometric.
			\item\label{obs-6} $\hat L_\ell$ is $\sigma$-minimal for $\ell=1,2$.			\\
		\end{enumerate}

	The statements (\ref{obs-1}) and  (\ref{obs-3}) are clear. By Claim 1,  (\ref{obs-1}) implies  (\ref{obs-2}) and 
	 (\ref{obs-3}) implies  (\ref{obs-4}). 
	 (\ref{obs-5}) follows by Lemma \ref{completely-pos-and-isom}. Let us prove  (\ref{obs-6}).
	Let us say $\hat L_{\ell}$ is not $\sigma$-minimal. Hence,
				$D_{\hat L_\ell}=D_{V^{\ast} \hat L_\ell V}$ where $V:H\to\RR^{d+1}$ is
				the inclusion of a proper reducing subspace $H$ of every $\tilde A_j$
				into $\RR^{d+1}$. But then since
					$$V^{\ast} \hat L_\ell(x_0-1,x_1,\ldots,x_g) V=
						V^\ast {^\text{h}\! \tilde L}_\ell(x_0,x_1,\ldots,x_g) V,$$
				we get 
					$$D_{^\text{h}\! \tilde L_\ell}=D_{V^{\ast} {^\text{h}\! \tilde L_\ell} V},$$ 
				which is a contradiction with $^\text{h}\! \tilde L_\ell$ being $\sigma$-minimal.\\

	Now, by the same proof as in the Case 1, we conclude that
	$d+1=e+1$ and there is a unitary operator 
	$\tilde U:\RR^{d+1}\to \RR^{d+1}$ such that
			$$
				\hat L_2 = \tilde U^\ast \hat L_1\tilde U,\quad\text{and whence}\quad
				\hat L_1 = \tilde U \hat L_2\tilde U^\ast.
			$$
	By  (\ref{obs-1}) of Claim 4, this implies
			\begin{equation}\label{unitary-equivalence}
			\tilde L_2 = \tilde U^\ast \tilde L_1\tilde U\quad\text{and whence}\quad
				\tilde L_1 = \tilde U \tilde L_2\tilde U^\ast.
			\end{equation}
	Write $U:\RR^{d+1}\to \RR^{d+1}$ in the form
		$$\left[\begin{array}{cc}  U_{11} & u_{12}\\ u_{21} & u_{22}\end{array}\right],$$
	where $U_{11}:\RR^d\to \RR^d$, $u_{12}, u_{21}$ are vectors in $\RR^d$ and $u_{22}$ is a real number.
	Plugging this in the equations above we get
		\begin{eqnarray*}
			L_2 &=&  U_{11}^\ast  L_1 U_{11}+u_{21}^{t}u_{21}\\
			L_1 &=& U_{11} L_2 U_{11}^\ast+u_{12}u_{12}^t.
		\end{eqnarray*} 
	Hence 
		$$L_1=U_{11}U_{11}^\ast L_1 U_{11}U_{11}^\ast + 
			U_{11}  u_{21}^t u_{21} U_{11}^\ast+u_{12}u_{12}^t.$$
	Thus for 
		$$W=U_{11}U_{11}^\ast,\quad \text{and}\quad S=U_{11}  u_{21}^t u_{21}
			 U_{11}^\ast+u_{12}u_{12}^t\succeq 0$$
	we get
		\begin{equation}\label{eq-for-L-1}
   			L_1=W^\ast L_1 W+S.
		\end{equation}

	\noindent \textbf{Claim 5.} The matrix $W$ is an isometry, i.e., $W^\ast W=I_d$. In particular,
		$$I_d=U_{11}U_{11}^\ast=U_{11}^\ast U_{11},\quad 0=u_{21}=u_{12}.$$

	From (\ref{eq-for-L-1}) it follows that 
		\begin{eqnarray*}
			A_j&=& W^\ast A_j W\quad \text{for every } j=1,\ldots,g,\\
			I_d&=& W^\ast W + S.
		\end{eqnarray*}
	 If $W^\ast W\neq I_d$, then if
		$$1\geq \lambda_1\geq \ldots\geq \lambda_d\geq 0$$
	are the eigenvalues of $W^\ast W$, we have $1>\lambda_d$.
	Since $1\geq \|W^\ast W\|=\|W\|^2$, it follows that $1\geq \|W\|$. 
	Let $\mu_{1},\ldots,\mu_d$ be the eigenvalues of $W$. From $1\geq \|W\|$ it follows that
	$|\mu_j|\leq 1$ for every $j=1,\ldots,d$.
	From
		$$1>\lambda_1\cdots \lambda_d=\det(W^\ast W)=\det(W)^2=|\mu_1|\cdots |\mu_d|$$
	it follows that there is an eigenvalue $\mu$ of $W$ such that $|\mu|<1$. 
	But then there is an eigenvector $v\in \CC^d$ such that $Wv=\mu v$.
	Since	$A_j=W^\ast A_j W$, we have 
		$$A_j=(W^k)^\ast A_j W^k\quad \text{for every  } k\in \NN.$$
	Therefore $A_jv=(W^k)^\ast A_j W^kv$ for every $k\in \NN$.
	But 
		$$\lim_{k\to \infty} \|(W^k)^\ast A_j W^kv\|\leq 
			\lim_{k\to \infty} \|W^\ast\|^k \|A_j\| \|W^kv\|\leq
			\|A_j\| \lim_{k\to \infty} |\mu|^k \|v\|=0.$$
	Hence 
		$$A_j v=0 \quad \text{for every }j.$$
	Since $L_1$ is $\sigma$-minimal, we conclude that $\cap_{j=1}^g \ker A_{j}=\{0\}$. But this is in contradiction with 
	$v\in \cap_{j=1}^g \ker A_{j}$. Hence $W^\ast W= I_d$
	and the claim is proved. \\

	By Claim 5, $L_2=U_{11}^\ast L_1 U_{11}$ for a unitary $U_{11}$. Hence $D_{^\text{h}\! L_1}=D_{^\text{h}\! L_2}$, which 
	contradicts the assumption of Case 2.
\end{proof}

\subsection{Non-existence of a $\sigma$-minimal whole subpencil} \label{non-existence-of-reduct}
Example \ref{without-reduct-subpencil} below shows that in contrast with a matrix pencil, an operator pencil does not necessarily have
a whole subpencil which is $\sigma$-minimal.

\begin{example}\label{without-reduct-subpencil}
	Let 
		$$L(x)=I_{\ell^2}+ \diag(\frac{n}{n+1})_{n\in\NN}  \; x \in B(\ell^2)$$ 
be a diagonal linear operator pencil on $\ell^2$. We claim that 
		$$D_{L}(m)=\{X\in \Sym_m \colon X\succeq -I_{\ell^2}\}$$ 
and there does not exist a $\sigma$-minimal whole subpencil of $L$.
	
	Let $\{e_{i}\}_{i\in \NN}$ be the standard basis of $\ell^2$, $e_i$ has 1 in the $i$-th coordinate and 0 elsewhere.
	
	First we will prove the claim about $D_L(m)$. Let $X\in \Sym_m$. 
	Since
		\begin{eqnarray*}
				\left\langle L(X)(e_i\otimes v),e_j\otimes u\right\rangle 
			&=&  \left\langle e_i\otimes v+ \frac{i}{i+1} e_i \otimes Xv ,e_j\otimes u\right\rangle\\
			&=&
		  		\left\langle e_i ,e_j\right\rangle_{\ell^2} \left\langle v ,u\right\rangle_{\RR^m}+
  			  	\left\langle \frac{i}{i+1}e_i ,e_j\right\rangle_{\ell^2} \left\langle Xv, u\right\rangle_{\RR^m}\\
			&=&	 0
		\end{eqnarray*}
	for every $i\neq j$ and every $u,v \in \RR^m$, $X\in D_L(m)$ if and only if 
		\begin{eqnarray*}
			0
			&\leq&\left\langle L(X)(e_i\otimes u),e_i\otimes u\right\rangle =
			 	\left\langle u ,u\right\rangle_{\RR^m}+\frac{i}{i+1} \left\langle Xu ,u\right\rangle_{\RR^m}\\
			&=& 	\left\langle (I_m+\frac{i}{i+1}X)u,u\right\rangle_{\RR^m}.
		\end{eqnarray*}
	for every $i\in \NN$ and every $u\in \RR^m$. This is equivalent to $I_m+\frac{i}{i+1}X\succeq 0$ for every $i\in \NN$
	and further on to $X\succeq -\frac{i+1}{i}I_m$ for every $i\in \NN$. Hence, $X\in D_{L}(m)$ if and only if $X\succeq -I_m$.
	
	To  prove that $L$ does not have a $\sigma$-minimal whole subpencil let us first argue that the only reducing closed subspaces of
	the operator $\Omega_1:=\diag(\frac{n}{n+1})$ are the subspaces $H$ with the orthonormal basis of the form
		\begin{equation} \label{form-of-H}
			\{e_{i_j}\colon i_j\in \NN, j\in \NN\}.
		\end{equation}
	For $i\in \NN$, let $e_i=h_1+h_2$ where $h_1\in H$ and $h_2\in H^\perp$.
	From
		$$\frac{i}{i+1}e_i=\underbrace{\frac{i}{i+1}h_1}_{H}+
						\underbrace{\frac{i}{i+1}h_2}_{H^\perp}=
			\Omega_1(e_i)=\Omega_1(h_1+h_2)=
					\underbrace{\Omega_1(h_1)}_{H}+
					\underbrace{\Omega_1(h_2)}_{H^\perp},$$
	it follows that 
		$$\Omega_1(h_i)=\frac{i}{i+1}h_i\quad \text{for }i=1,2.$$
	Since the eigenspace corresponding to the eigenvalue $\frac{i}{i+1}$ of
	$\Omega_1$ is $\Span\{e_i\}$ (one dimensional), it follows that $e_i=h_1$ or $e_i=h_2$. Hence,
	$e_i\in H$ or $e_i\in H^\perp$. Thus the orthonormal basis of $H$ is of the form 
	(\ref{form-of-H}).

	Therefore:
	\begin{enumerate}
		\item Subpencils of $L$ are of the form $V^\ast L V$, where 
			$V$ is the inclusion of some subspace $H$ with the orthonormal basis
			$\{e_{i_j}\colon i_j\in \NN, j\in \NN\}$ into $\ell^2$.
		\item A subpencil $V^\ast L V$ is whole if and only if the sequence $(i_j)_j$ diverges.
		\item The whole subpencil $V^\ast L V$ is not $\sigma$-minimal, since it has a whole subpencil 
				$\tilde V^\ast V^\ast L V \tilde V,$
			where $\tilde V$ is the inclusion 
				$\tilde V:H_1 \hookrightarrow H$
			of the subspace $H_1$ with the orthonormal basis
			$\{e_{i_j}\colon i_j\in \NN, j\in \NN\}\setminus \{e_{i_1}\}$ into $H$.
	\end{enumerate}
\end{example}

\subsection{Counterexample to the operator linear Gleichstellensatz} \label{counterexample for the LG}
By Example \ref{not-unitarily-equivalent} below $\sigma$-minimal \emph{operator} pencils with the same free Hilbert spectrahedron are not
necessarily unitarily equivalent. Hence Theorem  \ref{UnbounLG} does not extend from matrix to operator pencils.

Example  \ref{not-unitarily-equivalent} is constructed by the use of an outer $\ast$-automorphism \cite{ARC} of the Cuntz $C^{\ast}$-algebra $C^\ast(S_1,S_2)$  \cite{CUN} generated by the isometries $S_1,S_2\in B(\cH)$ on a Hilbert space $\cH$ such that
	$S_1S_1^\ast+S_2S_2^\ast=\id_\cH.$
Recall that a $\ast$-automorphism $\theta$ is \textit{outer} if there does not exist a unitary 
$U\in C^\ast(S_1,S_2)$
such that $\theta(A)=U^\ast AU$ for all $A\in C^\ast(S_1,S_2)$.

\begin{example} \label{not-unitarily-equivalent}
	Let $\NN=\{1,2,\ldots\}$ and let $e_i$ be a standard unit vector on a complex Hilbert space $\ell^2:=\ell^2(\NN)$, i.e., the only nonzero coordinate is the $i$-th one which is 1. 
	Let $S_1$ and $S_2$ be bounded operators on $\ell^2$ defined by $e_i\mapsto e_{2i-1}$ for $i\in \NN$ and
	$e_i\mapsto e_{2i}$ for $i\in\NN$ respectively. The $C^{\ast}$-algebra $C^{\ast}(S_1,S_2)$ was studied by
	Cuntz \cite{CUN}. He showed that there is a unique $\ast$-isomorphism 
		$$\theta:C^{\ast}(S_1,S_2)\to C^{\ast}(S_1,S_2)$$
	such that 
		$$\theta(S_1)=S_2,\quad \theta(S_2)=S_1.$$
	We claim that linear operator pencils 
		\begin{eqnarray*}
			L_1(x)
			&=& I_{\ell^2} + A_1 x_1 + A_2 x_2 + A_3 x_3 + A_4 x_4,\\
			L_2(x)
			&=& I_{\ell^2} + A_2 x_1 + A_1 x_2 + A_4 x_3 + A_3 x_4,
		\end{eqnarray*}
	where
		\begin{eqnarray*}
			A_1 &:=& S_1+S_1^\ast\in B(\ell^2), \quad
			A_2 :=S_2+S_2^\ast \in B(\ell^2),\\
			A_3 &:=& i(S_1-S_1^\ast)\in B(\ell^2), \quad
			A_4 := i(S_2-S_2^\ast)\in B(\ell^2),\\
		\end{eqnarray*}
	are $\sigma$-minimal pencils with $D_{L_1}=D_{L_2}$, but there is no unitary operator $U:\ell^2\to \ell^2$ such that
		\begin{equation}\label{unitary-eq-L-2-to-L-1}
			 L_2=U^\ast L_1 U\quad \text{or} \quad L_2=U^\ast \overline{L_1} U.
		\end{equation}
	
	\noindent \textbf{Claim 1.} $D_{L_1}=D_{L_2}$. \\

	Clearly, the $C^\ast$-algebra 
		$$\cA:= C^{\ast}(A_1,A_2,A_3,A_4)$$ 
	generated by $A_j$, $j=1,2,3,4$, equals to
	 $C^{\ast}(S_1,S_2)$.
	Hence $\theta$ maps $L_1$ to $L_2$ and $L_2$ to $L_1$.
	From $\theta(L_1)=L_2$, it follows that $D_{L_1}\subseteq D_{L_2}$ and similarly $\theta(L_2)=L_1$ implies 
	$D_{L_2}\subseteq D_{L_1}$. Thus $D_{L_1}=D_{L_2}$.\\

	\noindent \textbf{Claim 2.} $L_1$ and $L_2$ are $\sigma$-minimal. \\
	
	It is sufficient to prove that there is no common
	reducing subspace for the operators $A_1, A_2, A_3, A_4$. 
	Let us say that $H$ is their common reducing subspace. Then it is also reducing for the
	operators
		$$\frac{A_1-iA_3}{2}=S_1,\quad \text{and}\quad \frac{A_2-iA_4}{2}=S_2.$$
	By the proof of \cite[Theorem 1]{ARC}, $S_1$ and $S_2$ have no common proper reducing subspaces.
	Hence $L_1$ and $L_2$ are $\sigma$-minimal.\\

	\noindent\textbf{Claim 3.} There does not exist a unitary operator $U:\ell^2\to \ell^2$ satisfying 
	(\ref{unitary-eq-L-2-to-L-1}).\\

	If there would exist a unitary operator $U:\ell^2\to \ell^2$ satisfying (\ref{unitary-eq-L-2-to-L-1}), then in particular
		\begin{equation}\label{A-3-to-A-4}
			A_4=U^\ast A_3 U \quad \text{or}\quad A_4=U^\ast \overline{A_3} U.
		\end{equation}
	We will prove that $\ker{A_3}=\ker{\overline{A}_3}\neq \{0\}$ and $\ker{A_4}=\{0\}$ which contradicts to
	(\ref{A-3-to-A-4}).
	Note that $S_1^\ast$ and $S_2^\ast$ are bounded operators on $\ell^2$ defined by 
		$$e_{2i-1}\mapsto e_i, \; e_{2i}\mapsto 0 \quad \text{for } i\in \NN\quad\text{and}\quad
		e_{2i-1}\mapsto 0,\; e_{2i}\mapsto e_i  \quad \text{for } i\in \NN,$$
	respectively.  
	Hence, 	
		$$A_3 e_1=i(S_1-S_1^\ast)e_1=0=-i(S_1-S_1^\ast)e_1=\overline{A_3}e_1\quad \Rightarrow\quad e_1\in \ker{A_3}=\ker(\overline{A_3}).$$
	It remains to prove that $\ker{A_4}=\{0\}.$ Let us say
		$$f:=\sum_{j=1}^{\infty}\alpha_{j} e_{j}\in \ker{A_4}\quad \text{where }\alpha_{j}\in \CC\; \text{for all }j\in \NN.$$
	We define $e_{\frac{2k-1}{2}}=0$ for every $k\in \NN$. We have
		\begin{equation} \label{A4-eq}
			A_{4}f=i\sum_{j=1}^{\infty}\alpha_{j} e_{2j}-i\sum_{j=1}^{\infty}\alpha_{j} e_{\frac{j}{2}}=0.
		\end{equation}
	If $\alpha_{j_0}\neq 0$ for some $j_0\in \NN$, then it follows from (\ref{A4-eq}) inductively that 
		$$\alpha_{j_0}=\alpha_{4j_0}=\alpha_{16j_0}=\ldots=\alpha_{4^nj_0}=\ldots .$$
	But then $\|f\|=\infty$ and hence $f\notin \ell^2.$ Therefore $f=0$ and $\ker{A_4}=\{0\}.$
\end{example}

\section{Operator Positivstellensatz for multivariate matrix polynomials } \label{positivstellensatz-multivariable}

In this section we characterize multivariate matrix polynomials that are positive semidefinite on a free Hilbert spectrahedron (see Theorem \ref{Oper-to-mat-intro} above and its restatement Theorem \ref{CPsatz-matrix} below) and a free Hilbert spectrahedrop (see Theorem \ref{CPsatz-matrix-drops}). Under the assumption of finite-dimensional Hilbert spaces Theorem \ref{Oper-to-mat-intro} was proved in \cite{HKM2} by modifying the classical Putinar-type separation argument. By essentially using Corollary \ref{OperLP} we are able to apply the separation argument also for infinite dimensional Hilbert spaces $\cH$. Precisely, we use Corollary \ref{OperLP} to prove that a certain set of nc matrix polynomials, i.e., 
the truncated quadratic module generated by an operator pencil, is closed. If $\cH$ is finite-dimensional, its closedeness follows by Caratheodory's theorem and a compactness argument in $\RR^\nu$, while for infinite-dimensional $\cH$ the compactness argument only works after translating the question to finite dimensions by Corollary \ref{OperLP}.

\subsection{Restatement of Theorem \ref{Oper-to-mat-intro}}

To prove Theorem \ref{Oper-to-mat-intro} we have to refine its statement. For this sake we introduce some definitions.

For 
$P\in \RR^{\ell\times \nu}\!\left\langle x\right\rangle$, an element of the form 
$P^\ast P\in \RR^{\ell\times \nu}\!\left\langle x\right\rangle$ is caled a \textbf{hermitian square}. Let $\Sigma^\nu$
denote the cone of sums of squares of $\nu\times\nu$ matrix-valued polynomials, and, given a nonnegative integer $N$, let
$\Sigma_N^\nu\subseteq \Sigma^\nu$ denote sums of squares of polynomials of degree at most $N$. Thus elements of $\Sigma^\nu_N$ have degree at most $2N$, i.e., $\Sigma^\nu_N\subseteq \RR^{\nu\times\nu}\!\left\langle x\right\rangle_{2N}$. Since the highest order terms in a sum of squares cannot cancel, we have
	$\RR^{\nu\times\nu}\!\left\langle x\right\rangle_{2N}\cap \Sigma^\nu=\Sigma^\nu_N.$

Fix $\nu\in\NN$. Let $\tilde \Pi^{\cH}_{\nu}$ be the set of all triples $(\cK,\pi,V)$, where $\cK$ is a separable real Hilbert space, $V:\RR^{\nu}\to \cK$ an isometry and
$\pi:B(\cH)\to B(\cK)$ a $\ast$-homomorphism.

Let $L\in \Sym_\cH\! \left\langle x\right\rangle$ be a monic linear operator pencil. Given $\nu_1,\nu_2,\alpha,\beta\in \NN$, we define the \textbf{$(\nu_1,\nu_2; \alpha,\beta)$ truncated quadratic module generated by $L$}, 
	$$
		M_{\alpha,\beta}^{\nu_1,\nu_2}(L)
					:=  \Sigma^{\nu_1}_\alpha+ \left\{\sum_{k,(\cK_k,\pi_k,V_k)\in \tilde \Pi^{\cH\oplus \RR}_{\nu_2}}^{ \text{finite}}  B_k^\ast V_k^\ast \pi_k(\left[\begin{array}{cc} L & 0 \\ 0 & 1\end{array}\right]) V_k B_k\colon
		B_k \in \RR^{\nu_1\times \nu_2}\!\left\langle x\right\rangle_\beta
		\right\}.
	$$
In the case $D_L(1)$ is a bounded set, we can replace $\ast$-homomorphisms of the extended pencil $L\oplus 1$ by
$\ast$-homomorphisms of $L$ in the definition of the truncated quadratic module.

\begin{proposition}	 \label{bounded-trunc-module}
	If $D_L(1)$ is a bounded set, then:
	$$
		M_{\alpha,\beta}^{\nu_1,\nu_2}(L)
					=\Sigma^{\nu_1}_\alpha+ \left\{\sum_{k,(\cK_k,\pi_k,V_k)\in \tilde \Pi^\cH_{\nu_2}}^{ \text{finite}}  B_k^\ast V_k^\ast \pi_k(L) V_k B_k\colon
		B_k \in \RR^{\nu_1\times \nu_2}\!\left\langle x\right\rangle_\beta
		\right\}.
	$$
\end{proposition}

\begin{proof}
	It is sufficient to prove that for every isometry $V\in B(\RR^{\nu_2},\cK)$
	and every $\ast$-homomorphism $\pi: B(\cH\oplus \RR)\to B(\cK)$ there
	exist an isometry $\tilde V\in B(\RR^{\nu_2},\tilde \cK)$ and a $\ast$-homomorphism 
	$\tilde \pi: B(\cH)\to B(\tilde \cK)$ such that
		$$V^\ast\pi(\left[\begin{array}{cc} L & 0 \\ 0 & 1\end{array}\right])V=
			\tilde V^\ast \tilde \pi(L)\tilde V.$$
	Since $V^\ast\pi(\left[\begin{array}{cc} L & 0 \\ 0 & 1\end{array}\right])V$ is a monic linear 
	pencil positive semidefinite on $D_L$ and $D_L(1)$ is bounded, this is true by 
	Corollary \ref{OperLP}.(3).
\end{proof}

The following is the restatement of Theorem \ref{Oper-to-mat-intro}.

\begin{theorem}[Operator convex Positivstellensatz] \label{CPsatz-matrix}
		Let 
		$L\in\Sym_\cH\!\left\langle x\right\rangle$
	be a monic linear operator pencil and $F\in \RR^{\nu\times \nu}\left\langle x\right\rangle$ a matrix polynomial of degree at most $2d+1$. 
	If $F|_{D_{L}}\succeq 0$, then 
		$$F\in M_{d+1,d}^{\nu,\ell}(L),$$
	where $\ell:=\nu\cdot \sigma_{\#}(d)$ and 
	$\sigma_{\#}(d):=\dim(\RR\!\left\langle x\right\rangle_d)$.
\end{theorem}

The proof of Theorem \ref{CPsatz-matrix} is given in Subsection \ref{proof-CP-satz}. In the next two subsections 
we prove the connection between positive linear functionals and operators and 
show that the truncated quadratic module is closed. Both results are important ingredients for the separation argument in the proof of Theorem \ref{CPsatz-matrix}.

\subsection{Positive linear functionals and the GNS construction}

Proposition \ref{GNS-cons} below (see \cite[Proposition 2.5]{HKM2}), embodies the well known connection, through the Gelfand-Naimark-Segal (GNS) construction, between operators and positive linear functionals. The only difference between the statements of Proposition \ref{GNS-cons} and \cite[Proposition 2.5]{HKM2} is that the pencil $L$ is operator-valued here but was matrix-valued in \cite[Proposition 2.5]{HKM2}. Therefore, the proof of Proposition \ref{GNS-cons} needs an additional argument. Namely, in the notation of Proposition \ref{GNS-cons} the fact that a tuple of operators $X$ belongs to $D_L$ if $\lambda$
is nonnegative on $M_{k+1,k}^{\nu, \nu\sigma_{\#}(k)}(L)$ follows immediately by construction if $L$ is matrix-valued but needs a proof if $L$ is operator-valued.

\begin{proposition}\label{GNS-cons}
	If $\lambda:\RR^{\nu\times \nu}\!\left\langle x\right\rangle_{2k+2}\to \RR$ is a linear functional which is nonnegative on $\Sigma_{k+1}^\nu$  and positive on $\Sigma_{k}^\nu\setminus\{0\}$, then there exists a tuple $X=(X_1,\ldots,X_g)$ 
of symmetric operators on a Hilbert space $\mathcal X$ of dimension at most
$\nu \sigma_{\#}(k)=\nu \dim \RR\!\left\langle x\right\rangle_{k}$ and a vector $\gamma\in \mathcal X^{\oplus \nu}$, such
that
	$$\lambda(f)=\left\langle f(X)\gamma,\gamma\right\rangle$$
for all $f\in \RR^{\nu\times \nu}\!\left\langle x\right\rangle_{2k+1}$, where $\left\langle \cdot,\cdot\right\rangle$ is the inner
product on $\mathcal X$. Further, if $L\in \Sym_{\cH}\!\left\langle x\right\rangle$ is a monic linear operator pencil and $\lambda$ is nonnegative on $M_{k+1,k}^{\nu, \nu\sigma_{\#}(k)}(L)$, then $X\in D_L$.

		Conversely, if $X=(X_1,\ldots,X_g)$ is a tuple of symmetric operators on a Hilbert space $\cX$ of dimension $N$,
	the vector $\gamma\in \cX^{\oplus \nu}$, and $k$ a positive integer, then the linear functional
$\lambda:\RR^{\nu\times \nu}\!\left\langle x\right\rangle_{2k+2}\to \RR$, defined by
	$$\lambda(f)=\left\langle f(X)\gamma,\gamma\right\rangle$$
is nonnegative on $\Sigma_{k+1}^\nu$. Further, if $L\in \Sym_{\cH}\!\left\langle x\right\rangle$ is a monic linear operator pencil and $X\in D_L$, then $\lambda$ is nonnegative also on $M_{k+1,k}^{\nu, \ell}(L)$ for every $\ell\in \NN$.
\end{proposition}

In the proof we will need the following special case (see \cite[Theorem 3.1]{HKM4} and \cite[\S 6]{HM}) of a theorem due to Effros and Winkler \cite{EFF-WIN}.

\begin{theorem}\label{EFF-WIN-HB}
	If $\Gamma=(\Gamma(n))_{n\in\NN}\subseteq \Sym^g$ is a closed matrix convex set containing 0 and $X\in \Sym^g_m$ is not in $\Gamma(m)$, then there is a monic linear pencil $\cL$ if size $m$ such that $\cL(Y)\succeq 0$ for all $Y\in \Gamma$, but $\cL(X)\not\succeq 0 $.
\end{theorem}

\begin{proof}[Proof of Proposition \ref{GNS-cons}]
	The nontrivial direction is $(\Rightarrow)$. 
	The proof is the same as that of \cite[Proposition 2.5]{HKM2}, just that we need to show that
	in the case that $\lambda$ is nonnegative on $M_{k+1,k}^{\nu, \nu\sigma_{\#}(k)}(L)$, we have  $X\in D_{L}$.
	If $L$ is matrix-valued, then this follows by an elementary calculation. If $L$ is operator-valued, we will additionally need
	Corollary \ref{OperLP} and Theorem \ref{EFF-WIN-HB}.

	Let us assume the notation from the proof of \cite[Proposition 2.5]{HKM2}. Namely,
	the positive semidefinite symmetric bilinear form defined on the vector space 
	$K=\RR^{1\times \nu}\left\langle x\right\rangle_{k+1}$ by
		\begin{equation}\label{sym-bil-form} 
			\left\langle f,h\right\rangle=\lambda(h^\ast f). 
		\end{equation}	
	induces a positive definite bilinear form on the quotient $\tilde \cX:=K/\mathcal N$, where
	$\mathcal N:=\{f\in K\colon \left\langle f,f\right\rangle=0\},$ making it a Hilbert space.
 	By positive definiteness of the form (\ref{sym-bil-form}) on the subspace $\cX=\RR^{1\times \nu}\left\langle x\right\rangle_{k}$, $\cX$ can be considered as a subspace of $\tilde \cX$ with dimension $\nu \sigma_{\#}(k)$.	
	The symmetric operators $X_j: \cX\to \cX$ are defined by
		$$X_jf=Px_jf, \quad f\in \cX, \; 1\leq j\leq g,$$
	where $P$ is the orthogonal projection from $\tilde \cX$ onto $\cX$.%

	Suppose $\lambda$ is nonnegative on $M_{k+1,k}^{\nu, \nu\sigma_{\#}(k)}(L)$.
	Denote $L=I_{\cH}+\sum_{j=1}^{m} A_j x_j$ and $A:=(A_1,\ldots, A_g)\in B(\cH)^g$.
	Let $\ell:=\nu \sigma_{\#}(k)$. 
	Take an arbitrary isometry $V\in B(\RR^{\ell},\cK)$. Given 
		$$p=\left[\begin{array}{c} p_1 \\ \vdots \\ p_\ell \end{array}\right]\in \cX^{\oplus \ell},$$
	note that
	\begin{eqnarray*}
		&&\left\langle (V^\ast \pi(\left[\begin{array}{cc} L & 0 \\ 0 & 1\end{array}\right]
				)V)(X) p, p\right\rangle 
		=
			\left\langle (I_{\ell}-\Lambda_{V^{\ast}\pi(A\oplus 0)V}(X)p, p\right\rangle\\
		&=&
			\left\langle p-\sum \Lambda_{V^{\ast}\pi(A_j\oplus 0)V} Px_jp, p\right\rangle
		= \left\langle p-\sum \Lambda_{V^{\ast}\pi(A_j\oplus 0)V}x_jp, p\right\rangle\\
		&=& \left\langle (I_\ell-\sum \Lambda_{V^{\ast}\pi(A_j\oplus 0)V}x_j)p, p\right\rangle
		= \lambda(p^\ast(I_\ell- \Lambda_{V^{\ast}\pi(A\oplus 0)V}(x))p )\\
		&=&\lambda(p^{\ast}V^\ast \pi(\left[\begin{array}{cc} L & 0 \\ 0 & 1\end{array}\right]
				)Vp)\geq 0
	\end{eqnarray*}

\noindent	\textbf{Claim.} $X\in D_L(\ell)$.\\

	If $X\notin D_L(\ell)$, then by Theorem \ref{EFF-WIN-HB}, there is a monic linear 
	pencil $\mathcal{L}$ of size $\ell$ such that $\mathcal L(Y)\succeq 0$ for all $Y\in D_{L}$ and
	$\mathcal L(X)\not\succeq 0$. 
	But by Corollary \ref{OperLP}.(1), 
		$$\mathcal L=V^\ast \pi(\left[\begin{array}{cc} L & 0 \\ 0 & 1\end{array}\right])V$$
	for some $\ast$-homomorphism $\pi:B(\cH\oplus\RR)\to B(\cK)$ where $\cK$ is a separable real Hilbert space,
	and some isometry $V\in B(\RR^{\ell},\cK)$.
	By the calculation above,
		$$\mathcal L(X)\succeq 0,$$ 
	which is a contradiction. Hence, $X\in D_L(\ell)$.
\end{proof}

\subsection{The truncated quadratic module is closed}

	Fix $\alpha,\beta, \nu_1,\nu_2\in\NN$ and let $\kappa=\max\{2\alpha,2\beta+1\}.$ Let $L\in \Sym_\cH\!\left\langle x\right\rangle$ be a monic linear operator pencil. The truncated quadratic module $M_{\alpha,\beta}^{\nu_1,\nu_2}(L)$ generated by a monic linear operator pencil $L$ is a convex cone in 
$\RR^{\nu\times\nu}\!\left\langle x\right\rangle_k$.

	Given $\epsilon>0$, let 
		$$\mathcal B_{\epsilon}(n):=\{X\in \mathbb{S}^g_n\colon \left\|X\right\|\leq \epsilon\},
			\quad\text{and}\quad \mathcal B_\epsilon=\bigcup_{n\in \NN}\mathcal B_{\epsilon}(n).$$
	There is an $\epsilon>0$ such that for all $n\in \NN$, if $X\in \Sym_n^g$ and $\left\|X\right\|\leq \epsilon$, then
	$L(X)\succeq \frac{1}{2}I_\cH$. In particular, $\mathcal B_\epsilon\subseteq D_L$. 	
	Using this $\epsilon$ we norm $\RR^{\nu_1\times \nu_2}\left\langle x\right\rangle_k$ by
		\begin{equation} \label{norm} 
			\left\|p\right\|:=\max\{\left\|p(X)\right\|\colon X\in \mathcal B_\epsilon\}.
		\end{equation}
	(On the right-hand side of (\ref{norm}) the maximum is attained. This follows from the fact that the bounded nc semialgebraic set $\mathcal B_\epsilon$ is convex. See \cite[Section 2.3]{HM1} for details.)

By the Proposition \ref{tr-mat-closed} below,  $M_{\alpha,\beta}^{\nu_1,\nu_2}(L)$ is closed. This is the same result as \cite[Proposition 3.1]{HKM2} but the proof is much longer and uses Corollary \ref{OperLP} essentially. The latter is used to prove that a limit of a
certain convergent sequence of linear matrix pencils of the form $V_k \pi_k(\left[\begin{array}{cc} L & 0 \\ 0 & 1\end{array}\right])V_k$, where $\pi_k:B(\cH\oplus \RR)\to B(\cK_k)$ is a $\ast$-homomorphism, $\cK_k$ is a separable real Hilbert space and $V_k\in B(\RR^{\nu_2}, \cK_k)$ is an isometry, is again of the form 
	$V \pi(\left[\begin{array}{cc} L & 0 \\ 0 & 1\end{array}\right])V$ with $\pi$ a $\ast$-homomorphism and 
 $V\in B(\RR^{\nu_2}, \cK)$ an isometry.

\begin{proposition} \label{tr-mat-closed}
	The truncated quadratic module $M_{\alpha,\beta}^{\nu_1,\nu_2}(L)\subseteq \RR^{\nu_1\times \nu_1}\!\left\langle x\right\rangle_\kappa$ is closed.
\end{proposition}

\begin{proof}
Suppose $(P_n)$ is a sequence from $M_{\alpha,\beta}^{\nu_1,\nu_2}(L)$ 
	which converges to some
	$P\in \RR^{\nu_1\times\nu_1}\!\left\langle x\right\rangle$ of degree at most $\kappa$.
	By Caratheodory's theorem on convex hulls \cite[Theorem I.2.3]{BARV}, there is an $M$ (at most the dimension of
	$\RR^{\nu_1\times\nu_1}\!\left\langle x\right\rangle_k$ plus one) 
	such that for each $n$ there exist matrix-valued
	polynomials $R_{n,i}\in \RR^{\nu_1\times \nu_1}\!\left\langle x\right\rangle_{\alpha}$, 
	$T_{n,i}\in \RR^{\nu_1\times \nu_2}\!\left\langle x\right\rangle_{\alpha}$,
	$\ast$-homomorphisms $\pi_{n,i}:B(\cH\oplus\RR)\to B(\cK_{n,i})$ where $\cK_{n,i}$ is
	a separable real Hilbert space, and isometries	$V_{n,i}\in B(\RR^{\nu_2}, \cK_{n,i})$ 
	 such that
		$$P_n=\sum_{i=1}^{M}R_{n,i}^{\ast}R_{n,i}+ 
				\sum_{i=1}^{M}T_{n,i}^{\ast} V^{\ast}_{n,i} \pi_{n,i}(
\left[\begin{array}{cc} L & 0 \\ 0 & 1\end{array}\right])V_{n,i}T_{n,i}.$$

	\noindent \textbf{Claim 1.} The sequences $(R_{n,i})_n$ and $(T_{n,i})_n$ are bounded in norm for each $i$.\\

	The sequence $(P_n)_n$ is bounded in norm, i.e., $\left\|P_n\right\|\leq N^2$ for every $n\in \NN$ and some
	$N\in \NN$. Fix $i\in \NN$.
	For every $X\in\mathcal B_{\epsilon}$ and every $n$ we have
		$$R_{n,i}^{\ast}R_{n,i}(X)\succeq 0 \quad \text{and}\quad
			T_{n,i}^{\ast} V^{\ast}_{n,i} \pi_{n,i}(\left[\begin{array}{cc} L & 0 \\ 0 & 1\end{array}\right])V_{n,i}T_{n,i}(X)\succeq 0.
			 $$
	Thus for every $X\in\mathcal B_{\epsilon}$ and every $n$ it follows that
		$$P_n(X)\succeq R_{n,i}^{\ast}R_{n,i}(X)\quad\text{and}\quad
			P_n(X)\succeq T_{n,i}^{\ast} V^{\ast}_{n,i} \pi_{n,i}(\left[\begin{array}{cc} L & 0 \\ 0 & 1\end{array}\right])V_{n,i}T_{n,i}(X).$$ 
	Hence, for every $n$ we have
		\begin{eqnarray*}
			N^2&\geq& \left\|P_n\right\|\geq \left\|R_{n,i}^{\ast}R_{n,i}\right\|= \left\|R_{n,i}\right\|^2,\\  
			N^2&\geq& \left\|P_n\right\|\geq 
				\left\|T_{n,i}^{\ast} V^{\ast}_{n,i} \pi_{n,i}(\left[\begin{array}{cc} L & 0 \\ 0 & 1\end{array}\right])V_{n,i}T_{n,i}\right\|.
		\end{eqnarray*}
	So $(R_{n,i})_n$ is bounded. Let us denote $L(x)=I_\cH + \sum_j A_jx_j$.
	For the boudedness of $(T_{n,i})_n$ observe that
		\begin{eqnarray*}
		\left\|T_{n,i}^{\ast}T_{n,i}\right\|
			&=& \frac{1}{2} \left\|T_{n,i}^{\ast} V^{\ast}_{n,i} \pi_{n,i}\left(2\left[\begin{array}{cc} I_{\cH} & 0 \\ 0 & 1\end{array}\right]\right)V_{n,i}T_{n,i}\right\|\\
			&=& \frac{1}{2} \left\|T_{n,i}^{\ast} V^{\ast}_{n,i} \pi_{n,i}\left(2\left[\begin{array}{cc} I_{\cH} & 0 \\ 0 & 1\end{array}\right]+
				\left[\begin{array}{cc} \sum_j A_j x_j & 0 \\ 0 & 0\end{array}\right]\right.\right.\\
			&-&\left.\left.
				\left[\begin{array}{cc} \sum_j A_j x_j & 0 \\ 0 & 0\end{array}\right]\right)V_{n,i}T_{n,i}\right\|\\
			&\leq& \frac{1}{2}  \left\|T_{n,i}^{\ast} V^{\ast}_{n,i} \pi_{n,i}\left(\left[\begin{array}{cc} I_\cH+\sum_j A_j x_j & 0 \\ 0 & 1\end{array}\right]\right)V_{n,i}T_{n,i}\right\|\\
			&+& \frac{1}{2}  \left\|T_{n,i}^{\ast} V^{\ast}_{n,i} \pi_{n,i}\left(\left[\begin{array}{cc} I_\cH+\sum_j A_j (-x_j) & 0 \\ 0 & 1\end{array}\right]\right)V_{n,i}T_{n,i}\right\|\\
			&=& \left\|T_{n,i}^{\ast} V^{\ast}_{n,i} \pi_{n,i}\left(\left[\begin{array}{cc} L & 0 \\ 0 & 1\end{array}\right]\right)V_{n,i}T_{n,i}\right\|\\
			&\leq& N^2.
		\end{eqnarray*}
	By Claim 1 and since we are in finite dimensional vector spaces, $(R_{n,i})_n$, $(T_{n,i})_n$ have convergent subsequences with limits 
	$R_i\in \RR^{\nu_1\times\nu_1}\!\left\langle x\right\rangle$, 
	$T_i\in \RR^{\nu_2\times\nu_1}\!\left\langle x\right\rangle$. \\

\noindent \textbf{Claim 2.} The sequences $(V^{\ast}_{n,i} \pi_{n,i}\left(\left[\begin{array}{cc} L & 0 \\ 0 & 1\end{array}\right]\right)V_{n,i})_n\subseteq 
		\RR^{\nu_2\times \nu_2}\left\langle x\right\rangle$ are bounded in norm for each $i$.\\

The following estimate holds:
		\begin{eqnarray*}
		&&\left\|V^{\ast}_{n,i} \pi_{n,i}\left(\left[\begin{array}{cc} L & 0 \\ 0 & 1\end{array}\right]\right)V_{n,i}\right\| 
		= 	\max_{X\in \cB_\epsilon}
			\left\|\left(V^{\ast}_{n,i} \pi_{n,i}\left(\left[\begin{array}{cc} L & 0 \\ 0 & 1\end{array}\right]\right)V_{n,i}\right)(X)\right\|\\
		&=& \max_{X\in \cB_\epsilon}
			\left\|I_{\nu_2}\otimes I + 
		\sum_{j=1}^{g} V^{\ast}_{n,i} \pi_{n,i}(\left[\begin{array}{cc} A_j & 0 \\ 0 & 0\end{array}\right]) V_{n,i}\otimes X_j \right\|\\
		&\leq& 1 + \sum_{j=1}^{g}\max_{X\in \cB_\epsilon}
			\left\|V^{\ast}_{n,i} \pi_{n,i}(\left[\begin{array}{cc} A_j & 0 \\ 0 & 0\end{array}\right]) V_{n,i}\otimes X_j \right\|\\
		&\leq& 1 + \sum_{j=1}^{g}\max_{X\in \cB_\epsilon}
			\left\|V^{\ast}_{n,i} \pi_{n,i}(\left[\begin{array}{cc} A_j & 0 \\ 0 & 0\end{array}\right]) V_{n,i}\right\|\left\|X_j \right\|\\
		&\leq& 1 + \epsilon 
			\sum_{j=1}^{g}\left\|V^{\ast}_{n,i} \pi_{n,i}(\left[\begin{array}{cc} A_j & 0 \\ 0 & 0\end{array}\right]) V_{n,i}\right\|	\\
		&\leq& 1 + \epsilon 
			\sum_{j=1}^{g}\left\|V^{\ast}_{n,i}\right\|	 
			\left\|\pi_{n,i}\right\|\left\|\left[\begin{array}{cc} A_j & 0 \\ 0 & 0\end{array}\right]\right\|\left\|V_{n,i}\right\|\\
		&\leq& 1 + \epsilon \sum_{j=1}^{g}\left\|A_j\right\|.
		\end{eqnarray*}

	By Claim 2 and since we are in a finite dimensional vector space, the sequences from Claim 2 converge for each $i$ 
	to a monic linear matrix pencil  
		$$\hat L_i=I_{\nu_2}+\sum_{j} \hat A_{j,i} x_j\in \RR^{\nu_2\times \nu_2}\!\left\langle x\right\rangle.$$

	\noindent \textbf{Claim 3.} $D_L\subseteq D_{\hat L_i}$ for each $i$.\\

	Fix $i\in \NN$. Suppose there is 
	$X\in D_L(m)\setminus D_{\hat L_i}(m)$. Then there is
	a vector $v\in \RR^{\nu_2}\otimes \RR^m$ of norm 1 such that 
		\begin{equation} \label{negative} \left\langle \hat L_i(X)v,v\right\rangle<0 \end{equation} 
	Since 
	$(V^{\ast}_{n,i} \pi_{n,i}(\left[\begin{array}{cc} L & 0 \\ 0 & 1\end{array}\right])V_{n,i})_n$
	converges to $\hat L_i$ in the norm, there is $k_0\in \NN$ such that 
		\begin{equation} \label{estimate-1} 
			\left\|\hat L_i-V^{\ast}_{k_0,i} \pi_{k_0,i}(
			\left[\begin{array}{cc} L & 0 \\ 0 & 1\end{array}\right])V_{k_0,i}\right\|\leq 
			\frac{\epsilon}{2\left\|X\right\|}\frac{\left|\left\langle \hat L_i(X)v,v\right\rangle\right|}{2}.\end{equation} 
	Since $\hat L_i$ and $V^{\ast}_{k_0,i} \pi_{k_0,i}(\left[\begin{array}{cc} L & 0 \\ 0 & 1\end{array}\right])V_{k_0,i}$ are monic, the following estimate holds
	\begin{eqnarray*} 
	&&\left\|\left(\hat L_i-V^{\ast}_{k_0,i} \pi_{k_0,i}(\left[\begin{array}{cc} L & 0 \\ 0 & 1\end{array}\right])V_{k_0,i}\right)(X)\right\|\\
	&=&\frac{2\left\|X\right\|}{\epsilon}\left\|(\hat L_i-V^{\ast}_{k_0,i} \pi_{k_0,i}(\left[\begin{array}{cc} L & 0 \\ 0 & 1\end{array}\right])V_{k_0,i})
	(\underbrace{\frac{\epsilon}{2\left\|X\right\|} X}_{\in \cB_\epsilon})\right\|\\
	&\leq&	\frac{2\left\|X\right\|}{\epsilon}
	\left\|\hat L_i-V^{\ast}_{k_0,i} \pi_{k_0,i}(\left[\begin{array}{cc} L & 0 \\ 0 & 1\end{array}\right])V_{k_0,i}\right\|
	\underbrace{\leq}_{(\ref{estimate-1})}
	\frac{\left|\left\langle \hat L_i(X)v,v\right\rangle\right|}{2}.
	\end{eqnarray*} 
	But then, since $v$ is of norm one, we have
		$$
			\left|\left\langle\left(\hat L_i-V^{\ast}_{k_0,i} \pi_{k_0,i}(\left[\begin{array}{cc} L & 0 \\ 0 & 1\end{array}\right]\right)V_{k_0,i})(X)v,v\right\rangle\right|
			\leq\frac{ | \left\langle \hat L_i(X)v,v\right\rangle|}{2},
		$$
	and hence 
		\begin{equation} \label{negative-2}
			\left\langle \left(V^{\ast}_{k_0,i} \pi_{k_0,i}(
				\left[\begin{array}{cc} L & 0 \\ 0 & 1\end{array}\right])V_{k_0,i}\right)(X)v,v\right\rangle\leq
			\frac{\left\langle \hat L_i(X)v,v\right\rangle}{2}\underbrace{<}_{(\ref{negative})}0.
		\end{equation}
	But since 
		\begin{eqnarray*}
			&&(V^{\ast}_{k_0,i} \pi_{k_0,i}(\left[\begin{array}{cc} L & 0 \\ 0 & 1\end{array}\right])V_{k_0,i})(X)
			=\\
			&=&	(V^{\ast}_{k_0,i}\otimes I_m)  \left(\pi_{k_0,i}(\left[\begin{array}{cc} L & 0 \\ 0 & 1\end{array}\right])(X)\right)
				(V_{k_0,i}\otimes I)\\
			&=&	(V^{\ast}_{k_0,i}\otimes I_m)  \left((\pi_{k_0,i}\otimes I_m)(\left[\begin{array}{cc} L & 0 \\ 0 & 1\end{array}\right](X))\right)
				(V_{k_0,i}\otimes I_m)\\
			&=&	(V^{\ast}_{k_0,i}\otimes I_m)  \left((\pi_{k_0,i}\otimes I_m)(\left[\begin{array}{cc} L(X) & 0 \\ 0 & 1\end{array}\right])\right)
				(V_{k_0,i}\otimes I_m)\\
			&\underbrace{\succeq}_{X\in D_L(m)}& 0,
		\end{eqnarray*}
	where $\pi_{k_0,i}\otimes I_m$ is a $\ast$-homomorphism $A\otimes B\mapsto \pi_{k_0,i}(A)\otimes B$, 
	this is a contradiction with (\ref{negative-2}). Hence, $D_L\subseteq D_{\hat L_i}$.

	To conclude the proof we use Corollary \ref{OperLP}.(1). There is a triple $(\cK_{i},\pi_i,V_i)$ of a separable real Hilbert space $\cK_i$,
	a $\ast$-homomorphism $\pi_i:\cH\oplus\RR\to\cK_i$ and an
	isometry $V_i$ such that
		$$\hat L_i=V_i^{\ast}\pi_i(\left[\begin{array}{cc} L & 0 \\ 0 & 1\end{array}\right])V_i.$$
	Therefore $(P_n)_n$ converges to
		$$\sum_{i=1}^{M}R_{i}^{\ast}R_{i}+ 
				\sum_{i=1}^{M}T_{i}^{\ast} V^{\ast}_{i} \pi_{i}(\left[\begin{array}{cc} L & 0 \\ 0 & 1\end{array}\right])V_{i}T_{i}\in 
				M_{\alpha,\beta}^{\nu_1,\nu_2}(L).$$
\end{proof}

\subsection{Proof of Theorem  \ref{CPsatz-matrix}} \label{proof-CP-satz} In this subsection we prove Theorem \ref{CPsatz-matrix}.
	The argument is a classical one going back to at least Putinar \cite{PUT} and its noncommutative version in \cite{HM1}, but with a consequential difference. Namely, the difference is in the separating functional $\lambda$, which produces perfection, i.e., the Positivstellensatz holds not only for positive definite polynomials but for semidefinite ones as well and we also 
get degree bounds (like \cite{HKM2}).

\begin{proof}[Proof of Theorem \ref{CPsatz-matrix}]
	Suppose $F\notin M_{d+1,d}^{\nu,\ell}(L)$. 
	By Proposition \ref{tr-mat-closed} and the Hahn-Banach separation theorem  
	there exists a linear functional 
	$\lambda:\RR^{\nu\times\nu}\!\left\langle x\right\rangle_{2d+1}\to \RR$ that is nonnegative on 
	$M_{d+1,d}^{\nu,\ell}(L)$ and negative on $F$.
	By \cite[Lemma 3.2]{HKM2}, we can assume $\lambda$ is positive. 
	(Note that the functional $\hat \lambda$ from the proof of \cite[Lemma 3.2]{HKM2} 
	is non-negative also on $M_{d+1,d}^{\nu,\ell}(L)$.) By Proposition \ref{GNS-cons}
	with $k=d$, there is a tuple of symmetric matrices $X\in D_L$ acting on the finite-dimensional Hilbert space $\cX$
 	and a vector $\gamma$ such that
		$$\lambda(P)=\left\langle P(X)\gamma,\gamma\right\rangle$$
	for all $P\in \RR^{\nu\times\nu}\!\left\langle x\right\rangle.$ In particular,
		$$\left\langle F(X)\gamma,\gamma\right\rangle=\lambda(F)<0,$$
	so that $F(X)$ is not positive semidefinite, contradicting $D_L\subseteq D_F$ and the proof is complete.
\end{proof}

\subsection{Convex Positivstellensatz for free Hilbert spectrahedrops}

This subsection focuses on polynomials positive on a free Hilbert spectrahedrop. The main result, Theorem \ref{CPsatz-matrix-drops}, extends Theorem \ref{CPsatz-matrix} from free Hilbert spectrahedra to free Hilbert spectrahedrops.

Let $L$ be a monic linear operator pencil of the form
	$$L(x,y)=I_\cH+\sum_{j=1}^g \Omega_j x_j + \sum_{k=1}^{h}\Gamma_k y_k\in 
		\Sym_\cH\!\left\langle x\right\rangle$$
and let $\fK=\proj_x D_L$.
	Fix positive integers $\nu_1,\nu_2, d\in  \NN$.
	We define the \textbf{$(\nu_1,\nu_2;d)$ truncated quadratic module} in 
	$\RR^{\nu_1\times\nu_1}\!\left\langle x	\right\rangle_{2d+1}$ associated to $L$ and 
	$\fK=\proj_x D_{L}$ by
	\begin{eqnarray*}
	M_{x}^{\nu_1,\nu_2}(L)_d
				&:=& 
		\{\sigma + \sum_{k,(\cK_k,\pi_k,V_k)\in \tilde \Pi_{\nu_2}^{\cH\oplus \RR}}^{\text{finite}}
		R_k^\ast V_k^\ast\pi_k (\left[\begin{array}{cc} L & 0\\ 0 & 1\end{array}\right])V _k R_k \colon 
		\sigma\in \Sigma_d^{\nu_1}\!\left\langle x\right\rangle,\\ 
		&&R_k \in \RR^{\nu_2\times \nu_1}\!\left\langle x\right\rangle_d,				
				\sum_k R_k^\ast V_k^\ast\pi_k(\left[\begin{array}{cc} \Gamma_\ell & 0\\ 0 & 0\end{array}\right])V _k R_k=0
					\text{ for all }\ell\}.
	\end{eqnarray*}

In the case $D_L(1)$ is a bounded set, we can replace $\ast$-homomorphisms of the extended pencil $L\oplus 1$ by
$\ast$-homomorphisms of $L$ in the definition of the truncated quadratic module.

\begin{proposition}	 \label{bounded-trunc-module-2}
	If $D_L(1)$ is a bounded set, then:
	\begin{eqnarray*}
	M_{x}^{\nu_1,\nu_2}(L)_d
				&:=& 
		\{\sigma + \sum_{k,(\cK_k,\pi_k,V_k)\in \tilde \Pi_{\nu_2}^{\cH}}^{\text{finite}}
		R_k^\ast V_k^\ast\pi_k (L)V _k R_k \colon 
		\sigma\in \Sigma_d^{\nu_1}\!\left\langle x\right\rangle,\\ 
		&&R_k \in \RR^{\nu_2\times \nu_1}\!\left\langle x\right\rangle_d,				
				\sum_k R_k^\ast V_k^\ast\pi_k (\Gamma_\ell )V _k R_k=0
					\text{ for all }\ell\}.
	\end{eqnarray*}
\end{proposition}

\begin{proof}
	The proof is the same as the proof of Proposition \ref{tr-mat-closed} using
Theorem \ref{OperLP-trunc} instead of Corollary \ref{OperLP}.
\end{proof}

The main result of this subsection is the following Positivstellensatz:

\begin{theorem}	\label{CPsatz-matrix-drops}
		A polynomial $F\in \RR^{\nu_1\times \nu_1}\!\left\langle x\right\rangle_{2d+1}$ is positive semidefinite on $\fK$
	if and only if $F\in M_{x}^{\nu_1,\nu_2}(L)_d$.
\end{theorem}

\begin{remark}
Several remarks are in order.
\begin{enumerate}
\item In case there are no $y$-variables in $L$, Theorem \ref{CPsatz-matrix-drops} reduces to Theorem \ref{CPsatz-matrix}.
\item If $d= 0$, i.e., $F$ is linear, then Theorem \ref{CPsatz-matrix-drops} reduces to Theorem \ref{OperLP-trunc}.
\item If $L$ is matrix-valued, then Theorem \ref{CPsatz-matrix-drops} reduces to \cite[Theorem 5.1]{HKM3}. 
\item If $L$ is matrix-valued and variables commute, a Positivstellensatz for commutative polynomials strictly positive on spectrahedrops was established by Gouveia and Netzer in \cite{GN11}. A major distinction is that the degrees of
the $R_k$ and $\sigma$ in the commutative theorem behave very badly.
\end{enumerate}
\end{remark}

\subsection{Proof of Theorem \ref{CPsatz-matrix-drops}}
The proof uses the same idea as the proof of Theorem \ref{CPsatz-matrix}, i.e., construction of a positive separating functional and then the connection with operators via the GNS construction. What has to be proved additionally is that the truncated quadratic module $M_{x}^{\nu_1,\nu_2}(L)_d$ is closed (see Proposition \ref{trunc-proj-quad-closed}) and that the tuple of operators $X$ from the GNS construction belongs to the closure of the free Hilbert spectrahedrop (see Proposition \ref{GNS-cons-2}).

\begin{proposition} \label{trunc-proj-quad-closed}
	The truncated module $M_{x}^{\nu_1,\nu_2}(L)_d$  is closed.
\end{proposition}

\begin{proof}
		The proof is the same as the proof of Proposition 
	\ref{bounded-trunc-module} using	Theorem \ref{OperLP-trunc} instead of
	Corollary \ref{OperLP}.
\end{proof}

\begin{proposition}\label{GNS-cons-2}
	If $\lambda:\RR^{\nu\times \nu}\!\left\langle x\right\rangle_{2k+2}\to \RR$ is a linear functional which is nonnegative on $\Sigma_{k+1}^\nu$  and positive on $\Sigma_{k}^\nu\setminus\{0\}$, then there exists a tuple $X=(X_1,\ldots,X_g)$ 
of symmetric operators on a Hilbert space $\mathcal X$ of dimension at most
$\nu \sigma_{\#}(k)=\nu \dim \RR\!\left\langle x\right\rangle_{k}$ and a vector $\gamma\in \mathcal X^{\oplus \nu}$, such
that
	$$\lambda(f)=\left\langle f(X)\gamma,\gamma\right\rangle$$
for all $f\in \RR^{\nu\times \nu}\!\left\langle x\right\rangle_{2k+1}$, where $\left\langle \cdot,\cdot\right\rangle$ is the inner
product on $\mathcal X$. Further, if $L(x,y)\in \Sym_{\cH}\!\left\langle x\right\rangle$ is a monic linear operator pencil and $\lambda$ is nonnegative on $M_{x}^{\nu, \nu\sigma_{\#}(k)}(L)_{k}$, then $X$
is in the closure $\overline{\proj_x D_L}$ of the free spectrahedrop $\proj_x D_L$ coming from $L$.

		Conversely, if $X=(X_1,\ldots,X_g)$ is a tuple of symmetric operators on a Hilbert space $\cX$ of dimension $N$,
	the vector $\gamma\in \cX^{\oplus \nu}$, and $k$ a positive integer, then the linear functional
$\lambda:\RR^{\nu\times \nu}\!\left\langle x\right\rangle_{2k+2}\to \RR$, defined by
	$$\lambda(f)=\left\langle f(X)\gamma,\gamma\right\rangle$$
is nonnegative on $\Sigma_{k+1}^\nu$. Further, if $L(x,y)\in \Sym_{\cH}\!\left\langle x\right\rangle$ is a monic linear operator pencil and $X\in \overline{\proj_x D_L}$, then $\lambda$ is nonnegative also on $M_{x}^{\nu, \ell\sigma_{\#}(k)}(L)_{k}$ for every $\ell\in\NN$.
\end{proposition}

\begin{proof}
		The nontrivial direction is $(\Rightarrow)$. The proof goes the same as the proof of 
	\cite[Proposition 5.4]{HKM3}, just that we need to add the explanation, why in the case that 
	$\lambda$ is nonnegative on $M_{x}^{\nu, \nu\sigma_{\#}(k)}(L)_{k}$, we have 
	$X\in\overline{\proj_x D_L}$. If $L$ is matrix-valued, then this is \cite[Proposition 5.4]{HKM3}, while	for operator-valued case there are minor changes.  To establish \cite[Equality 5.5]{HKM3}
	use Theorem \ref{OperLP-trunc} with $\Gamma$ and $\Omega$ replaced by $\pi(\Gamma)$ and $\pi(\Omega)$ for some $\ast$-homomorphism $\pi:B(\cH)\to B(\cK)$, where $\cK$ is a separable real Hilbert space.  In the next statement replace 
		$$\sum_{j=1}^{\eta}W_j^\ast \Gamma W_j=0,\quad \sum_{j=1}^{\eta}W_j^\ast \Omega 	
				W_j=\Lambda$$
by 
		$$W^{\ast} \pi(\Gamma\oplus 0)W=0,\quad 
			W^\ast \pi(\Omega\oplus 0) W=\Lambda, W\in B(\RR^{\sigma},\cK)\;
			\text{an isometry}.$$
Now the \cite[Equality 5.7]{HKM3} becomes
	\begin{eqnarray*}
	0&>& u^\ast \cL_\Lambda (X) u= (\sum_i e_i \otimes v_i)^\ast \cL_\Lambda (X)  (\sum_i e_i \otimes v_i)\\
	& &		\sum_{i,j}(e_i\otimes v_i)^{\ast}(W \otimes I)^\ast \pi(L(X,Y)\oplus 1)  (W\otimes I)
			(e_i\otimes v_i)\\
	&=&	\sum_{i,j}(e_i\otimes p_i(X)\gamma)^{\ast} (W \otimes I)^\ast \pi(L(X,Y)\oplus 1)  (W\otimes I)
			(e_i\otimes p_j(X)\gamma)
	\end{eqnarray*}
and \cite[Equality 5.8]{HKM3} becomes
	$$0> (\vec p(X) \gamma)^\ast  (W \otimes I)^\ast \pi(L(X,Y)\oplus 1)  (W\otimes I) (\vec p(X) \gamma)=
		\lambda(q),$$
where $\vec p(x)=\sum_j e_j\otimes p_j(x)\in \RR^{ \nu\sigma_{\#}(k)\times \nu}\!\left\langle x\right\rangle_k$ and 
	$$q=\vec p(X)^\ast  (W \otimes I)^\ast \pi(L(X,Y)\oplus 1)  (W\otimes I) \vec p(X)\in M_{x}^{\nu, \nu\sigma_{\#}(k)}(L)_{k}.$$
This is in contradiction with the nonnegativity of $\lambda$ on $M_{x}^{\nu, \nu\sigma_{\#}(k)}(L)_{k}$.
\end{proof}

\begin{proof}[Proof of Theorem \ref{CPsatz-matrix-drops}]
		The proof is the same as the proof of \cite[Theorem 5.1]{HKM3}, just that 
	we use Proposition \ref{trunc-proj-quad-closed} instead of \cite[Proposition 5.3]{HKM3}
	and Proposition \ref{GNS-cons-2} instead of \cite[Proposition 5.4]{HKM3}.
\end{proof}

\section{Operator Positivstellensatz for univariate operator polynomials} \label{positivstellensatz-univariate}

In this section we extend Theorem \ref{CPsatz-matrix}  in the univariate case from matrix-valued polynomials to operator-valued ones. Namely, in the univariate case, $F$ in Theorem \ref{CPsatz-matrix} can be operator-valued but the conclusion still holds. For the precise statement see Theorem \ref{univariate-intro} above. The main step is the reduction to the inclusion of free Hilbert spectrahedra
by the use of variants of the operator Fej\'er-Riesz theorem \cite{ROS}. In Subsection \ref{non-monic-case} we also study the case of a non-monic $L$. By Examples \ref{monicity-needed} and \ref{empty-operator-spectrahedron}, Theorem \ref{univariate-intro} does not extend to the non-monic case.

\subsection{Proof of Theorem \ref{univariate-intro}}

	Since $L$ is monic, the set $D_{L}(1)$ is an interval with non-empty interior. We separate three cases.\\
	
\noindent		\textbf{Case 1:} $D_{L}(1)=[a,b]$, $a<b$, $a,b\in \RR$.\\ 

By the linear change of variables we may assume that $D_{L}(1)=[-1,1]$. 
		By \cite[Proposition 3]{CZ}, 
			$$F(y)=\sum_j R_j^\ast R_j + \sum_k \tilde Q_k^\ast
			\left[\begin{array}{cc} (1+y)I_\cK & 0\\ 0 & (1-y) I_\cK\end{array}\right] 
			\tilde Q_k,$$ where
	$R_j\in B(\cK)\!\left\langle y\right\rangle$, $\tilde Q_k\in B(\cK,\cK^2)\!\left\langle y\right\rangle$ and
	$$\deg(R_j)\leq \frac{\deg(F)}{2},\,\quad 
		\deg(\tilde B_k)\leq \frac{\deg(F)}{2}.$$ 
(For the degree bounds see \cite[Theorem 2.5]{DS} and
	use the identity $x(1-x)=x^2(1-x)+(1-x)^2x$.)
	
	It remains to prove the statement of the theorem for
	the pairs $(L(y),(1+y)I_\cK)$, $(L(y),(1-y)I_\cK)$. Further on, it suffices to prove it for
	the pairs  $(L(y),1+y)$, $(L(y),1-y)$. 
	We use Corollary \ref{OperLP-scal} and conclude the proof.\\

\noindent	
\textbf{Case 2:} $D_{L}(1)=[a,\infty)$ or $(-\infty,a]$, $a\in \RR$. \\

By the linear change of variables we may assume that 
	$D_{L}(1)=[-1,\infty)$. By \cite[Proposition 3]{CZ}, 
		$$F=\sum_j  R_j^\ast R_j + \sum_k \tilde Q_k^\ast (1+y)I_\cK \tilde Q_k ,$$ 
	where	$R_j\in B(\cK)\!\left\langle y\right\rangle$, $\tilde Q_k\in B(\cK)\!\left\langle y\right\rangle$ and
	$$\deg(A_j)\leq \frac{\deg(F)}{2},\quad 
	\deg(\tilde B_k)\leq \frac{\deg(F)}{2}.$$ 
(The degree bounds are easy to see by comparing the leading coefficients.)
	
	It remains to prove the statement of the theorem for the pair $(L(y),(1+y)I_\cK)$. As in Case 1 it suffices to prove it
	for the pair $(L(y),1+y)$. By Corollary \ref{OperLP-scal} the statement follows.\\
	
%
%
%
%
	
\noindent\textbf{Case 3:} $D_{L}(1)=\RR$. \\

	By \cite[Proposition 3]{CZ},
		$F=\sum_j  R_j^\ast R_j$
	where	$R_j\in B(\cH_2)\!\left\langle y\right\rangle$ and
		$\deg(R_j)\leq \frac{\deg(F)}{2}.$

\subsection{Non-monic case}\label{non-monic-case}

If $L$ is not monic in Theorem \ref{univariate-intro}, then the conclusion is not true in general (see Example \ref{monicity-needed} in
Section \ref{LOI-domination-section} above). However, by \cite[Corollary 4.3.1]{KS}, it extends to the matrix-valued pencil
$L$ with $D_L=\emptyset$. (The case $F=-1$ is the content of  \cite[Corollary 4.3.1]{KS}, while for an arbitrary $F$
one uses the identity $\frac{(F+1)^\ast(F+1)-(F-1)^\ast(F-1)}{4}$.)  But the following counterexample shows that 
\cite[Corollary 4.3.1]{KS} does not extend to the operator-valued pencil $L$ with $D_L=\emptyset$.

\begin{example} \label{empty-operator-spectrahedron}
		Let $L(y)=A_0 + A_1 y \in B(\ell^2)$ be a linear operator pencil, where
		$$A_0=\oplus_{n\in \NN}(-\frac{1}{n}),\quad A_1=\oplus_{n\in \NN}(\frac{1}{n^2}).$$
	Then the spectrahedron $D_{L}(1)$ is $\emptyset$ and $\ell(y)=-1$ is non-negative on $D_L(1)$, but there do
	not exist a Hilbert space $\cK$, a unital $\ast$-homomorphism 
	$\pi:B(\ell^2)\to B(\cK)$, polynomials $r_j\in \RR\!\left\langle y\right\rangle$ and operator polynomials
	$b_k\in B(\RR,\cK)\!\left\langle y\right\rangle$
	such that 
		\begin{equation} \label{zapis-za--1}
			-1= \sum_j r_j^2 + \sum_k q_k^\ast \pi(L) q_k.
		\end{equation}
\end{example}

\begin{proof}
	Let us say that $\cK$, $\pi$, $r_j$, $q_k$ satisfying (\ref{zapis-za--1}) exist. 
	Observe that $A_1=A_0^\ast A_0$.
	Therefore 
		$$\sum_k q_k^\ast \pi(L(y)) q_k
			=\sum_k q_k^\ast \pi(A_0)^\ast \pi(A_0) q_k\cdot y 
			+\sum_k q_k^\ast \pi(A_0) q_k.$$
	If $\sum_k q_k^\ast \pi(A_0) q_k=0$, then 
		\begin{equation} \label{zapis-za--1-2}
			-1=\sum_j r_j^2 + \sum_k q_k^\ast \pi(A_0)^\ast \pi(A_0) q_k y.
		\end{equation}
	This is a contradiction since the right-hand side of (\ref{zapis-za--1-2})  is nonnegative for $y\geq 0$, while the left-hand side is always -1.
	Therefore $\sum_k q_k^\ast \pi(A_0) q_k\neq 0.$
	Let us write 
		\begin{eqnarray*}
		      	r_j(y) &=&\sum_{m=0}^{N_j} r_{j,m}y\in \RR\!\left\langle y \right\rangle,\quad
			q_k(y)=\sum_{m=0}^{M_k} q_{k,m}y\in B(\RR,\cK)\!\left\langle y \right\rangle,
		\end{eqnarray*}
	where $N_j\in \NN_0$ is such that $r_{N_j}\neq 0$ and 
	$M_k\in \NN_0$ is such that $\pi(A_0) q_{k,M_k}\neq 0$.
	We can indeed choose such $M_k$, since otherwise 
	$\pi(A_0) q_{k,M_k}=q_{k,M_k}^\ast \pi(A_0)=0$ and hence
		\begin{eqnarray*}
			& &q_k^\ast \pi(A_0) q_k
			= \sum_{\ell,j=0}^{M_k} q_{\ell}^\ast \pi(A_0) q_j y^{\ell+j}\\
			&=&\sum_{j=0}^{M_k}  \underbrace{q_{M_k}^\ast \pi(A_0)}_{0} q_j y^{M_k+j}
				+\sum_{\ell=0}^{M_k-1} q_{\ell}^\ast \underbrace{\pi(A_0) q_{M_k}}_{0} y^{M_k+j}
				+\sum_{\ell,j=0}^{M_k-1} q_{\ell}^\ast \pi(A_0) q_{j} y^{\ell+j}\\
			&=& \sum_{\ell,j=0}^{M_k-1} q_{\ell}^\ast \pi(A_0) q_j y^{\ell+j}
			= (\sum_{m=0}^{M_k-1} q_{k,m}y)^{\ast}  \pi(A_0) (\sum_{m=0}^{M_k-1} q_{k,m}y)
		\end{eqnarray*}
	and
		\begin{eqnarray*}
			& &\sum_k q_k^\ast \pi(A_1) q_k
			= \sum_{\ell,j=0}^{M_k} q_{\ell}^\ast \pi(A_0)\pi(A_0) q_j y^{\ell+j}\\
			&=& \sum_{j=0}^{M_k}  \underbrace{q_{M_k}^\ast \pi(A_0)}_{0} 
					\pi(A_0)q_j y^{M_k+j}
				+\sum_{\ell=0}^{M_k-1} q_{\ell}^\ast \pi(A_0)
					\underbrace{\pi(A_0) q_{M_k}}_{0} y^{\ell+M_k}+\\
				&+&\sum_{\ell,j=0}^{M_k-1} q_{\ell}^\ast \pi(A_1)q_{j} y^{\ell+j}\\
			&=& \sum_{\ell,j=0}^{M_k-1} q_{\ell}^\ast \pi(A_1) q_j y^{\ell+j}
			= (\sum_{m=0}^{M_k-1} q_{k,m}y)^{\ast}  \pi(A_1) (\sum_{m=0}^{M_k-1} q_{k,m}y)
		\end{eqnarray*}
	We are repeating this calculation until $\pi(A_0)q_{k,M_k-r_k}\neq 0$ and take $q_k:=\sum_{m=0}^{M_k-r_k} q_{k,m}y$.
	
	The highest monomial according to the ordering of $\RR\!\left\langle y\right\rangle$
		$$dy^m\succeq c y^{n} \Leftrightarrow m>n \text{ or } m=n, d\geq c$$
	 in:
	\begin{enumerate}
		\item $r_j^2$ is $\underbrace{r_{j,N_j}^2}_{\neq 0} y^{2N_j},$
		\item $q_k^\ast \pi(A_1) q_k y$ is 
			$\underbrace{q_{k,M_k}^\ast \pi(A_0)^\ast \pi(A_0)q_{k,M_k}}_{\neq 0} 
				y^{2M_k+1}$,
		\item $q_k^\ast \pi(A_0) q_k$ is `at most'
			$q_{k,M_k}^\ast  \pi(A_0)q_{k,M_k} y^{2M_k}$ (or smaller).
	\end{enumerate}
	Let $M:=\text{max}\{N_j,M_k\colon j,k\}$. Therefore, the highest monomial on the right-hand side of (\ref{zapis-za--1-2})
	is
		$$\left\{
			\begin{array}{cc}	
				\displaystyle\sum_{j\colon N_j=M} \underbrace{r_{j,N_j}^2}_{> 0}
					y^{2M},& \text{if } M\neq M_k 
				\text{ for every }k\\
				\displaystyle
				\sum_{k\colon M_k=M} \underbrace{q_{k,M_k}^\ast \pi(A_0)^\ast
					 \pi(A_0)q_{k,M_k}}_{> 0}
						y^{2M+1},& \text{if } M= M_k \text{ for some }k
			\end{array}.
		\right.$$
	Since the highest monomial on left-hand side of (\ref{zapis-za--1-2})  is $-1$, we conclude that $M=0$ and $q_k=0$ for every $k$.
	Thus $-1=\sum_j r_j^2$ which is a contradiction.
\end{proof}

%

\begin{remark}
	Theorem \ref{univariate-intro} extends to non-monic $L(y)=A_0+A_1y\in \Sym_\cH\!\left\langle y\right\rangle$
	in the following cases:
	\begin{enumerate}
		\item $D_L(1)\neq \emptyset$ and $\Span\{A_0,A_1\}$ contains an invertible positive definite element.
		\item $D_L(1)=\{a\}$  and $A_0, A_1$ are linearly dependent.
		\item $D_L(1)=\emptyset$ and $D_{PLP}(1)$ is compact for some finite-dimensional projection $P\in B(\cH)$.
	\end{enumerate}
\begin{proof}
	The proof of (1) is the same as the proof of Theorem \ref{univariate-intro} just that we use a non-monic version of Corollary \ref{OperLP-scal} 
	(see Remark \ref{rem-after-OperLP}.(3)).
	
	Now we prove (2). By a linear change of variables we may assume that $D_L(1)=\{0\}.$
If $A_0\neq 0$, then we have 
	$A_0+A_1 y = A_0(1+\lambda y)$ for some $\lambda\in\RR$.
	Hence $A_0\succeq 0$. Thus	$\lim_{y\to\infty}L(y)\succeq 0$ or
	$\lim_{y\to-\infty}L(y)\succeq 0$. This is a contradiction.
Hence $A_0=0$ and $L(y)=A_1 y$. Since $D_L(1)=\{0\}$, there are $v_1,v_2\in\cH$ such that $\left\langle A_1v_1,v_1\right\rangle>0$ and
$\left\langle A_1v_2,v_2\right\rangle<0$. 
	So
		$$y=\frac{\left\langle L v_1,v_1\right\rangle}{\left\langle A_1v_1,v_1\right\rangle},\quad
		 -y=\frac{\left\langle L v_2,v_2\right\rangle}{|\left\langle A_1v_2,v_2\right\rangle|}.$$
By the identity $-y^2=\frac{y(y-1)^2-y(y+1)^2}{4}$ we conclude that
	$-y^2$ is of the form 
		$$-y^2=\sum_{j} r_j^2 + \sum_k Q_k^\ast P_NLP_N Q_k,$$
	where $r_j\in \RR\!\left\langle x \right\rangle$ are scalar polynomials and $Q_k\in M(\RR,\RR^N)\!\left\langle x \right\rangle$ are matrix polynomials.
	 Thus also $-y^{2\ell},-y^{2\ell+1}$ are of the above form for every $\ell\in \NN$ and so every $F\in B(\cK)\!\left\langle y\right\rangle$ satisfying $F(0)\succeq 0$ is of the from
		$$\sum_{j} R_j^\ast R_j + \sum_k Q_k^\ast P_NLP_N Q_k,$$
	where $R_j\in B(\cK)\!\left\langle x \right\rangle$  and $Q_k\in B(\cK,\RR^N)\!\left\langle x \right\rangle$ are operator 		
	polynomials.

	Finally we prove (3). Let $(P_n)_n$ be an increasing sequence of projections from $\cH$ to a $n$ dimensional subspace of 
	$\cH$ such that $P=P_\ell$ for $\ell=\dim \Ran(P).$
	We have the following decreasing sequence of compact sets:
		$$D_{P_\ell LP_\ell}(1)\supseteq D_{P_{\ell+1}LP_{\ell+1}}(1)
			\supseteq \cdots \supseteq \cap_{k=\ell}^\infty D_{P_k LP_k}(1)= 
			D_{L}(1)=\emptyset.$$
	Note that the equality $\cap_{k=\ell}^\infty D_{P_k LP_k}(1)=D_{L}(1)$ follows by the convergence of the sequence
	$P_kL P_k$ to $L$ in the weak operator topology.
	Since $D_{P_\ell L P_\ell}(1)$ is compact and 
		$D_{P_\ell LP_\ell }(1)\subset \bigcup_{k=\ell}^\infty D_{P_kLP_k}(1)^{c}$ is a 
	an open covering, it follows that 
		$$D_{P_\ell LP_\ell }(1)\subset \bigcup_{k}^{N} D_{P_kLP_k}(1)^{c}= D_{P_NLP_N}(1)^{c}$$
	for some $N\in\NN$. Hence $D_{P_NLP_N}(1))=\emptyset.$
	By \cite[Corollary 4.3.1]{KS}, $-1$ is of the form
		$$-1=\sum_{j} r_j^2 + \sum_k Q_k^\ast P_NLP_N Q_k,$$
	where $r_j\in M_N(\RR)\!\left\langle x \right\rangle$ are scalar polynomials and $Q_k\in M(\RR,\RR^N)\!\left\langle x \right\rangle$ are matrix polynomials.
	By the equality $F=\frac{(F+1)^\ast(F+1)-(F-1)^\ast(F-1)}{4}$, arbitrary $F$ is of the form
		$$F=\sum_{j} R_j^\ast R_j + \sum_k Q_k^\ast P_NLP_N Q_k,$$
	where $R_j\in B(\cK)\!\left\langle x \right\rangle$  and $Q_k\in B(\cK,\RR^N)\!\left\langle x \right\rangle$ are operator 		
	polynomials. (For the degree bounds see \cite[Theorem 4.3.3]{KS}.)
\end{proof}
\end{remark}

\noindent\textbf{Acknowledgement.} The paper was written while I was visiting Igor Klep at the University of Auckland. I would like to thank him for the hospitality, many helpful discussions and improvements of the article.
I am also thankful to Scott McCullough and Jurij Vol\v ci\v c who read parts of the manuscript and provided numerous helpful suggestions.

\end{document}